\documentclass[12pt, pdftex]{amsart}
\usepackage[margin=1in]{geometry}
\usepackage[all]{xy}
\usepackage{comment}
\usepackage{amsmath}
\usepackage{amssymb}
\usepackage{amsthm}
\usepackage{here}
\usepackage[pdftex]{graphicx}
\usepackage{mathrsfs}
\usepackage{mathtools}
\usepackage{mathabx}
\usepackage{scalefnt}
\usepackage{url}
\theoremstyle{plain}
\newtheorem{thm}{Theorem}[section]
\newtheorem{lem}[thm]{Lemma}
\newtheorem{cor}[thm]{Corollary}
\newtheorem{prop}[thm]{Proposition}
\newtheorem{conj}[thm]{Conjecture}
\theoremstyle{definition}

\newtheorem{rem}[thm]{Remark}

\newtheorem{defn}[thm]{Definition}

\numberwithin{equation}{section}
\def\A{{\mathbb A}}
\def\F{{\mathbb F}}
\def\Q{{\mathbb Q}}
\def\R{{\mathbb R}}
\def\Z{{\mathbb Z}}
\def\C{{\mathbb C}}
\def\P{{\mathbb P}}
\def\id{\mathop{\mathrm{id}}\nolimits}

\def\Gal{\mathop{\mathrm{Gal}}\nolimits}

\def\Ker{\mathop{\mathrm{Ker}}\nolimits}

\def\Pic{\mathop{\mathrm{Pic}}\nolimits}

\def\GL{\mathop{\mathrm{GL}}\nolimits}
\def\SL{\mathop{\mathrm{SL}}\nolimits}
\def\SU{\mathop{\mathrm{SU}}\nolimits}

\def\det{\mathop{\mathrm{det}}\nolimits}
\def\dim{\mathop{\mathrm{dim}}\nolimits}
\def\div{\mathop{\mathrm{div}}\nolimits}
\def\Div{\mathop{\mathrm{Div}}\nolimits}

\def\L{\mathscr{L}}

\def\M{\mathscr{M}}
\def\F{\mathscr{F}}

\def\OO{\mathscr{O}}

\def\A{\mathcal{A}}
\def\D{\mathscr{D}}
\def\p{\mathfrak{p}}
\def\f{\mathfrak{f}}
\def\a{\alpha}
\def\l{\langle}
\def\r{\rangle}

\def\exp{\mathop{\mathrm{exp}}\nolimits}
\def\Sp{\mathop{\mathrm{Sp}}\nolimits}
\def\SO{\mathop{\mathrm{SO}}\nolimits}
\def\SU{\mathop{\mathrm{SU}}\nolimits}

\def\U{\mathrm{U}}
\def\O{\mathrm{O}}

\def\v{\mathrm{vol}_{HM}}

\allowdisplaybreaks[4]

\newcommand{\defeq}{\vcentcolon=}

\begin{document}

\title[Reflective obstructions of unitary modular varieties]
{Reflective obstructions of unitary modular varieties}

\author{Yota Maeda}
\address{Advanced Research Laboratory, Technology Infrastructure Center, Technology Platform, Sony Group Corporation, 1-7-1 Konan, Minato-ku, Tokyo, 108-0075, Japan/Department of Mathematics, Faculty of Science, Kyoto University, Kyoto 606-8502, Japan}
\email{y.maeda.math@gmail.com}

\date{\today}

\maketitle

\begin{abstract}
To prove that a modular variety is of general type, there are three types of obstructions: reflective, cusp and elliptic obstructions.
In this paper, we give a quantitative estimate of the reflective obstructions for the unitary case.
This shows in particular that the reflective obstructions are small enough in higher dimension, say greater than $138$.
Our result reduces the study of the Kodaira dimension of unitary modular varieties to the construction of a cusp form of small weight in a quantitative manner.
As a byproduct, we formulate and partially prove the finiteness of Hermitian lattices admitting reflective modular forms, which is a unitary analog of the conjecture by Gritsenko-Nikulin in the orthogonal case.
Our estimate of the reflective obstructions uses Prasad's volume formula.
\end{abstract}
\setcounter{tocdepth}{1}
\tableofcontents

\section{Introduction}
\label{introduction}
The study of the birational type of modular varieties is an important problem.
Tai \cite{Tai}, Freitag \cite{Freitag} and Mumford \cite{Mumford2} showed that the Siegel modular varieties $\A_g$ are of general type if $g\geq 7$.
Gritsenko-Hulek-Sankaran \cite{GHS} showed that the moduli spaces of polarized K3 surfaces, which are 19-dimensional orthogonal modular varieties, are of general type if the polarization degree is larger than 61.
Moreover, Ma \cite{Ma} proved that orthogonal modular varieties are of general type if their dimension is larger than 107.
A common theme in this series of works implies that if the data defining modular varieties is ``sufficiently large", then the associated modular varieties are of general type.

Motivated by these works, we study an analogous problem for unitary modular varieties.
There exist three types of obstructions to prove that they are of general type, as in the orthogonal case \cite[Theorem 1.1]{GHS}.
They are \textit{reflective obstructions}, arising from branch divisors, \textit{cusp obstructions}, arising from divisors at
infinity, and \textit{elliptic obstructions}, arising from singularities.
Note that the elliptic obstructions were resolved in \cite{Behrens}.
In this paper, we study the reflective obstructions and prove that they are sufficiently small in higher dimension.
The key to the proof is to apply Prasad's volume formula to estimate the dimension of the space of modular forms on ball quotients.

\subsection{Unitary modular varieties}
Before we state the main results, let us prepare some notions.
Let $F$ be an imaginary quadratic field with discriminant $-D$ and $\OO_F$ be its ring of integers.
Let $(L, \l\ ,\ \r)$ be a Hermitian lattice over $\OO_F$ of signature $(1,n)$ with $n>2$.
Accordingly, we have the unitary group $\U(L)$ over $\Z$.
Let $\D_L$ be the Hermitian symmetric domain associated with $\U(L\otimes_{\Z}\R)$:
\[\D_L\defeq\{v\in V\otimes_F\C\setminus\{0\}\mid\l v,v\r>0\}/\C^{\times},\]
which is the $n$-dimensional complex ball.
Then, for a finite index subgroup $\Gamma\subset\U(L)$, we define
\[\F_L(\Gamma)\defeq \D_L/\Gamma.\]
This is a quasi-projective variety over $\C$ and called a \textit{unitary modular variety} or a \textit{ball quotient}.
Let $\overline{\F_L(\Gamma)}$ be the canonical toroidal compactification of $\F_L(\Gamma)$.
The canonical bundle of $K_{\overline{\F_L(\Gamma)}}$ is described as  
\begin{align}
\label{intoro:K}
    K_{\overline{\F_L(\Gamma)}}\sim_{\Q}(n+1)\L-\sum_{d}\frac{d-1}{d}\overline{B_d}-\Delta
\end{align}
in $\Pic(\overline{\F_L(\Gamma)})\otimes_{\Z}\Q$, where $\L$ is the Hodge line bundle and $B_d$ is the union of the branch divisors of the map $\D_L\to\F_L(\Gamma)$ with branch index $d$ and $\Delta$ is the boundary.
Here, $\overline{B_d}$ is the closure of $B_d$ in the toroidal compactification.

One strategy to prove that $\F_L(\Gamma)$ is of general type is to rewrite (\ref{intoro:K}) as 
\begin{align*}
    K_{\overline{\F_L(\Gamma)}}\sim_{\Q}\M_{\Gamma}(a)+\left\{(n+1-a)\L-\Delta\right\},
\end{align*}
for some positive integer $a>0$, where 
\[\M_{\Gamma}(a)\defeq a\L-\sum_{d}\frac{d-1}{d}\overline{B_d},\]
and show that
\begin{align*}
    \mathrm{(A)}&\ (\mathrm{Reflective\ obstructions})\ \M_{\Gamma}(a)\ \mathrm{is\ big},\\
    \mathrm{(B)}&\ (\mathrm{Cusp\ obstructions})\ (n+1-a)\L-\Delta\ \mathrm{is\ effective}.
\end{align*}
If (A) and (B) hold, then $K_{\overline{\F_L(\Gamma)}}$ is big.
Hence, this would imply that $\overline{\F_L(\Gamma)}$ is of general type, combined with the result of \cite{Behrens}.
Note that Behrens proved $\overline{\F_L(\Gamma)}$ has canonical singularities for $n \geq 13$ and $F\neq\Q(\sqrt{-1}),\Q(\sqrt{-2}),\Q(\sqrt{-3})$.
This condition gives a restriction to $(L, \l\ ,\ \r)$ when we apply the above strategy.
In this paper, we give a solution to (A) in a quantitative manner with respect to $a$.
The remaining problem, namely the effectiveness of $(n+1-a)\L-\Delta$, can be resolved if one constructs a non-zero cusp form on $\D_L$ of weight $n+1-a<n+1$; we do not consider this (see Remark \ref{rem:Kudla} and Subsection \ref{mainappI}).

\subsection{Main results}
\label{subsection:main_results}
Let $X_L\defeq\F_L(\U(L))$, $\M(a)\defeq\M_{\U(L)}(a)$ and $\theta\defeq\prod_{p}p$
where $p$ runs over any prime number which divides $D$ and $\det(L)$.
Let us introduce an important assumption.
\[(\heartsuit)\ \SU(L')\ \mathrm{and}\ \SU(\ell^{\perp}\cap L)\ \mathrm{are\ principal\ for\ any}\ [\ell]\in \mathcal{R}_L(F),\ \mathrm{where}\ L'\defeq\ell\OO_F\oplus (\ell^{\perp}\cap L)\subset L.\]
The definition of ``principal" is given in Subsection \ref{subsection:tools}; this is the condition under which the localization of a unitary group is a paraholic subgroup, discussed in \cite{BP, Prasad}.
The set $\mathcal{R}_L(F)$ is defined in (\ref{eq:R_L(F)}) and below in Section \ref{section:ramification_divisors}, which is a unitary analog of \cite[Lemma 4.1]{Ma}.
It is the set of reflective vectors, defining branch divisors geometrically, up to the action of unitary groups.
The main theorem in this paper is as follows.
\begin{thm}[Theorem \ref{thm:volume_conclusion_oe}]
\label{mainthm:big}
Let $L$ be a primitive Hermitian lattice over $\OO_F$ of signature $(1,n)$ with $n>2$.
Assume $(\heartsuit)$.
Then, for a positive integer $a$, the line bundle $\M(a)$ is big if $\dim X_L=n$ or $\theta$ is sufficiently large.
\end{thm}

It follows that the reflective obstructions can be resolved for $\F_L(\Gamma)$ with sufficiently large $n$ or $\theta$.
Next, we work on specific lattices. 
We say $L$ is \textit{unramified square-free} if $\det(L)$ is odd square-free and any prime divisor $p$ of $\det(L)$ is unramified in $F_0$, introduced below.
We will prove that they satisfy $(\heartsuit)$ and stronger inequalities regarding $n$ and $\theta$ rather than those given in Theorem \ref{mainthm:big} or Remark \ref{rem:evaluation} hold (Lemma \ref{Lem:lattice_parahoric}, Proposition \ref{ex:unimodular},  \ref{ex:square-free}, Corollary \ref{cor:ams}, Subsection \ref{subsection:general_case} and \ref{subsection:unimod_sqfree_case}).

Throughout this paper, $F_0\neq\Q(\sqrt{-3})$ denotes an imaginary quadratic field, whose discriminant is not a multiple of 4.
This condition corresponds to that the prime 2 does not ramify, required from the Bruhat-Tits theory, and the unit group is $\{\pm 1\}$, needed from the lattice-theoretic computation.
The reason for introducing $F_0$ is because we can prove that certain lattices $L$ over such an $F_0$ in fact satisfies ($\heartsuit$); see 
Proposition \ref{ex:unimodular} and \ref{ex:square-free}.
Let us denote by $-D_0$ the discriminant of $F_0$ below.

\begin{cor}[Corollary \ref{cor:unimodular}, Subsection \ref{subsection:unimod_sqfree_case}]
\label{maincor:big_unimod_sq-free}
\begin{enumerate}
    \item Up to scaling, assume that $L$ is unramified square-free over $\OO_{F_0}$.
If $n>138$, or $D_0>30$ and $n$ is even, then $\M(1)$ is big and hence, $\M(a)$ is big for any $a>1$.
\item Moreover, for fixed $F_0$ and $a\geq 1$, there are only finitely many unramified square-free lattices so that $\M(a)$ is not big with $n>2$, up to scaling.
\end{enumerate}
\end{cor}
For two Hermitian lattices $L$ and $M$ of the same rank over $\OO_F$, we consider an equivalence relation and denote by $L\sim M$ if $L=aM$ as Hermitian forms for some $a\in\OO_F$.
If this is satisfied, we call these Hermitian lattices equivalent up to scaling.

\begin{rem}[Subsection \ref{subsection:general_case}]
\label{rem:evaluation}
Note that a lower bound for $n$ and $\theta$ in Theorem \ref{mainthm:big} can be easily computed.
This is essentially done by estimating certain functions $f_F^{odd}$ and $f_F^{even}$ below. 
For example, we will show that $\M(1)$ is big if $n>582$, and then $\M(a)$ is big for any $a>1$.
In the notation below, this is equivalent to $W(L,F,a)\leq W(L,F,1)<0$ for $n>582$.
\end{rem}
\subsection{Application I: Kodaira dimension}
In this subsection, we assume $n\geq 13$ and $F\neq\Q(\sqrt{-1}), \Q(\sqrt{-2}), \Q(\sqrt{-3})$. 
These assumptions come from \cite[Theorem 4]{Behrens}, which asserts that $\F_L(\Gamma)$ has at worst canonical singularities and branch divisors of the map $\D_L\to\F_L(\Gamma)$ do not exist at the boundary.
Note that $\overline{X_L}$ contains no irregular cusps \cite{Maeda2}.
\label{mainappI}
Assuming the existence of a cusp form referred to in (B), we state an application to the birational type of $X_L$.
\begin{thm}[Corollary \ref{cor:unimodular}, Theorem \ref{thm:gen_type}]
\label{mainthm:gen_type}
Assume that $(\heartsuit)$ holds and there exists a non-zero cusp form of weight lower than $n+1$ with respect to $\U(L)$.
Then,  $X_L$ is of general type if $\dim X_L=n$ or $\theta$ is sufficiently large.
\end{thm}

\begin{rem}
\label{rem:Kudla}
One way to construct a cusp form for $\U(1,n)$ is the theta lifting \cite{Kudla}.
However, this produces only cusp forms of weight greater than $n$.
\end{rem}

\subsection{Application II: Reflective modular forms}
\label{subsection:appII:reflective}
Next, let us consider reflective modular forms.
Let $f$ be a modular form of some weight for $\Gamma$ with a character on $\D_L$.
We say that $f$ is \textit{reflective} if the divisor of $L$ is set-theoretically contained in the ramification divisors of $\D_L\to\F_L(\Gamma)$.
Reflective modular forms appear in many fields of mathematics; see \cite{Gritsenko1, Gritsenko2, GN}.
Gritsenko-Nikulin \cite[Conjecture 2.5.5]{GN} conjectured the finiteness of quadratic lattices admitting a non-zero reflective modular form, and Ma \cite[Corollary 1.9]{Ma} proved it.
Here, we consider an analogous problem for Hermitian lattices.
We say that $L$ is \textit{reflective with slope $r$} for $r>0$ if there exists a reflective modular form on $\D_L$ with slope $r$.
The \textit{slope} of a modular form is the ratio of its (arithmetic) weight, which will be defined later, and its vanishing order at the cusps; for the precise definition of the slope of a modular form, see \cite[Subsection 1.3]{Ma}.
\begin{conj}[Finiteness of Hermitian lattices admitting reflective modular forms]
\label{conj:ams}
For a fixed imaginary quadratic field $F$ and an $r>0$, 
\[\{\mathrm{Hermitian\ reflective\ lattices}\ \mathrm{with\ slope\ less\ than\ }r \}/\sim\]
is a finite set.
\end{conj}
We can partially prove Conjecture \ref{conj:ams} from a computation of the Hirzebruch-Mumford volumes.

\begin{cor}[Corollary \ref{cor:ams}]
\label{maincor:ams}
For a fixed $F_0$ and an $r>0$, 
\[\{\mathrm{Unramified\ square}\mathchar`-\mathrm{free\ reflective\ lattices}\ \mathrm{with\ slope\ less\ than\ }r \mid n>2\}/\sim\]
is a finite set.
\end{cor}

We note that this notion is different from one in quadratic hyperbolic lattices which implies the group spanned by reflections is finite index in an orthogonal group; see \cite{Dol}. 

\subsection{Technical tools}
\label{subsection:tools}
To prove that $\M(a)$ is big, we will use the function $V(L,F)$ (see Definition \ref{def:volume}), which represents the asymptotic growth of the dimension of the space of modular forms vanishing on the ramification divisors.
This function depends only on $L$ and $F$. 
To compute $V(L,F)$, we use Prasad's volume formula \cite{Prasad}.
This approach is different from the one by Gritsenko-Hulek-Sankaran \cite{GHSHM1} and Ma \cite{Ma} who use the calculation of local densities.
We define 
\[W(L,F,a)\defeq V(L,F)-\frac{2a}{n}\cdot \begin{cases}
\displaystyle{\left(1+\frac{1}{a}\right)^{1-n}}&(F\neq\Q(\sqrt{-1}), \Q(\sqrt{-3})),\\
\displaystyle{2\left(1+\frac{3}{a}\right)^{1-n}}&(F=\Q(\sqrt{-1})),\vspace{2pt}\\
\displaystyle{3\left(1+\frac{5}{a}\right)^{1-n}}&(F=\Q(\sqrt{-3})),
\end{cases}
\]
for a positive integer $a>0$. 
Roughly speaking, this gives an estimation of the obstruction space for $\M(a)$ being big.
For the proof of Theorem \ref{mainthm:big}, we use the following criterion.
This is a unitary analog of \cite[Proposition 4.3]{Ma}.

\begin{prop}[Proposition \ref{thm:bigness_criterion}]
\label{mainthm:criterion}
The line bundle $\M(a)$ is big if
\[W(L,F,a)<0.\]
\end{prop}
This criterion reduces the proof of Theorem \ref{mainthm:big} to estimating $V(L,F)$ which occupies the bulk of this paper.
We define functions on $m$ by
\begin{align*}
    f_F^{odd}(m)&\defeq
\frac{3\cdot2^5\cdot(2\pi)^{2m+1}}{ (2m)!\cdot L(2m+1)}\cdot
\begin{cases}
(1+ 2^{4m+1} +  2^{8m+2})&(F\neq\Q(\sqrt{-1}), \Q(\sqrt{-3})), \\
2(3+3\cdot 2^{4m+1} + 2^{8m+2})&(F=\Q(\sqrt{-1})), \\
3(5+2\cdot 3^{4m+1} + 2^{8m+2})& (F=\Q(\sqrt{-3})),\\
\end{cases}\\
    f_F^{even}(m)&\defeq
\frac{2^{2m+5/2}\cdot 3\cdot (2\pi)^{2m}}{ (2m-1)!\cdot\zeta(2m)}\cdot
\begin{cases}
(1+ 2^{4m-1} + 2^{8m-2})&(F\neq\Q(\sqrt{-1}), \Q(\sqrt{-3})), \\
(3+3\cdot 2^{4m-1} + 2^{8m-2})&(F=\Q(\sqrt{-1})), \\
(5+2\cdot 3^{4m-1} + 2^{8m-2})& (F=\Q(\sqrt{-3})).\\
\end{cases}
\end{align*}

\begin{thm}[Theorem \ref{thm:volume_o}, Theorem \ref{thm:volume_e}]
\label{mainthm:volume_oe}
Let $L$ be a primitive Hermitian lattice over $\OO_F$ of signature $(1,2m)$ $\mathrm{(}$resp. $(1,2m-1)\mathrm{)}$ with $m>1$.
Assume $(\heartsuit)$.
Then, we obtain the following:
\[V(L,F)\leq \frac{f_F^{odd}(m)}{\theta}\quad (\mathrm{resp.\ }V(L,F)\leq \frac{f_F^{even}(m)}{\theta}).\]
Moreover, if $P(\a:L)$ holds $\mathrm{(}$see Definition $\ref{def:P(M)}\mathrm{)}$ for some $\a>0$, we have
\[V(L,F)\leq \frac{f_F^{odd}(m)}{D(L)^{1/\a}\cdot \theta}\quad (\mathrm{resp.\ }V(L,F)\leq \frac{f_F^{even}(m)}{D(L)^{1/\a}\cdot \theta}),\]
where $D(L)$ be the exponent of the discriminant group $L^{\vee}/L$.
\end{thm}
Note that the growth of $f_F^{odd}(m)$ and $f_F^{even}(m)$ with respect to $m$ is like the inverse of the Gamma function $1/\Gamma(m)$.
This implies that for a fixed $a$, the inequality $W(L,F,a)<0$ always holds for every pair $(L,F)$ if $n$ is sufficiently large, hence $\M(a)$ is big in that range of $n$.

\begin{rem}
\label{rem:how_large_general}
We will discuss how large values of $m$, in Theorem \ref{mainthm:big}, Corollary \ref{maincor:big_unimod_sq-free} and Theorem \ref{mainthm:volume_oe}, we need to take to satisfy $W(L,F,a)<0$ in Subsection \ref{subsection:general_case} and \ref{subsection:unimod_sqfree_case}.
\end{rem}

Finally, we shall define the notion ``principal" and prepare to discuss $(\heartsuit)$.
Below, we use the special unitary group $G^1\defeq\SU(L)$, group scheme over $\Z$.
To estimate $V(L,F)$, we need to compute the Hirzebruch-Mumford volume of $G^1(\Z)$.
Since $G^1$ is semi-simple and simply connected, we can use Prasad's volume formula \cite[Theorem 3.7]{Prasad}.
Since Prasad's theorem requires an arithmetic subgroup to be principal for some coherent parahoric family, we consider when our arithmetic subgroups satisfy this condition.
Below, $v$ denotes a finite place.
Let $P_v$ be a parahoric subgroup of $G^1(\Q_v)$.
We say that $\{P_v\}_{v}$ is \textit{a coherent parahoric family} if $G^1(\R)\prod_v P_v\subset G^1(\mathbb{A})$ is an open compact subgroup.
We say that $G^1(\Z)$ is \textit{principal} for a coherent parahoric family $\{P_v\}_v$ if $G^1(\Z)=G^1(\Q)\cap\prod_v P_v$ and the closure of the image of $G^1(\Z)$ by the canonical embedding $\iota_v:G^1(\Q)\hookrightarrow G^1(\Q_v)$ is $P_v$.
It follows from the strong approximation theorem and the proof of \cite[Proposition 1.6]{RC} that the closure of $\iota(G^1(\Z))$ is $G^1(\Z_v)$.
Moreover, it follows that $G^1(\Z)=G^1(\Q)\cap\prod_v G^1(\Z_v)$.
Hence, combining these observations, it follows that $G^1(\Z)$ is principal with respect to $\{G^1(\Z_v)\}_v$ if $G^1(\Z_v)$ is parahoric for any $v$. 
Accordingly, we will compute the volume function $V(L,F)$ under $(\star)$ on $L$.
\[(\star)\ \SU(L\otimes\Z_v)\ \mathrm{is\ a\ parahoric\ subgroup\ of\ } \SU(L\otimes\Q_v)\ \mathrm{for\ any\ }v\nmid\infty.\]
Condition $(\star)$ on $L$ implies that $\SU(L)$ is principal. 
Hence, $(\heartsuit)$ can be rephrased as follows;
\[\ell\OO_F\oplus (\ell^{\perp}\cap L)\ \mathrm{and}\ \ell^{\perp}\cap L\ \mathrm{satisfy\ }(\star)\ \mathrm{for\ any}\ [\ell]\in \mathcal{R}_L(F).\]

\begin{rem}
\begin{enumerate}
    \item Hermitian lattices satisfying $(\star)$ actually exist; see Proposition \ref{ex:unimodular} and \ref{ex:square-free}.
    \item Condition $(\star)$ holds for the special linear group \cite[Example 3.2.4]{Thilmany}, i.e., $\SL_n(\Z_v)$ is parahoric for any $v$.
    \item From \cite[Proposition 1.4 (iv)]{BP}, if $G^1(\Z)$ is maximal, then $G^1(\Z_v)$ is parahoric for any $v$.
    Note that the maximal arithmetic subgroups are classified in \cite[Theorem 2.6]{RC}.
    \item Hijikata \cite[Introduction]{Hijikata} stated that the maximal compact open subgroups of an algebraic group over $p$-adic fields can be obtained from the stabilizer of a maximal lattice.
Bruhat \cite[Section 5]{Bruhat} proved it for unitary groups.
On the other hand, Gan-Hanke-Yu \cite[Introduction, Section 3]{GHY} showed that the stabilizer of any maximal Hermitian lattice in a unitary group over $p$-adic fields is a maximal parahoric subgroup except when the field extension is split.
\end{enumerate}
\end{rem}

\begin{rem}
We refer to the relationship between modular varieties of non-general type and reflective modular forms, and moduli representations of ball quotients.
\begin{enumerate}
\item Gritsenko \cite{Gritsenko1, Gritsenko2} constructed reflective modular forms and showed that some orthogonal modular varieties have negative Kodaira dimension.
The author and Odaka \cite{MO} formulated the notion ``special reflective modular forms" and proved that some orthogonal or unitary modular varieties are Fano (e.g., the moduli space of Enriques surfaces).
In these works, reflective modular forms played an important role.
In this paper, we deal with these modular forms in Subsection \ref{subsection:ams}, and show a certain finiteness result (Corollary \ref{cor:ams}).

\item Deligne-Mostow \cite{DM} realized some ball quotients as the periods of hypergeometric forms, and consequently proved that they are related to moduli spaces of some weighted points in the projective line.
On the other hand, Allcock-Carlson-Toledo \cite{ACT1, ACT2} showed that some ball quotients are moduli spaces of cubic surfaces or threefolds.
In this context, Dolgachev-Kond\={o} \cite[Section 1]{DK} conjectured that all ball quotients arising from the Deligne-Mostow theory are related to the moduli spaces of K3 surfaces.
\end{enumerate}
\end{rem}

\subsection{Outline of the proof of Theorem \ref{mainthm:big}}
First, we prove a criterion (Proposition \ref{mainthm:criterion}) asserting when the line bundle $\M(a)$ is big.
Since the branch divisors with higher branch indices may occur in our setting unlike orthogonal modular varieties, we need to classify them in more detail than \cite{Ma}.
Based on the classification, Proposition \ref{mainthm:criterion} follows from  the Hirzebruch-Mumford proportionality principle.
Second, by using Prasad's volume formula \cite[Theorem 3.7]{Prasad}, we compute the Hirzebruch-Mumford volume of  principal arithmetic subgroups.
The application of Prasad's volume formula to birational geometry seems to be new and is one of the differences from the previous studies on the geometry of modular varieties.
Our work is based on the classification of the maximal reductive quotient of the reduction of the smooth integral models \cite{Cho_case1, Cho_case2, GY}.
Combining this computation (Theorem \ref{mainthm:volume_oe}) with the above criterion (Proposition \ref{mainthm:criterion}), it follows that $\M(a)$ is big if $n$ is sufficiently large.
This implies Theorem \ref{mainthm:big}.
To obtain more explicit estimates, we will evaluate $f_F^{odd}(m)$ and $f_F^{even}(m)$ in Subsection \ref{subsection:general_case} and  \ref{subsection:unimod_sqfree_case}.

\subsection{Organization of the paper}
In Section \ref{section:dimension_formula}, we describe the asymptotic behavior of the dimension of modular forms in terms of the Hirzebruch-Mumford volume.
In Section \ref{section:ramification_divisors}, we clarify the description of ramification divisors in terms of Hermitian lattices.
In Section \ref{section:criterion}, we show a criterion when the line bundle $\M(a)$ is big, by using the Hirzebruch-Mumford volume.
In Section \ref{section:calculation}, we recall Prasad's volume formula.
In Section \ref{section:comp_local_factors}, we compute the local factors appearing in the Hirzebruch-Mumford volume.
In Section \ref{section:volume_estimation}, we prove $V(L,F)\leq \theta^{-1}f_F^{odd}(m)$ or $\theta^{-1}f^{even}(m)$.
This calculation shows that $\M(a)$ is big for sufficiently large $n$.
In Section \ref{section:conclusion}, we state the main results and estimate the value of the function $V(L,F)$ explicitly.

\section{Dimension formula}
\label{section:dimension_formula}
In this section, we study the dimension formula of the space of modular forms.
Gritsenko-Hulek-Sankaran \cite{GHSHM1} derived a formula for orthogonal modular forms from Hirzebruch's proportionality principle obtained by Mumford \cite{Mumford}.

Let $L$ be a free $\OO_F$-module of rank $n+1$ and $\l\ ,\ \r$ be an $\OO_F$-valued Hermitian form on $L$:
\[\l\ ,\ \r:L\times L\to\OO_F,\]
for $n>2$.
We call a pair $(L,\l\ ,\ \r)$ a \textit{Hermitian lattice} and assume its signature is $(1,n)$.
Here, Hermitian forms are supposed to be complex linear in the first argument and complex
conjugate linear in the second argument.
We say that $L$ is \textit{primitive} if there does not exist a Hermitian lattice $L'\subset L$ of the same rank as $L$, so that the quotient $L/L'$ is a non-trivial torsion $\OO_F$-module.
Let $V\defeq L\otimes_{\Z}\Q$.
We define the \textit{dual lattice} $L^{\vee}$ of $L$ as
\[L^{\vee}\defeq\{v\in V\mid \l v,w\r\in\OO_F\ \mathrm{for\ any\ }w\in L\}\]
We say that $L$ is \textit{unimodular} if $L=L^{\vee}$.
We call $A_L\defeq L^{\vee}/L$ the \textit{discriminant group}, which is a finite $\OO_F$-module.
\begin{rem}
Note that the definition of ``unimodular" considered in this paper is Allcock's one \cite{Allcock1}, different from \cite{Maeda1, Maeda2, WW}; see also \cite[Subsection 2.1]{WW}.
\end{rem}

Let $\D_L$ be the Hermitian symmetric domain associated with $\U(L\otimes_{\Z}\R)$ and $\D^c_L$ be its compact dual.
In other words, $\D_L$ is the $n$-dimensional complex ball and $\D^c_L$ is the $n$-dimensional projective space.
First, we define modular forms.
Let $V_{\C}\defeq V\otimes_F\C$ and 
\[\D_L^0\defeq\{v\in V_{\C}\mid [v]\in \D_L\}\]
be the principal $\C^{\times}$-bundle on $\D_L$.
Let $k$ be a non-negative integer and $\chi:\Gamma\to\C^{\times}$ be a character.
We say that a holomorphic function $f$ on $\D_L^0$ is a \textit{modular form of weight $k$ for $\Gamma$ with character $\chi$} if the following holds:
\begin{align*}
    f(tz)&=t^{-k}f(z),\\
    f(\gamma z)&=\chi(\gamma)f(z),
\end{align*}
for all $t\in\C^{\times}$ and $\gamma\in\Gamma$.
We denote by $M_k(\Gamma,\chi)$ the set consisting of modular forms of weight $k$ for $\Gamma$ with a character $\chi$.
Let $M_k(\Gamma)\defeq M_k(\Gamma,\id)$.

Next, we shall define cusp forms.
For $f\in M_k(\Gamma,\chi)$, we have a Fourier-Jacobi expansion via the realization of the Siegel domain model of $\D_L$ by choosing cusps.
For the corresponding function $g_f$ on the Siegel domain model, the Fourier-Jacobi expansion is 
\[g_f(\tau, u)=\sum_{n\in\Z}a_n(u)\exp(2\pi\sqrt{-1}\tau)\]
for $\tau\in\C$ and $u\in\C^{n-1}$.
We say that $f$ is a \textit{cusp form} if $a_0(u)=0$ at all cusps and denote by $S_k(\chi,\Gamma)$ the set consisting of cusp forms of weight $k$ for $\Gamma$ with a character $\chi$.
Let $S_k(\Gamma)\defeq S_k(\Gamma,\id)$.

For an arithmetic subgroup $\Gamma\subset\U(L\otimes_{\Z}\Q)$, if $\Gamma$ acts on $\D_L$ freely, the \textit{Hirzebruch-Mumford volume} of $\Gamma$ is defined by
\[\v(\Gamma)\defeq\frac{e(\D_L/\Gamma)}{e(\D^c_L)}=\frac{e(\D_L/\Gamma)}{n+1}.\]
If $\Gamma$ does not act freely, we take a finite index normal subgroup $\Gamma'\triangleleft\Gamma$ which acts on $\D_L$ freely and define
\[\v(\Gamma)\defeq\frac{\v(\Gamma')}{[\overline{\Gamma}:\Gamma']},\]
where $\overline{\Gamma}$ is $\Gamma$ modulo center.
Note that the Hirzebruch-Mumford volume does not depend on the choice of $\Gamma'$.
Before we recall the celebrated result by Mumford, let us refer to the difference between weight, which is the \textit{arithmetic weight}, in this paper and the \textit{geometric weight} in some literature.
A modular form of weight $(n+1)k$ for $\Gamma$ with the character  $\det^k$ defines a $k$-pluricanonical form on the smooth part of $\F_L(\Gamma)$.
We denote $S_k^{\mathrm{geom}}(\Gamma)$ by the space consisting of such cusp forms.
This quantity $(n+1)k$ is called a geometric weight $k$; for the orthogonal groups, see \cite{GHSHM1}.
Here, $n$ is the dimension of $\F_L(\Gamma)$.
In other words, a geometric weight $k$ is a weight $(n+1)k$.
Note that what Ash-Mumford-Rapoport-Tai \cite{AMRT} and Mumford \cite{Mumford} call weight in their paper is geometric weight.
Under this notion, we can obtain a dimension formula for the space of modular forms of weight $(n+1)k$ on $\D_L$.
\begin{thm}[{\cite[Corollary 3.5]{Mumford}}]
\label{thm:Mumford}
Assume that $\Gamma$ is neat, that is, for any $g\in\Gamma$, the group generated by the eigenvalues of the action of $g$ on $\D_L$ is torsion-free.
Then, 
\[\dim S^{\mathrm{geom}}_k(\Gamma)=\dim S_{(n+1)k}(\Gamma,\det^k) = \v(\Gamma) h^0(\omega_{\D_L^c}^{1-k})+P_1(k),\]
for some polynomial $P_1(k)$ of degree at most $\dim\F_L(\Gamma)-1=n-1$ with respect to $k$.
\end{thm}
We shall apply this result to unitary groups and obtain a formula for the asymptotic growth of the dimension of the space of cusp forms.

\begin{prop}
\label{Prop:asymtotic_growth}
We assume that 
\begin{enumerate}
    \item If $-\id\in\Gamma$, then $\chi(-\id)=(-1)^k$.
    \item If $F=\Q(\sqrt{-1})$ and $\sqrt{-1}\id\in\Gamma$, then $\chi(\sqrt{-1}\id)=\sqrt{-1}^k$.
    \item If $F=\Q(\sqrt{-3})$ and $\omega\id\in\Gamma$, then $\chi(\omega\id)=\omega^k$.
\end{enumerate}
Then, 
\[\dim S_k(\Gamma,\chi)=\frac{1}{n!}\v(\Gamma)k^n+O(k^{n-1})\]
for sufficiently divisible $k$.
\end{prop}
\begin{proof}
We follow the proof of \cite[Proposition 1.2]{GHSHM1} or \cite[Proposition 2.1]{Tai}.
By applying the Lefschetz fixed point theorem \cite[Appendix to Section 2]{Tai}, we may assume that $\Gamma$ is neat.
Note that we use the assumption on $\chi$ here.
For sufficiently divisible $k$, the asymptotic growth of the dimension of the space of cusp forms of weight $k$ with a character $\chi$ remains the same even when the character is replaced with the trivial character because $\L$ and $\L\otimes\chi$ only differ by torsion.
Hence, we also assume that $\chi$ is trivial.

Note that $S_k(\Gamma)=H^0(\overline{\F_L(\Gamma)},\L^{\otimes k}(-\Delta))$.
We calculate the dimension of modular forms by using the  Hirzebruch-Riemann-Roch theorem and Hirzebruch's proportionality principle (Theorem \ref{thm:Mumford}).
First, since $\L$ is big and nef, by the Kawamata-Viehweg vanishing theorem, we  obtain
\begin{equation}
\label{proof:KW}
    \chi(\overline{\F_L(\Gamma)},\L^{\otimes k}(-\Delta))=h^0(\overline{\F_L(\Gamma)},\L^{\otimes k}(-\Delta))
\end{equation}
for sufficiently divisible $k$.
When we think of the above as a function of $k$, the Riemann-Roch polynomial is given by 
\begin{equation}
\label{proof:c_1}
\chi(\overline{\F_L(\Gamma)},\L^{\otimes k}(-\Delta))=\frac{c_1^n(\L^{\otimes k}(-\Delta))}{n!}k^n + O(k^{n-1}).
\end{equation}

On the other hand, by Theorem \ref{thm:Mumford},
\begin{align*}
   h^0(\overline{\F_L(\Gamma)},\L^{\otimes (n+1)k}(-\Delta))&=h^0(\overline{\F_L(\Gamma)},(\L^{\otimes k}\otimes\det^k)^{\otimes n+1}(-\Delta))\\ 
   &=\dim S_{(n+1)k}(\Gamma,\det^k)\\
   &=\dim S_k^{\mathrm{geom}}(\Gamma)\\
   &=\v(\Gamma)h^0(\omega^{1-k}_{\P^n})+O(k^{n-1})
\end{align*}
for sufficiently divisible $k$.
Since the compact dual of $\D_L$ is $\P^n$, a standard calculation, for sufficiently divisible $k$, gives
\begin{align}
\label{proof:cpt_dual}
   \chi(\P^n,\omega_{\P^n}^{1-k})&=h^0(\P^n,\omega_{\P^n}^{1-k}) \notag \\
   &=\frac{(n+1)^n}{n!}k^n+O(k^{n-1})
\end{align}
as a function of $k$.
Hence, from (\ref{proof:c_1}) and (\ref{proof:cpt_dual}), it follows that 
\[\frac{c_1^n(\L^{\otimes (n+1)k}(-\Delta))}{n!}=\frac{(n+1)^n}{n!}\v(\Gamma).\]
This implies
\[\frac{c_1^n(\L^{\otimes k}(-\Delta))}{n!}=\frac{1}{n!}\v(\Gamma).\]
Combining this with (\ref{proof:KW}), we conclude that
\begin{align*}
    \dim S_k(\Gamma)&=h^0(\overline{\F_L(\Gamma)},\L^{\otimes k}(-\Delta))\\
    &=\frac{1}{n!}\v(\Gamma)k^n+O(k^{n-1}).
\end{align*}

\end{proof}
\begin{rem}
\begin{enumerate}
\item Gritsenko-Hulek-Sankaran \cite[Proposition 1.2]{GHSHM1} derived a similar dimension formula for orthogonal groups.

\item The asymptotic growth of the dimension of the space of modular forms is the same as that of cusp forms because only line bundles supported on the boundary contribute to their difference; see \cite{GHSHM1}.
\end{enumerate}
\end{rem}

\section{Branch divisors}
\label{section:ramification_divisors}
We already know that the canonical divisor $K_{\overline{\F_L(\Gamma)}}$ is described as
\[
\begin{cases}
    \displaystyle{(n+1)\L-\frac{\overline{B_2}}{2}-\Delta} &(F\neq\Q(\sqrt{-1}), \Q(\sqrt{-3})),\vspace{2pt}\\
    \displaystyle{(n+1)\L-\frac{\overline{B_2}}{2}-\frac{3}{4}\overline{B_4}-\Delta} & (F=\Q(\sqrt{-1})),\vspace{2pt}\\
    \displaystyle{(n+1)\L-\frac{\overline{B_2}}{2}-\frac{2}{3}\overline{B_3}-\frac{5}{6}\overline{B_6}-\Delta} & (F=\Q(\sqrt{-3})),\\
\end{cases}
\]
in $\Pic(\overline{\F_L(\Gamma)})\otimes_{\Z}\Q$ from \cite{Behrens}.
In this section, we shall study the branch divisors via Hermitian lattices.
Below, we shall mainly work on $\Gamma=\U(L)$.

Recall that the \textit{reflection} $\sigma_{\ell,\xi}$ with respect to a primitive vector $\ell\in L$ with $\l\ell,\ell\r< 0$ and $\xi\in\OO_F^{\times}\backslash\{1\}$ is defined by
\[\sigma_{\ell,\xi}:V\to V,\quad v\to v-(1-\xi)\frac{\l v,\ell\r}{\l\ell,\ell\r}\ell.\]
By \cite[Proposition 2]{Behrens}, the branch  divisors are the union of fixed divisors of reflections:
\begin{align*}
    B_2&=\left(\bigcup_{\ell\in A_2}H(\ell)\right)/\Gamma,\\
    B_3&=\left(\bigcup_{\ell\in A_3}H(\ell)\right)/\Gamma\quad (F=\Q(\sqrt{-3})),\\
    B_4&=\left(\bigcup_{\ell\in A_4}H(\ell)\right)/\Gamma\quad (F=\Q(\sqrt{-1})),\\
    B_6&=\left(\bigcup_{\ell\in A_6}H(\ell)\right)/\Gamma\quad (F=\Q(\sqrt{-3})),
\end{align*}
where 
\begin{align*}
    A_2&=\{\ell\in L\mid \xi\id\cdot\sigma_{\ell,-1}\in\U(L)\ \mathrm{for\ a\ }\xi\in\{\pm 1\}\},\\
    A_3&=\{\ell\in L\mid\xi\id\cdot\sigma_{\ell,\omega^k}\in\U(L)\ \mathrm{for\ }\xi\in\OO_{\Q(\sqrt{-3})}^{\times}\ \mathrm{and}\  3\nmid k\ \mathrm{with}\ \xi\cdot\omega^k\ \mathrm{order}\ 3\}, \\
    A_4&=\{\ell\in L\mid \xi\id\cdot\sigma_{\ell,\sqrt{-1}^k}\in\U(L)\ \mathrm{for\ }\xi\in\OO_{\Q(\sqrt{-1})}^{\times}\ \mathrm{and}\  4\nmid k\ \mathrm{with}\ \xi\cdot\sqrt{-1}^k\ \mathrm{order}\ 4\},\\
    A_6&=\{\ell\in L\mid \xi\id\cdot\sigma_{\ell,(-\omega)^k}\in\U(L)\ \mathrm{for\ }\xi\in\OO_{\Q(\sqrt{-3})}^{\times}\ \mathrm{and}\  6\nmid k\ \mathrm{with}\ \xi\cdot(-\omega)^k\ \mathrm{order}\ 6\}.
\end{align*}
Here, $H(\ell)$ denotes a Heegner divisor on $\D_L$ with respect to $\ell$:
\[H(\ell)\defeq\{v\in \D_L\mid \l v,\ell\r=0\}.\]
We say that $\ell$ is \textit{reflective with index $d$} if $\ell\in A_d$.
We will investigate branch divisors that obstruct the automorphic line bundle with zeros on branch divisors from being big.
First, we classify them according to \cite[Lemma 4.1]{Ma}.
For a primitive vector $l\in L$ with $\l\ell,\ell\r<0$, let $K_{\ell}\defeq \ell^{\perp}\cap L$ be its orthogonal complement, $\Div(\ell)$ be the ideal generated by $\{\l v,\ell\r\mid v\in L\}$, and 
\[I_{\ell}\defeq\l\ell,\ell\r\cdot\Div(\ell)^{-1}\subset\OO_F\]
also be an ideal of $\OO_F$.
Here, note that $K_{\ell}$ can be naturally considered as a Hermitian lattice whose signature is $(1,n-1)$.
Then, we have
\[L/\ell\OO_F\oplus K_{\ell}\cong\OO_F/I_{\ell}.\]
Note that, unlike the case of orthogonal groups, $\Div(\ell)$ is not  a principal ideal in general.

\begin{lem}
\label{Lem:classification_reflective_vectors_-1}
Let $F=\Q(\sqrt{-1})$.
Then,
\begin{enumerate}
    \item $\ell$ is reflective of index $2$ if and only if $L\supset\ell\OO_{\Q(\sqrt{-1})}\oplus K_{\ell}$ and $L/\ell\OO_{\Q(\sqrt{-1})}\oplus K_{\ell}\cong \OO_{\Q(\sqrt{-1})}/2\OO_{\Q(\sqrt{-1})}$ holds.
    \item $\ell$ is reflective of index $4$ if and only if one of the following holds:
    \begin{enumerate}
    \item $L=\ell\OO_{\Q(\sqrt{-1})}\oplus K_{\ell}$.
    \item $L\supset\ell\OO_{\Q(\sqrt{-1})}\oplus K_{\ell}$ and $L/\ell\OO_{\Q(\sqrt{-1})}\oplus K_{\ell}\cong \OO_{\Q(\sqrt{-1})}/(1+\sqrt{-1})\OO_{\Q(\sqrt{-1})}$.
    \end{enumerate}
\end{enumerate}

\end{lem}
\begin{proof}

(1) $\ell$ is reflective with index 2 if and only if 
\[\frac{2\l v,\ell\r}{\l\ell,\ell\r}\in\OO_F\ \mathrm{and}\ (1+\sqrt{-1})\frac{\l v,\ell\r}{\l\ell,\ell\r}\not\in\OO_F\]
for all $v\in L$, and this equals 
\[2\in I_{\ell}\ \mathrm{and}\ 1+\sqrt{-1}\not\in I_{\ell}.\]
This shows $I_{\ell}=2\OO_{\Q(\sqrt{-1})}$.
Thus the isomorphism $L/\OO_{\Q(\sqrt{-1})}\ell\oplus K_{\ell}\cong \OO_{\Q(\sqrt{-1})}/2\OO_{\Q(\sqrt{-1})}$ is proved.
The sufficient condition can be proved in a similar way as in the proof of \cite[Lemma 4.1]{Ma}.

(2) As in (1), it suffices to determine an ideal $I_{\ell}$ containing $1+\sqrt{-1}$.
This holds if and only if $I_{\ell}=\OO_{\Q(\sqrt{-1})}$ or $(1+\sqrt{-1})\OO_{\Q(\sqrt{-1})}$.
\end{proof}

\begin{lem}
\label{Lem:classification_reflective_vectors_-3}
Let $F=\Q(\sqrt{-3})$. 
Then,
\begin{enumerate}
    \item $\ell$ is reflective of index $2$ if and only if $L\supset\ell\OO_F\oplus K_{\ell}$ and $L/\ell\OO_F\oplus K_{\ell}\cong \OO_F/2\OO_F$ holds.
    \item $\ell$ is reflective of index $3$ if and only if $L\supset\ell\OO_{\Q(\sqrt{-3})}\oplus K_{\ell}$ and $L/\ell\OO_{\Q(\sqrt{-3})}\oplus K_{\ell}\cong \OO_{\Q(\sqrt{-3})}/\sqrt{-3}\OO_{\Q(\sqrt{-3})}$ holds.
    \item $\ell$ is reflective of index $6$ if and only if $L=\ell\OO_{\Q(\sqrt{-3})}\oplus K_{\ell}$ holds.
\end{enumerate}

\end{lem}
\begin{proof}
We follow the strategy in the proof of Lemma \ref{Lem:classification_reflective_vectors_-1}.

(1) It suffices to determine an ideal $I_{\ell}$ containing $2$ and not containing $1+\omega=-\omega^2$.
This holds if and only if $I_{\ell}=2\OO_{\Q({\sqrt{-3}})}$.

(2) It suffices to determine an ideal $I_{\ell}$ containing $1-\omega=\sqrt{-3}\omega$ and not containing $-\omega^2$.
This holds if and only if $I_{\ell}=\sqrt{-3}\OO_{\Q({\sqrt{-3}})}$.

(3) It suffices to determine an ideal $I_{\ell}$ containing $1+\omega=-\omega^2$.
This holds if and only if $I_{\ell}=\OO_{\Q({\sqrt{-3}})}$.
\end{proof}

\begin{lem}
\label{Lem:classification_reflective_vectors_ow1}
We assume that $F\neq\Q(\sqrt{-1})$ and the discriminant $-D$ of $F$ is a multiple of 4.
Then, $\ell$ is reflective of index $2$ if and only if one of the following holds:
    \begin{enumerate}
        \item $L=\ell\OO_F\oplus K_{\ell}$.
        \item $L\supset\ell\OO_{F}\oplus K_{\ell}$ and $L/\ell\OO_F\oplus K_{\ell}\cong \OO_F/2\OO_F$.
        \item $L\supset\ell\OO_F\oplus K_{\ell}$ and $L/\ell\OO_F\oplus K_{\ell}\cong \OO_F/\p\OO_F$, where $\p$ is a prime ideal such that $2\OO_F=\p^2$.
    \end{enumerate}

\end{lem}
\begin{proof}
This can be proved in a similar way as Lemma \ref{Lem:classification_reflective_vectors_-1} or Lemma \ref{Lem:classification_reflective_vectors_-3}.
\end{proof}

\begin{lem}
\label{Lem:classification_reflective_vectors_ow2}
We assume that the discriminant $-D$ of $F$ satisfies $-D\equiv 1\bmod 8$.
Let $\p_1$ and $\p_2$ be given prime ideals satisfying  $2\OO_F=\p_1\p_2$.
Then, $\ell$ is reflective of index $2$ if and only if one of the following holds:
    \begin{enumerate}
        \item $L=\ell\OO_F\oplus K_{\ell}$.
        \item $L\supset\ell\OO_F\oplus K_{\ell}$ and $L/\ell\OO_F\oplus K_{\ell}\cong \OO_F/2\OO_F$.
    \item $L\supset\ell\OO_F\oplus K_{\ell}$ and $L/\ell\OO_F\oplus K_{\ell}\cong \OO_F/\p_1\OO_F$.
    \item $L\supset\ell\OO_F\oplus K_{\ell}$ and $L/\ell\OO_F\oplus K_{\ell}\cong \OO_F/\p_2\OO_F$.
    \end{enumerate}

\end{lem}
\begin{proof}
This can be proved in a similar way as Lemma \ref{Lem:classification_reflective_vectors_-1} or Lemma \ref{Lem:classification_reflective_vectors_-3}.
\end{proof}

\begin{lem}
\label{Lem:classification_reflective_vectors_ow3}
We assume that $F\neq\Q(\sqrt{-3})$ and its discriminant $-D$ of $F$ satisfies $-D\equiv 5\bmod 8$.
Then, $\ell$ is reflective of index $2$ if and only if one of the following holds:
    \begin{enumerate}
        \item $L=\ell\OO_F\oplus K_{\ell}$.
        \item $L\supset\ell\OO_F\oplus K_{\ell}$ and $L/\ell\OO_F\oplus K_{\ell}\cong \OO_F/2\OO_F$.
    \end{enumerate}

\end{lem}
\begin{proof}
This can be proved in a similar way as Lemma \ref{Lem:classification_reflective_vectors_-1} or Lemma \ref{Lem:classification_reflective_vectors_-3}.
\end{proof}

We denote by $\mathcal{R}_L(F,i)$ the set of $\U(L)$-equivalent classes of reflective vectors in $L$ of index $i$ and define the set
\begin{align}
\label{eq:R_L(F)}
    \mathcal{R}_L(F)\defeq \coprod_i \mathcal{R}_L(F,i).
\end{align}
Reflective vectors and this set $\mathcal{R}_L(F)$ play roles in classical lattice theory \cite{Nik} and algebraic geometry \cite{Dol}.
For convenience, we will write the imaginary quadratic field $F$, defining $L$,  explicitly.
Note that any element $[\ell]\in\mathcal{R}_L(F,i)$ corresponds to an irreducible component of the branch divisors with branch index $i$.
Moreover, let 

\begin{align*}
    \mathcal{R}_L(\Q(\sqrt{-1}),4)_{I}&\defeq
\left\{[\ell]\in \mathcal{R}_L(\Q(\sqrt{-1}),4)\mid L=\ell\OO_F\oplus K_{\ell}\right\},\\
\mathcal{R}_L(\Q(\sqrt{-1}),4)_{II}&\defeq
\left\{[\ell]\in \mathcal{R}_L(\Q(\sqrt{-1}),4)\mid L/\ell\OO_F\oplus K_{\ell}\cong\OO_F/(1+\sqrt{-1})\OO_F\right\},\\
\mathcal{R}_L(F,2)_{I}&\defeq
\left\{[\ell]\in \mathcal{R}_L(F,2)\mid L=\ell\OO_F\oplus K_{\ell}\right\},\\
\mathcal{R}_L(F,2)_{II}&\defeq
\left\{[\ell]\in \mathcal{R}_L(F,2)\mid L/\ell\OO_F\oplus K_{\ell}\cong\OO_F/2\OO_F\right\},\\
\mathcal{R}_L(F,2)_{III}&\defeq
\begin{cases}
\left\{[\ell]\in \mathcal{R}_L(F,2)\mid L/\ell\OO_F\oplus K_{\ell}\cong\OO_F/\p\OO_F\right\} & (4|D\ \mathrm{and}\ D\neq -4),  \\
\emptyset & (\mathrm{otherwise}), 
\end{cases}\\
\mathcal{R}_L(F,2)_{IV}&\defeq
\begin{cases}
\{[\ell]\in \mathcal{R}_L(F,2)\mid L/\ell\OO_F\oplus K_{\ell}\cong\OO_F/\p_1\OO_F\} & (-D\equiv 1\bmod 8),  \\
\emptyset & (\mathrm{otherwise}), 
\end{cases}\\
\mathcal{R}_L(F,2)_{V}&\defeq
\begin{cases}
\{[\ell]\in \mathcal{R}_L(F,2)\mid L/\ell\OO_F\oplus K_{\ell}\cong\OO_F/\p_2\OO_F\} & (-D\equiv 1\bmod 8),  \\
\emptyset & (\mathrm{otherwise}).
\end{cases}
\end{align*}

From Lemma \ref{Lem:classification_reflective_vectors_-1}, \ref{Lem:classification_reflective_vectors_-3}, \ref{Lem:classification_reflective_vectors_ow1}, \ref{Lem:classification_reflective_vectors_ow2} and  \ref{Lem:classification_reflective_vectors_ow3}, we have
\begin{align*}
   \mathcal{R}_L(\Q(\sqrt{-1}), 4)&=\mathcal{R}_L(F,4)_I\coprod \mathcal{R}_L(F,4)_{II},\\
   \mathcal{R}_L(\Q(\sqrt{-3}),2)&=\mathcal{R}_L(\Q(\sqrt{-3}),2)_{II},\\
   \mathcal{R}_L(F, 2)&=\mathcal{R}_L(F,2)_I\coprod \mathcal{R}_L(F,2)_{II}\coprod \mathcal{R}_L(F,2)_{III}\coprod \mathcal{R}_L(F,2)_{IV}\coprod \mathcal{R}_L(F,2)_{V},
\end{align*}
for any imaginary quadratic field $F$.
We say that a reflective vector $[\ell]\in \mathcal{R}_L(F)$ is  \textit{split type} if $L=\ell\OO_F\oplus K_{\ell}$, according to \cite{Ma}.
This means that $[\ell]$ is contained in $R(F,2)_I$, $R(\Q(\sqrt{-1}), 4)_I$ or $R(\Q(\sqrt{-3}), 6)$.
Otherwise, we call $[\ell]\in\mathcal{R}_L(F)$ \textit{non-split type}.
Below, for a finite $\OO_F$-module $M$ and a prime element $p$, let us denote by $l(M_P)$ the length of the $p$-part of $M$ as a $\OO_F$-module.
\begin{lem}
\label{Lem:stabilizer}
Let $\Gamma_{\ell}\subset\U(K_{\ell})$ be the stabilizer of a reflective vector $[\ell]\in \mathcal{R}_L(F)$.
\begin{enumerate}
    \item For $[\ell]\in \mathcal{R}_L(F,2)_I, \mathcal{R}_L(\Q(\sqrt{-1}), 4)_I, \mathcal{R}_L(\Q(\sqrt{-3}),6)$, we have $\Gamma_{\ell}=\U(K_{\ell})$.
    \item For $[\ell]\in \mathcal{R}_L(\Q(\sqrt{-1}),4)_{II}$, we have $[\U(K_{\ell}):\Gamma_{\ell}]<2^{r_{1+\sqrt{-1}}}$, where $r_{1+\sqrt{-1}}\defeq l((A_{K_{\ell}})_{1+\sqrt{-1}})$.
    \item For $[\ell]\in \mathcal{R}_L(\Q(\sqrt{-3}),3)$, we have $[\U(K_{\ell}):\Gamma_{\ell}]<3^{r_{\sqrt{-3}}}$, where $r_{\sqrt{-3}}\defeq l((A_{K_{\ell}})_{\sqrt{-3}})$.
    \item For $[\ell]\in \mathcal{R}_L(F,2)_{II}$, we have $[\U(K_{\ell}):\Gamma_{\ell}]<4^{r_2}$, where $r_2\defeq l((A_{K_{\ell}})_2)$.
    \item For $[\ell]\in \mathcal{R}_L(F,2)_{III}$, we have $[\U(K):\Gamma_{\ell}]<2^{r_{\p}}$, where $r_{\p}\defeq l((A_{K_{\ell}})_{\p})$.
    \item For $[\ell]\in \mathcal{R}_L(F,2)_{IV}$, we have $[\U(K):\Gamma_{\ell}]<2^{r_{\p_1}}$, where $r_{\p_1}\defeq l((A_{K_{\ell}})_{\p_1})$.
    \item For $[\ell]\in \mathcal{R}_L(F,2)_{V}$, we have $[\U(K):\Gamma_{\ell}]<2^{r_{\p_2}}$, where $r_{\p_2}\defeq l((A_{K_{\ell}})_{\p_2})$.
\end{enumerate}
\end{lem}
\begin{proof}
This can be proved in the same way as \cite[Lemma 4.2]{Ma}.
\end{proof}

\section{Reflective obstructions}
\label{section:criterion}

We shall study when the line bundle $\M(a)$ is big in terms of the asymptotic growth of the dimension of the space of modular forms.
Since the line bundle $\L$ is big, the obstruction for $\M(a)$ being big is the branch divisors $B_i$.
To estimate this obstruction, we use the unitary analog of the construction  \cite[Proposition 4.1]{GHSHM2}.

For $F\neq\Q(\sqrt{-1}), \Q(\sqrt{-3})$, $\ell_1, \dots, \ell_r$ denotes a complete system of representatives of the set $\mathcal{R}_L(F,2)$.
For $F=\Q(\sqrt{-1})$, let $\ell_{2,1}, \dots, \ell_{2, s_{2}}$ (resp. $\ell_{4,1},\dots,\ell_{4,s_{4}}$) be a complete system of representatives of the set $\mathcal{R}_L(\Q(\sqrt{-1}),2)$ (resp. $\mathcal{R}_L(\Q(\sqrt{-1}),4)$).
For $F=\Q(\sqrt{-3})$, let $\ell_{2,1}, \dots, \ell_{2, t_{2}}$ (resp. $\ell_{3,1},\dots,\ell_{3,t_{3}}$, $\ell_{6,1},\dots,\ell_{6,t_{6}}$) be a complete system of representatives of the set $\mathcal{R}_L(\Q(\sqrt{-3}),2)$ (resp. $\mathcal{R}_L(\Q(\sqrt{-3}),3)$, $\mathcal{R}_L(\Q(\sqrt{-3}),6)$).

Below, we reduce the estimation of the space consisting of sections of $\M(a)$ to the spaces of modular forms.
Before stating the precise statement, let us explain the method \textit{quasi-pullback} of a modular form, only used in the proof of Lemma \ref{Lem:exact_sequence}.
For the formal and detailed description of this notion, see \cite{Bor, BKPS}.
Let us consider a modular form $F$ on $\D_L$ and an embedding of balls $\D_K\hookrightarrow\D_L$, according to a sublattice $K\subset L$.
When we simply restrict $F$ to $\D_K$, if $\D_K$ is fully contained in $\div(F)$ set-theoretically, its restriction is identically zero as a function on $\D_K$.
To obtain a meaningful modular form, first we divide $F$ by the defining equation of $\div(F)$, containing $\D_K$, and then restrict to $\D_K$, which we call a quasi-pullback of $F$.
We denote the defining equation of a Heegner divisor $H(\ell)$ by $\l\ ,\ell\r$.
If $\div(F)$ is given by Heegner divisors, for instance a union of Heegner divisors with respect to $\ell\in L$, being our case below, then quasi-pullback is the restriction of the function $F/\l\ ,\ell\r^{2j}$ to $\D_K$ for $\ell\in K^{\perp}$, where $j$ is the multiplicity of $H(\ell)$ in $\div(F)$.
Note that the resulting function is in fact a modular form on $\D_K$ and its weight is added $2j$ to the weight of $F$.

\begin{lem}
\label{Lem:exact_sequence}
The following inequalities hold.
\begin{enumerate}
    \item For $F\neq\Q(\sqrt{-1}),\Q(\sqrt{-3})$, when $k$ and $ka$ are even, we have 
    \[h^0(k\M(a)) \geq \dim M_{ka}(\U(L)) - \sum_{i=1}^r\sum_{j=0}^{k/2-1}\dim M_{ka+2j}(\Gamma_i).\]
    \item For $F=\Q(\sqrt{-1})$, when $k$ and $ka$ are multiples of 4, we have 
    \begin{align*}
        h^0(k\M(a)) &\geq \dim M_{ka}(\U(L))- \sum_{i=1}^{s_{2}}\sum_{j_2=0}^{k/4-1}\dim M_{ka+4j_2}(\Gamma_i) - \sum_{i=1}^{s_{4}} \sum_{j_4=0}^{3k/4-1} \dim M_{ka+4j_4}(\Gamma_i).
    \end{align*}
    \item For $F=\Q(\sqrt{-3})$, when $k$ and $ka$ are multiples of 6, we have 
    \begin{align*}
    h^0(k\M(a))\geq &\dim M_{ka}(\U(L))- \sum_{i=1}^{t_{2}}\sum_{j_2=0}^{k/6-1}\dim M_{ka+6j_2}(\Gamma_i)\\
    & - \sum_{i=1}^{t_{3}}\sum_{j_3=0}^{k/3-1}\dim M_{ka+6j_3}(\Gamma_i) - \sum_{i=1}^{t_{6}}\sum_{j_6=0}^{5k/6-1}\dim M_{ka+6j_6}(\Gamma_i).\end{align*}
\end{enumerate}
\end{lem}
\begin{proof}
(1) can be shown in a similar way as \cite[Lemma 4.4]{Ma}.
For a non-negative $j$, there is the quasi-pullback:
\begin{align*}
    H^0(ka\L-j\overline{B_2}) &\to M_{ka+2j}(\Gamma_i)\\
    F &\mapsto \frac{F}{\l\ ,\ell_i \r^{2j}}\Big\rvert_{D_{K_i}}.
\end{align*}
From this, we derive the exact sequence,
\[0\to H^0(ka\L-(j+1)\overline{B_2})\to H^0(ka\L-j\overline{B_2})\to \bigoplus_{i=1}^{r}M_{ka+2j}(\Gamma_i).\]
Iteration for $j=0,\dots,k/2-1$ yields the desired inequality.

(2) As in \cite[Lemma 4.3 (1)]{WW}, since $\sqrt{-1}\id\in\Gamma_i$, the vanishing order of $F$ along $D_{K_i}$ is a multiple of $4$ and $M_t(\Gamma_i) = 0$ unless  $4|t$.
From this, we have the quasi-pullback maps:
\begin{align*}
    H^0(ka\L-2j\overline{B_2}) &\to M_{ka+4j}(\Gamma_i)\\
    F &\mapsto \frac{F}{\l\ ,\ell_i \r^{4j}}\Big\rvert_{D_{K_i}},\\
    H^0(ka\L-j\overline{B_4}) &\to M_{ka+4j}(\Gamma_i)\\
    F &\mapsto \frac{F}{\l\ ,\ell_i \r^{4j}}\Big\rvert_{D_{K_i}}.
\end{align*}

There exist exact sequences:
\begin{align}
    0\to H^0(ka\L-2(j_2+1)\overline{B_2})\to H^0(ka\L-2j_2\overline{B_2})\to \bigoplus_{i=1}^{s_2}M_{ka+4j_2}(\Gamma_i), \label{exact:-1_2} \\
    0\to H^0(ka\L-\frac{k}{2}\overline{B_2}-(j_4+1)\overline{B_4})\to H^0(ka\L-\frac{k}{2}\overline{B_2}-j_4\overline{B_4})\to \bigoplus_{i=1}^{s_4}M_{ka+4j_4}(\Gamma_i).\label{exact:-1_4}
\end{align}

Iteration of (\ref{exact:-1_2}) for $j_2=0,\dots, k/4-1$ and (\ref{exact:-1_4}) for $j_4=0,\dots,3k/4-1$ yields the desired inequality.

(3) As in \cite[Lemma 4.3 (2)]{WW}, since $-\omega\id\in\Gamma_i$, the vanishing order of $F$ along $D_{K_i}$ is a multiple of $6$ and $M_t(\Gamma_i) = 0$ unless  $6|t$.
From this, we have the quasi-pullback maps:
\begin{align*}
    H^0(ka\L-3j\overline{B_2}) &\to M_{ka+6j}(\Gamma_i)\\
    F &\mapsto\frac{F}{\l\ ,\ell_i \r^{6j}}\Big\rvert_{D_{K_i}},\\
    H^0(ka\L-2j\overline{B_3}) &\to M_{ka+6j}(\Gamma_i)\\
    F &\mapsto\frac{F}{\l\ ,\ell_i \r^{6j}}\Big\rvert_{D_{K_i}},\\
    H^0(ka\L-j\overline{B_6}) &\to M_{ka+6j}(\Gamma_i)\\
    F &\mapsto\frac{F}{\l\ ,\ell_i \r^{6j}}\Big\rvert_{D_{K_i}}.
\end{align*}

There exist exact sequences:
\begin{align}
    0\to H^0(ka\L-3(j_2+1)\overline{B_2})\to H^0(ka\L-3j_2\overline{B_2})\to \bigoplus_{i=1}^{t_2}M_{ka+6j_2}(\Gamma_i), \label{exact:-3_2} \\
    0\to H^0(ka\L-\frac{k}{2}\overline{B_2}-2(j_3+1)\overline{B_3})\to H^0(ka\L-\frac{k}{2}\overline{B_2}-j_3\overline{B_3})\to \bigoplus_{i=1}^{t_3}M_{ka+6j_3}(\Gamma_i), \label{exact:-3_3} \\
    0\to H^0(ka\L-\frac{k}{2}\overline{B_2}-\frac{2k}{3}\overline{B_3}-(j_6+1)\overline{B_6})\to H^0(ka\L-\frac{2k}{3}\overline{B_3}-j_6\overline{B_6})\to \bigoplus_{i=1}^{t _6}M_{ka+6j_6}(\Gamma_i). \label{exact:-3_6}
\end{align}

Iteration of (\ref{exact:-3_2}) for $j_2=0,\dots, k/6-1$, (\ref{exact:-3_3}) for $j_3=0,\dots,k/3-1$ and (\ref{exact:-3_6}) for $j_3=0,\dots,5k/6-1$ yields the desired inequality.
\end{proof}
\begin{rem}
We cannot evaluate $h^0(\M(a)-\Delta)$ directly, because we do not know how to construct cusp forms vanishing on cusps with high order.
\end{rem}

For $[\ell]\in \mathcal{R}_L(F)$, let 
\[\v(L,K_{\ell})\defeq\frac{\v(\U(K_{\ell}))}{\v(\U(L))}.\]

\begin{defn}
\label{def:volume}
For $F\neq\Q(\sqrt{-1}), \Q(\sqrt{-3})$, let 
\begin{align*}
    V(L,F)\defeq &
\displaystyle{\sum_{[\ell]\in R(F ,2)_I}}\v(L,K_{\ell})+2^n\displaystyle{\sum_{[\ell]\in \mathcal{R}_L(F, 2)_{III}, \mathcal{R}_L(F, 2)_{IV}, \mathcal{R}_L(F, 2)_{V}}}\v(L,K_{\ell})\\
&+4^n\displaystyle{\sum_{[\ell]\in \mathcal{R}_L(F, 2)_{II}}}\v(L,K_{\ell}).
\end{align*}
For $F=\Q(\sqrt{-1})$, let
\begin{align*}
    V(L,\Q(\sqrt{-1}))\defeq&\displaystyle {3\sum_{[\ell]\in \mathcal{R}_L(\Q(\sqrt{-1}) ,4)_I}}\v(L,K_{\ell})+3\cdot 2^n \displaystyle{\sum_{[\ell]\in \mathcal{R}_L(\Q(\sqrt{-1}), 4)_{II}}}\v(L,K_{\ell})\\
    &+4^n \displaystyle{\sum_{[\ell]\in \mathcal{R}_L(\Q(\sqrt{-1}), 2)_{II}}}\v(L,K_{\ell}).
\end{align*}
For $F=\Q(\sqrt{-3})$, let 
\begin{align*}
  V(L,\Q(\sqrt{-3}))\defeq&
  \displaystyle{5\sum_{[\ell]\in \mathcal{R}_L(\Q(\sqrt{-3}) ,6)}}\v(L,K_{\ell})+2\cdot 3^n \displaystyle{\sum_{[\ell]\in \mathcal{R}_L(\Q(\sqrt{-3}), 3)}}\v(L,K_{\ell})\\
  &+4^n \displaystyle{\sum_{[\ell]\in \mathcal{R}_L(\Q(\sqrt{-3}), 2)}}\v(L,K_{\ell}).
\end{align*}
\end{defn}

\begin{prop}
Let $a$ be a positive integer.
\label{thm:bigness_criterion}
\begin{enumerate}
    \item For $F\neq\Q(\sqrt{-1}), \Q(\sqrt{-3})$, $\M(a)=a\L-B_2/2$ is big if 
\begin{equation}
\label{thm:big_inequality_ow}
    V(L,F)<\left(1+\frac{1}{a}\right)^{1-n}\frac{2a}{n}.
\end{equation}
    \item For $F=\Q(\sqrt{-1})$, $\M(a)=a\L-\overline{B_2}/2-3\overline{B_4}/4$ is big if 
\begin{equation}
\label{thm:big_inequality_-1}
V(L,\Q(\sqrt{-1}))<\left(1+\frac{3}{a}\right)^{1-n}\frac{4a}{n}.
\end{equation}
    \item For $F=\Q(\sqrt{-3})$, $\M(a)=a\L-\overline{B_2}/2-2\overline{B_3}/3-5\overline{B_5}/6$ is big if 
\begin{equation}
\label{thm:big_inequality_-3}
V(L,\Q(\sqrt{-3}))<\left(1+\frac{5}{a}\right)^{1-n}\frac{6a}{n}.
\end{equation}
\end{enumerate}
\end{prop}

\begin{proof}
(1) We follow the strategy of \cite[Proposition 4.3]{Ma}.
We calculate the right side of the inequality of Lemma \ref{Lem:exact_sequence} (1) in terms of Proposition \ref{Prop:asymtotic_growth}.

First, we have
\[\dim M_{ka}(\U(L))=\frac{1}{n!}\v(\U(L))\cdot a^n\cdot k^n + O(k^{n-1}).\]

Second, we have
\begin{align*}
    &\sum_{i=1}^r\sum_{j=0}^{k/2-1}\dim M_{ka+2j}(\Gamma_i)\\
    &=\sum_{i=1}^r\sum_{j=0}^{k/2-1}\left\{\frac{1}{(n-1)!}\v(\Gamma_i)\cdot (ka+2j)^{n-1} + O(k^{n-2})\right\}\\
    &\leq \sum_{i=1}^r\frac{k}{2}\left\{\frac{1}{(n-1)!}\v(\Gamma_i)\cdot (a+1)^{n-1}\cdot k^{n-1} + O(k^{n-2})\right\}\\
    &=\frac{(a+1)^{n-1}}{2\cdot (n-1)!}\cdot\left(\sum_{i=1}^r\v(\Gamma_i)\right)\cdot k^n + O(k^{n-1}).
\end{align*}

Combining the above, we get  
\begin{align*}
    & h^0(k\M(a))\\
    &\geq \dim M_{ka}(\U(L)) - \sum_{i=1}^r\sum_{j=0}^{k/2-1}\dim M_{ka+2j}(\Gamma_i)\\
    &\geq \frac{a^n}{n!}\v(\U(L))\left\{1-\frac{n}{2a}(1+\frac{1}{a})^{n-1}\sum_{i=1}^r\frac{\v(\Gamma_i)}{\v(\U(L))}\right\}k^n + O(k^{n-1}).
\end{align*}
We need to estimate $\v(\Gamma_i)/\v(\U(L))$, in terms of $\v(L,K_{\ell})$ from Lemma \ref{Lem:stabilizer}.

\begin{align*}
     \frac{\v(\Gamma_i)}{\v(\U(L))}&=[\U(K_{\ell}):\Gamma_i]\v(L,K_{\ell})\\
&    \begin{cases}
    =\v(L,K_{\ell})&  ([\ell]\in \mathcal{R}_L(F,2)_{I}),\\
    \leq 4^n\v(L,K_{\ell})&  ([\ell]\in \mathcal{R}_L(F,2)_{II}),\\
    \leq 2^n\v(L,K_{\ell})&  ([\ell]\in \mathcal{R}_L(F,2)_{III}\coprod \mathcal{R}_L(F,2)_{IV}\coprod \mathcal{R}_L(F,2)_{V}).
    \end{cases}  
\end{align*}

Hence, since 
\[\mathcal{R}_L(F)=\mathcal{R}_L(F,2)=\mathcal{R}_L(F,2)_I\coprod \mathcal{R}_L(F,2)_{II}\coprod \mathcal{R}_L(F,2)_{III}\coprod \mathcal{R}_L(F,2)_{IV}\coprod \mathcal{R}_L(F,2)_{V},\]
the line bundle $\M(a)=a\L-B_2/2$ is big if
\begin{align*}
1-&\frac{n}{2a}\left(1+\frac{1}{a}\right)^{n-1}\cdot\Bigg\{\sum_{[\ell]\in \mathcal{R}_L(F,2)_I}\v(L,K_{\ell})\\
+& 4^n\sum_{[\ell]\in \mathcal{R}_L(F,2)_{II}}\v(L,K_{\ell}) + 2^n\sum_{[\ell]\in \mathcal{R}_L(F,2)_{III}\coprod \mathcal{R}_L(F,2)_{IV} \coprod \mathcal{R}_L(F,2)_{V}}\v(L,K_{\ell})\Bigg\}>0    
\end{align*}

holds.

(2) Here, we calculate the right side of the inequality of Lemma \ref{Lem:exact_sequence} (2). 
As in the above calculation, we have 
\begin{align*}
    &\sum_{i=1}^{s_2}\sum_{j_2=0}^{k/4-1}\dim M_{ka+4j_2}(\Gamma_i) + \sum_{i=1}^{s_4}\sum_{j_4=0}^{3k/4-1}\dim M_{ka+4j_4}(\Gamma_i)\\
    &\leq\frac{(a+3)^{n-1}}{4\cdot (n-1)!}\left\{\sum_{i=1}^{s_2}\v(\Gamma_i) + 3\sum_{i=1}^{s_4}\v(\Gamma_i) \right\} k^n + O(k^{n-1}).
\end{align*}

Then, 
\begin{align*}
    & h^0(k\M(a))\\
    &\geq \frac{a^n}{n!}\v(\U(L))\left[1-\frac{n}{4a}\left(1+\frac{3}{a}\right)^{n-1}\left\{\sum_{i=1}^{s_2}\frac{\v(\Gamma_i)}{\v(\U(L))} + 3\sum_{i=1}^{s_4}\frac{\v(\Gamma_i)}{\v(\U(L))}\right\}\right]k^n + O(k^{n-1}).
\end{align*}
Moreover, we need to estimate $\v(\Gamma_i)/\v(\U(L))$ in terms of $\v(L,K_{\ell})$ from Lemma \ref{Lem:stabilizer}.

\begin{align*}
     \frac{\v(\Gamma_i)}{\v(\U(L))}&=[\U(K_{\ell}):\Gamma_i]\v(L,K_{\ell})\\
&    \begin{cases}
    =\v(L,K_{\ell})&  ([\ell]\in \mathcal{R}_L(\Q(\sqrt{-1}),4)_{I}),\\
    \leq 2^n\v(L,K_{\ell})&  ([\ell]\in \mathcal{R}_L(\Q(\sqrt{-1}),4)_{II}),\\
    \leq 4^n\v(L,K_{\ell})&  ([\ell]\in \mathcal{R}_L(\Q(\sqrt{-1}),2)).
    \end{cases}  
\end{align*}

Hence, since 
\[\mathcal{R}_L(\Q(\sqrt{-1}))=\mathcal{R}_L(\Q(\sqrt{-1}),2)\coprod \mathcal{R}_L(\Q(\sqrt{-1}),4)_I\coprod \mathcal{R}_L(\Q(\sqrt{-1}),4)_{II},\]
the line bundle $\M(a)=a\L-B_2/2-3B_4/4$ is big if
\begin{align*}
    1-&\frac{n}{4a}\left(1+\frac{3}{a}\right)^{n-1}\Bigg\{3\sum_{[\ell]\in \mathcal{R}_L(\Q(\sqrt{-1}),4)_I}\v(L,K_{\ell})\\
    +& 3\cdot 2^n\sum_{[\ell]\in \mathcal{R}_L(\Q(\sqrt{-1}),4)_{II}}\v(L,K_{\ell}) + 4^n\sum_{[\ell]\in \mathcal{R}_L(\Q(\sqrt{-1}),2)}\v(L,K_{\ell})\Bigg\}
    >0
\end{align*}
holds.

(3) Here, we calculate the right side of the inequality of Lemma \ref{Lem:exact_sequence} (3). 
As in the above calculation, we have 
\begin{align*}
    &\sum_{i=1}^{t_2}\sum_{j_2=0}^{k/6-1}\dim M_{ka+6j_2}(\Gamma_i) + \sum_{i=1}^{t_3}\sum_{j_3=0}^{k/3-1}\dim M_{ka+6j_3}(\Gamma_i) + \sum_{i=1}^{t_6}\sum_{j_6=0}^{5k/6-1}\dim M_{ka+6j_6}(\Gamma_i)\\
    &\leq\frac{(a+5)^{n-1}}{6\cdot (n-1)!}\left\{\sum_{i=1}^{t_2}\v(\Gamma_i) + 2\sum_{i=1}^{t_3}\v(\Gamma_i) + 5\sum_{i=1}^{t_6}\v(\Gamma_i)\right\} k^n + O(k^{n-1}).
\end{align*}

Then, 
\begin{align*}
    & h^0(k\M(a))\\
    &\geq \frac{a^n}{n!}\v(\U(L))\Bigg[1-\frac{n}{6a}\left(1+\frac{5}{a}\right)^{n-1}\Bigg\{\sum_{i=1}^{t_2}\frac{\v(\Gamma_i)}{\v(\U(L))} + 2\sum_{i=1}^{t_3}\frac{\v(\Gamma_i)}{\v(\U(L))}\\
    &+  5\sum_{i=1}^{t_6}\frac{\v(\Gamma_i)}{\v(\U(L))}\Bigg\}\Bigg]k^n + O(k^{n-1}).
\end{align*}
We need to estimate $\v(\Gamma_i)/\v(\U(L))$ in terms of $\v(L,K_{\ell})$ from Lemma \ref{Lem:stabilizer}.

\begin{align*}
     \frac{\v(\Gamma_i)}{\v(\U(L))}&=[\U(K_{\ell}):\Gamma_i]\v(L,K_{\ell})\\
&    \begin{cases}
    =\v(L,K_{\ell})&  ([\ell]\in \mathcal{R}_L(\Q(\sqrt{-3}),6)),\\
    \leq 3^n\v(L,K_{\ell})&  ([\ell]\in \mathcal{R}_L(\Q(\sqrt{-3}),3)),\\
    \leq 4^n\v(L,K_{\ell})&  ([\ell]\in \mathcal{R}_L(\Q(\sqrt{-3}),2)).
    \end{cases}
\end{align*}

Hence, since 
\[\mathcal{R}_L(\Q(\sqrt{-3}))=\mathcal{R}_L(\Q(\sqrt{-3}),2)\coprod \mathcal{R}_L(\Q(\sqrt{-3}),3)\coprod \mathcal{R}_L(\Q(\sqrt{-3}),6),\]
the line bundle $\M(a)=a\L-B_2/2-2B_3/3-5B_6/6$ is big if
\begin{align*}
    &1-\frac{n}{6a}\left(1+\frac{5}{a}\right)^{n-1}\Bigg\{5\sum_{[\ell]\in \mathcal{R}_L(\Q(\sqrt{-3}),6)}\v(L,K_{\ell})\\
    &+ 2\cdot 3^n\sum_{[\ell]\in \mathcal{R}_L(\Q(\sqrt{-3}),3)}\v(L,K_{\ell}) + 4^n\sum_{[\ell]\in \mathcal{R}_L(\Q(\sqrt{-3}),2)}\v(L,K_{\ell})\Bigg\}>0
\end{align*}
holds.
\end{proof}

Next, we estimate the cardinality of the sets of split vectors.
Let $\mathcal{R}_{\mathrm{split}}$ be the subset of $\mathcal{R}_L(F)$ consisting of the elements $[\ell]\in \mathcal{R}_L(F)$ satisfying $L=\ell\OO_F\oplus K_{\ell}$.
We divide up $\mathcal{R}_{\mathrm{split}}$ as 
\[\mathcal{R}_{\mathrm{split}}=\coprod_{w|D(L)}\mathcal{R}_{\mathrm{split}}(w).\]
As in \cite{Ma}, $\mathcal{R}_{\mathrm{split}}(w)$ is canonically identified with the set of isometry classes of Hermitian lattices $K$ such that $K\oplus\l -w\r\cong L$. 
By the cancellation theorem \cite[Theorem 10]{Wall}, if
\[\l -w\r\oplus K\cong\l -w\r\oplus K',\]
we have $K\cong K'$ because $K$ is indefinite of rank $\geq 3$.
Hence, the following holds.
\begin{prop}
\label{prop:cardinality}
\[|\mathcal{R}_{\mathrm{split}}(w)|\leq 1.\]
\end{prop}

\section{Prasad's volume formula}
\label{section:calculation}
We will apply Prasad's volume formula to compute $V(L,F)$.
In Subsection \ref{subsection:Prasad_preparation}, we introduce Prasad's volume formula and in Subsection \ref{subsection:local_Jordan} we show that unramified square-free lattices  satisfy ($\heartsuit$).
\subsection{Preparation}
\label{subsection:Prasad_preparation}
Below, let $v$ be a finite place.
Let $F_v$ be the completion of $F$ at $v$, $\OO_{F_v}$ be a maximal compact subring and $\p_v$ be a maximal ideal.
Let $\f_v\defeq\OO_{F_v}/\p_v$ and $q_v\defeq|\f_v|$.
If $v$ ramifies, let $\pi$ be a  uniformizer of $F_v$.
Otherwise, let $\pi$ be a uniformizer of $\Q_v$.
Prasad \cite[Theorem 3.7]{Prasad} proved the $S$-arithmetic volume formula of arithmetic subgroups.
We shall apply it to our special unitary groups.

Now, let us assume that the arithmetic subgroup $\SU(L)$ is principal with respect to the coherent parahoric family $\{\SU(L\otimes \Z_v)\}_v$ in the sense of \cite{Prasad}.
By the strong approximation theorem, it holds that 
\[\SU(L)=\SU(L\otimes \Q)\cap \prod_{v\nmid\infty}\SU(L\otimes \Z_v).\]
Also, from the proof of \cite[Proposition 2.6]{RC}, the closure of the image of $\SU(L)$ in $\SU(L\otimes\Q_v)$ is $\SU(L\otimes\Z_v)$.
Hence our assumption means that $\SU(L\otimes\Z_v)$ is a parahoric subgroup for all $v$.

By Prasad's volume formula, we obtain, for a Hermitian lattice $L$ satisfying $(\star)$,
\[\v(\SU(L))=
\begin{cases}
\displaystyle{D^{\frac{n(n+3)}{4}}\prod_{i=1}^n\frac{i!}{(2\pi)^{i+1}}\zeta(2)L(3)\zeta(4)\dots L(n+1)\prod_{v\nmid\infty}\lambda_v^L}& (2\mid n),\\
\displaystyle{D^{\frac{(n-1)(n+2)}{4}}\prod_{i=1}^n\frac{i!}{(2\pi)^{i+1}}\zeta(2)L(3)\zeta(4)\dots\zeta(n+1)\prod_{v\nmid\infty}\lambda_v^L}& (2\nmid n).
\end{cases}
\]
Here, the local factor $\lambda_v^L$ is defined as follows.
By assumption, $\SU(L\otimes\Z_v)$ is a parahoric subgroup, thus there exists the smooth integral model $\underline{H}$ in the sense of Bruhat-Tits \cite{Tits} up to unique isomorphism; see also \cite{BT}.
Hence, there exists a reduction map $\underline{H}(\OO_{F_v})\to\underline{H}(\f_v)$.
Let $M_v^L$ be the maximal reductive quotient $\mathcal{H}(\f_v)$.

From \cite[Subsection 2.4]{PY}, if $v$ is inert in $F$, then 
\[\lambda_v^L=q_v^{(\dim M_v^L-n)/2}\cdot|M_v^L|^{-1}\cdot\prod_{i=2}^{n+1}(q_v^i-(-1)^i).\]
If $v$ splits in $F$, then 
\[\lambda_v^L=q_v^{(\dim M_v^L-n)/2}\cdot|M_v^L|^{-1}\cdot\prod_{i=2}^{n+1}(q_v^i-1).\]
If $v$ ramifies in $F$, then 
\[\lambda_v^L=q_v^{(\dim M_v^L-[(n+1)/2])/2}\cdot|M_v^L|^{-1}\cdot\prod_{i=1}^{[n+1/2]}(q_v^{2i}-1).\]

\subsection{Local Jordan decomposition}
\label{subsection:local_Jordan}
For local Hermitian lattices, there exists the Jordan decomposition; see \cite[Corollary 4.3]{GY} or \cite[Section 4]{Jacobowitz}:

\[L\otimes_{\Z}\Z_v=\bigoplus^{k_v}_{j=1}L_{v,j}(\pi^j),\]
where $L_{v,j}$ is a unimodular lattice over $\OO_{F_v}=\OO_F\otimes\Z_v$ and $k_v$ is an integer.
The local Jordan decomposition is unique up to its type in the sense of \cite[Remark 2.3]{Cho_case1}.
Let $n_{v,j}\defeq\mathrm{rk}(L_{v,j})$, so $\sum_{j=1}^{k_v}n_{v,j}=n+1$ for all finite places $v$.
Let
\begin{align*}
\l\ell,\ell\r&=\prod_{v\nmid\infty}v^{\nu_v},\\
K_{\ell}\otimes_{\Z}\Z_v&=\bigoplus^{k_v}_{j=1}K_{\ell,v,j}(\pi^j)\quad (K_{\ell,v,j}:\mathrm{unimodular}).
\end{align*}
In this notation, it follows that
\begin{align*}
    K_{\ell,v,j}&=L_{v,j}\quad (j\neq \nu_v),\\
    \mathrm{rk}(K_{\ell,v,\nu_v}) &= n_{v,\nu_v}-1. 
\end{align*}

\begin{rem}
\label{rem:lattice_chain}
For a semisimple simply connected algebraic group over $\Q_v$, the stabilizer of a point in the affine Bruhat-Tits building is parahoric \cite[Proposition 4.6.2]{BT}, \cite[Subsection 3.5.2]{Tits}.
Hence, if a Hermitian lattice $L\otimes \Z_v$ over $\Z_v$ defines a point in the affine Bruhat-Tits building, then $\SU(L\otimes\Z_v)$ is a parahoric subgroup of $\SU(L\otimes\Q_v)$.
We can interpret a point in the affine Bruhat-Tits building as a lattice chain  \cite[THÉORÈME 2,12]{BT2}, \cite[Subsection 1.6]{Lemaire} for unitary groups if $v\neq 2$ or $F_2/\Q_2$ is unramified; see \cite[Subsection 2.2]{BT2} or \cite[Definition 1.5]{Lemaire}.
Note that the structure of the reduced building of a unitary group is the same as that of a special unitary group; see \cite[Subsection 1.6]{Lemaire}.
\end{rem}

Let us consider when a Hermitian lattice forms a lattice chain.
We say that a Hermitian lattice $L$ over $\OO_{F_v}$ is \textit{primitive} if there does not exist a Hermitian lattice $L'$ of the same rank as $L$ over $\OO_{F_v}$ and a positive integer $i$ satisfying $L=L'(\pi^i)$.
Below, up to scaling, we will mainly consider primitive Hermitian lattices.

\begin{lem}
\label{Lem:lattice_parahoric}
Let $K$ be a quadratic extension of $\Q_p$, or be $\Q_p\times\Q_p$.
Assume that $K$ is not a ramified quadratic extension of $\Q_2$.
Let $M$ be a primitive Hermitian lattice over $\OO_K$.
If the discriminant group $A_M$ of $M$ satisfies
\[A_M\cong (\OO_K/\pi\OO_K)^k\]
for some non-negative integer $k$, then $\SU(M)$ is a parahoric subgroup of $\SU(M\otimes\Q_p)$.
Here, as before, $\pi$ is a uniformizer of $K$ if $K$ is a ramified extension, and $\pi=p$ if not.
\end{lem}
\begin{proof}
We denote by 
\[M=\bigoplus_{j=0}^t M_j(\pi^j)\quad (m_j\defeq\mathrm{rank}(M_j))\]
a Jordan decomposition of $M$ for some integer $t$.
First, we assume that $K$ is unramified over $\Q_p$ or equals $\Q_p\times\Q_p$.
Then, from \cite[Section 7]{Jacobowitz} or \cite[Proposition 4.2, Section 9]{GY}, it follows that
for some units $\delta_{j, i}\in\OO_K$.
In this situation, if $M$ satisfies 
\[M\subset \frac{1}{\pi}M^{\sharp}\subset\frac{1}{\pi}M,\]
then it defines a self-dual lattice chain; see \cite[Subsection 2.1]{Stevens}. 
Here 
\[M^{\sharp}\defeq\{v\in M\otimes\Q_p\mid \l v,w\r\in \pi\OO_K\ \mathrm{for\ any\ }w\in M\}.\]
This implies $0\leq j\leq 1$, that is, $M_j=0$ for $j>1$.
Therefore, if the Jordan decomposition of $M$ has the form 
\begin{equation}
\label{eq:lattice_chain}
    M=\bigoplus_{j=0}^1 M_j(\pi^j),
\end{equation}
then it defines a point in the affine Bruhat-Tits building.
Since the stabilizer of this lattice chain in $\SU(M\otimes\Q_p)$ is $\SU(M)$, from Remark \ref{rem:lattice_chain}, this finishes the proof for the  unramified or split cases.

Second, let us consider the case that $K$ is a ramified extension of $\Q_p$ with $p\neq 2$.
For odd $j$, from \cite[Proposition 8.1 (b)]{Jacobowitz} and invoking the same discussion as above, the condition 
\[M_j(\pi^j)\subset \frac{1}{\pi}\{M_j(\pi^j)\}^{\sharp}\subset\frac{1}{\pi}M_j(\pi^j)\]
implies $M_j=0$ for odd $j>1$.
Now, let $j$ be even.
Then, from \cite[Proposition 8.1 (a)]{Jacobowitz}, it follows that 
\[M_j(\pi^j)\cong\l \delta_{j, 1}\pi^{j/2}\r\oplus\cdots\oplus\l \delta_{j, m_j-1}\pi^{j/2}\r\oplus\l\delta_{j, m_j}\pi^{(j+1)/2}\r,\]
for some units $\delta_{j, i}\in\OO_K$.
Then, the condition 
\[M_j(\pi^j)\subset \frac{1}{\pi}\{M_j(\pi^j)\}^{\sharp}\subset\frac{1}{\pi}M_j(\pi^j)\]
implies $M_j=0$ for even $j>1$ because for such $j$, the last factor $\l\delta_{j, m_j}\pi^{(j+1)/2}\r$ does not satisfy the above inclusion relations.
Combining these computations completes the proof for the ramified case.
\end{proof}
\begin{rem}
We can prove the above when $K$ is a ramified extension over $\Q_2$ in a similar way as in  \cite[Section 9, 10, 11]{Jacobowitz} or \cite[Theorem 2.10]{Cho_case1}.
However, in this case, points in the building constitute a subset of the set of self-dual lattice chains \cite[Subsection 1.6]{Lemaire}, hence more detailed calculations seem to be needed.
For our purpose, it suffices to assume that $v=2$ is unramified in $F$ in the following examples because of the consideration of reflective vectors.
Hence, we will restrict Lemma \ref{Lem:lattice_parahoric} to this case, for simplicity. 
\end{rem}

Below, for a reflective vector $\ell\in L$, we use the same notation for the local Jordan decomposition of $L'\otimes\Z_v$ of a Hermitian lattice $L'\defeq\ell\OO_F\oplus K_{\ell}$ over $\OO_F$ as above.
First, we shall explain that unimodular lattices over $\OO_{F_0}$ satisfy $(\heartsuit)$.

\begin{prop}[Unimodular]
\label{ex:unimodular}
A unimodular Hermitian lattice $L$ of signature $(1,n)$ over $\OO_{F_0}$ satisfies $(\heartsuit)$.
\begin{proof}
Under the assumption on $F_0$, we can divide the set $\mathcal{R}_L(F_0,2)$ into the disjoint union according to their types:
\[\mathcal{R}_L(F_0,2)=\mathcal{R}_L(F_0,2)_I\coprod \mathcal{R}_L(F_0,2)_{II}\coprod \mathcal{R}_L(F_0,2)_{IV}\coprod \mathcal{R}_L(F_0,2)_{V};\]
see (\ref{eq:R_L(F)}) and discussion there.
For a reflective vector $[\ell]\in \mathcal{R}_L(F_0)=\mathcal{R}_L(F_0,2)$, let $L'\defeq\ell\OO_{F_0}\oplus K_{\ell}$, where  $K_{\ell}\defeq\ell^{\perp}\cap L$.
Then, we obtain 
\[L/L'\cong
\begin{cases}
1&([\ell]\in \mathcal{R}_L(F_0,2)_{I}),\\
\OO_F/2\OO_{F_0}&([\ell]\in \mathcal{R}_L(F_0,2)_{II}),\\
\OO_F/\p_i\OO_{F_0}&([\ell]\in \mathcal{R}_L(F_0,2)_{IV}\coprod \mathcal{R}_L(F,2)_{V}),
\end{cases}
\]
from the definition of reflective vectors.

If $[\ell]\in \mathcal{R}_L(F_0,2)_{I}$, then $K_{\ell}$ is also unimodular and local Jordan decompositions of $L$ and $K_{\ell}$ have the trivial forms
\begin{align*}
    L\otimes\Z_v&=L_{v,0},\\
    K_{\ell}\otimes\Z_v&=K_{\ell,v,0}.
\end{align*}

Now, consider the case of non-split vectors.
Let $\ell\in L$ be a non-split vector, i.e., $[\ell]\in \mathcal{R}_L(F_0,2)_{II}\coprod \mathcal{R}_L(F_0,2)_{IV}\coprod \mathcal{R}_L(F_0,2)_{V}$.
We refer to the proof of \cite[Lemma 2.2]{WW}.
Since $L$ is unimodular, $\sigma_{\ell,-1}\in\U(L)=\widetilde{\U}(L)$.
Hence, 
\[\frac{2\l v,\ell\r}{\l\ell,\ell\r}\in\OO_{F_0}\]
for any $v\in L=L^{\vee}$.
Since $\ell$ is primitive, it follows that  $\l\ell,\ell\r/2\not\in\OO_{F_0}\setminus\OO_{F_0}^{\times}$.
Hence if $[\ell]\in \mathcal{R}_L(F_0,2)_{II}$, then we have $\l\ell,\ell\r=-2$.
This means that, since $I_{\ell}=2\OO_{F_{0}}$, the discriminant groups of $L'=\ell\OO_{F_0}\oplus K_{\ell}$ and $K_{\ell}$ are 
\[A_{L'}\cong(\OO_{F_0}/2\OO_{F_0})^2,\ A_{K_{\ell}}\cong \OO_{F_0}/2\OO_{F_0}.\]
We conclude that the Jordan decompositions of $L'\otimes\Z_v$ and $K_{\ell}\otimes\Z_v$ are 
\begin{align*}
    L'\otimes\Z_v&= 
    \begin{cases}
         \displaystyle{\bigoplus_{j=0}^1} L'_{2,j}(\pi^j)& (v=2),\\
     L'_{v,0}& (\mathrm{otherwise}),\\
    \end{cases}\\
    K_{\ell}\otimes\Z_v &=
    \begin{cases}
        \displaystyle{\bigoplus_{j=0}^1} K_{\ell,2,j}(\pi^j) &(v=2),\\
    K_{\ell,v,0}&  (\mathrm{otherwise}),
    \end{cases}
\end{align*}
where 
\begin{align*}
\mathrm{rk}(L'_{2,0})&=n-1,\ \mathrm{rk}(L'_{2,1})=2,\\
\mathrm{rk}(K_{\ell,2,0})&=n-1,\ \mathrm{rk}(K_{\ell,2,1})=1.
\end{align*}
For $[\ell]\in \mathcal{R}_L(F_0,2)_{IV}$, from the same discussion as above, we have $\l\ell,\ell\r=-2$.
This means that, since $I_{\ell}=\p_1$,
\[A_{L'}\cong \OO_{F_0}/2\OO_{F_0},\]
and $K_{\ell}$ is unimodular.
We conclude that the Jordan decompositions of $L'$ and $K_{\ell}$ are the same as above except $v=2$.
For $v=2$, the local Jordan decompositions are 
\begin{align*}
    L'\otimes\Z_2 &= \bigoplus_{j=0}^1 L'_{2,j}(\pi^j),\\
    K_{\ell}\otimes\Z_2 &= K_{\ell,2,0},
\end{align*}
where 
\[\mathrm{rk}(L'_{v,0})=n,\ \mathrm{rk}(L'_{v,1})=1.\]

In all cases, for any $v$, the local Jordan decompositions of $L'=\ell\OO_F\oplus K_{\ell}$ and $K_{\ell}$ have the form (\ref{eq:lattice_chain}).
Hence, by Lemma \ref{Lem:lattice_parahoric}, $\SU(L'\otimes\Z_v)$ and $\SU((\ell^{\perp}\cap L)\otimes\Z_v)$ are parahoric for any $v$.
This implies $(\star)$ for $L'$ and $K_{\ell}$, and from the discussion in Subsection \ref{subsection:tools}, we conclude that $L$ satisfies $(\heartsuit)$.
\end{proof}
\end{prop}

Second, by generalizing the above proof, we prove that unramified square-free lattices over $\OO_{F_0}$ satisfy $(\heartsuit)$.
\begin{prop}[Unramified square-free]
\label{ex:square-free}
A primitive unramified square-free lattice $L$ over $\OO_{F_0}$ of signature $(1,n)$  satisfies $(\heartsuit)$.
\begin{proof}
Let $\det(L)=p_1\dots p_k$ be odd square-free.
Here, by the definition of unramified square-free lattices, any prime divisor $p_i$ is unramified in $F_0$.
For a split reflective vector $[\ell]\in \mathcal{R}_L(F_0)_I$, we denote by
\[\l\ell,\ell\r=\prod_{v\nmid\infty}v^{\nu_v}=\prod_{i=1}^{k'}p_i,\]
for some order and $k'\leq k$.
In other words, since the inner product $\l\ell.\ell\r$ divides $\det(L)$, the number of prime divisors of $\l\ell.\ell\r$ is $k'$ and we arrange prime elements $p_1,\cdots,p_k$ to satisfy the above.
Then the local Jordan decomposition of $L\otimes\Z_v$ is
\begin{align*}
    L\otimes\Z_v&=
    \begin{cases}
        \displaystyle{\bigoplus_{j=0}^1L_{p_i,j}}(\pi^j)& (v=p_i\ \mathrm{for}\ i=1,\cdots,k),\\
    L_{v,0}& (\mathrm{otherwise}),
    \end{cases}
\end{align*}
where
\[\mathrm{rk}(L_{p_i,0})=n,\  \mathrm{rk}(L_{p_i,1})=1,\]
for $i=1,\dots,k$
.We also have
\begin{align*}
    K_{\ell}\otimes\Z_v&=
    \begin{cases}
        \displaystyle{\bigoplus_{j=0}^1}K_{\ell,p_i,j}(\pi^j)& (v=p_i\ \mathrm{for}\ i=k'+1,\cdots,k),\\
    K_{\ell, v,0}& (\mathrm{otherwise}),
    \end{cases}
\end{align*}
where 
\[\mathrm{rk}(K_{\ell,p_i,0})=n-1,\  \mathrm{rk}(K_{\ell,p_i,1})=1,\]
for $i=k'+1,\cdots,k$, 
Now, We choose an element $e\in L$ so that  
\begin{align}
\label{eq:definition_e}
   A_L\cong\OO_{F_0}/p_1\dots p_k\OO_{F_0}=\big\l\frac{1}{p_1\dots p_k}e\big\r 
\end{align}
holds as $\OO_{F_0}$-modules.
If $[\ell]\in \mathcal{R}_L(F_0,2)_{II}$, first, we shall consider the case of $\sigma_{\ell,-1}\in\widetilde{\U}(L)$.
This occurs if and only if $\l e,\ell\r=0$.
In this situation, by the same discussion as Proposition \ref{ex:unimodular}, we have $\l\ell,\ell\r=-2$, and 
\[A_{L'}\cong(\OO_{F_0}/2\OO_{F_0})^2\times\OO_{F_0}/p_1\dots p_k\OO_{F_0},\ A_{K_{\ell}}\cong\OO_{F_0}/2p_1\dots p_k\OO_{F_0}.\]
We conclude that the Jordan decompositions of $L'\otimes\Z_v$ and $K_{\ell}\otimes\Z_v$ are 
\begin{align*}
    L'\otimes\Z_v &=
    \begin{cases}
         \displaystyle{\bigoplus_{j=0}^1} L'_{v,j}(\pi^j)& (v= 2,p_1,\cdots,p_k),\\
    L'_{v,0}& (\mathrm{otherwise}),\\
    \end{cases}\\
    K_{\ell}\otimes\Z_v &=
    \begin{cases}
         \displaystyle{\bigoplus_{j=0}^1} K_{\ell,v,j}(\pi^j)& (v= 2,p_1,\cdots,p_k),\\
    K_{\ell,v,0}& (\mathrm{otherwise}),
    \end{cases}
\end{align*}
where for $v= p_1,\cdots,p_k$,
\begin{align*}
\mathrm{rk}(L'_{2,0})&=n-1,\ \mathrm{rk}(L'_{2,1})=2,\\
\mathrm{rk}(L'_{v,0})&=n,\ \mathrm{rk}(L'_{v,1})=1\quad (v=p_1,\dots,p_k),\\
\mathrm{rk}(K_{\ell,v,0})&=n-1,\ \mathrm{rk}(K_{\ell,v,1})=1 \quad (v=2,p_1,\dots,p_k).
\end{align*}
Second, we consider the case of $\sigma_{\ell,-1}\not\in\widetilde{\U}(L)$, i.e., $\l e,\ell\r\neq 0$.
By the definition of the discriminant group and choice of the vector $e$ in (\ref{eq:definition_e}), the inner product $\l e,\ell\r$ is contained in $p_1\dots p_k\OO_{F_0}$.
In addition, since $\ell$ is primitive, it follows that $\l e,\ell\r=p_1\dots p_k$ by replacing $e$ with $-e$, if necessary.
On the other hand, since $2e\in L'=\ell\OO_{F_0}\oplus K_{\ell}$, we have 
\[2e=a\ell+bk_{\ell}\]
for some $a\neq 0,b\in\OO_{F_0}$ and $k_{\ell}\in K_{\ell}$.
Taking an inner product of both sides with $\ell$, we have
\[2\l e,\ell\r=2p_1\dots p_k=a\l\ell,\ell\r.\]
Now, the definition of $\mathcal{R}_L(F_0,2)_{II}$ implies that $2$ divides $\l\ell,\ell\r$, hence we have $\l e,\ell\r=2p_1\dots, p_{k'}$ for some integer $k'<k$, by changing the order of $p_1,\cdots p_k$, if necessary.
Then, this implies 
\begin{align*}
A_{L'}&\cong\OO_{F_0}/2p_1\dots p_{k'}\times\OO_{F_0}/2p_{k'+1}\dots p_k\OO_{F_0}\cong (\OO_{F_0}/2\OO_{F_0})^2\times\OO_{F_0}/p_1\dots p_k\OO_{F_0},\\
A_{K_{\ell}}&\cong \OO_{F_0}/2p_{k`+1}\dots p_{k}\OO_{F_0}.
\end{align*}
Hence, the Jordan decompositions of $L'\otimes\Z_v$ and $K_{\ell}\otimes\Z_v$ are 
\begin{align*}
    L'\otimes\Z_v &=
    \begin{cases}
    \displaystyle{\bigoplus_{j=0}^1} L'_{v,j}(\pi^j)& (v= 2,p_1,\cdots,p_k),\\
    L'_{v,0}& (\mathrm{otherwise}),\\
    \end{cases}\\
    K_{\ell}\otimes\Z_v &=
    \begin{cases}
     \displaystyle{\bigoplus_{j=0}^1} K_{\ell,v,j}(\pi^j)&(v= 2,p_{k'+1},\cdots,p_k),\\
     K_{\ell,v,0}& (\mathrm{otherwise}),
    \end{cases}
\end{align*}
where
\begin{align*}
\mathrm{rk}(L'_{2,0})&=n-1,\ \mathrm{rk}(L'_{2,1})=2,\\
\mathrm{rk}(K_{\ell,v,0})&=n-1,\ \mathrm{rk}(K_{\ell,v,1})=1\quad (v= 2,p_{k'+1},\cdots,p_k),\\
\mathrm{rk}(L'_{v,0})&=n,\ \mathrm{rk}(L'_{v,1})=1\quad (v\neq 2,p_{k'},\cdots,p_k).
\end{align*}

For $[\ell]\in \mathcal{R}_L(F_0,2)_{IV}\coprod \mathcal{R}_L(F_0,2)_{V}$, we can also calculate the local Jordan decompositions in the same way, and get 
\[A_{L'}\cong\OO_{F_0}/2p_1\dots p_k\OO_{F_0},\ A_{K_{\ell}}\cong\OO_{F_0}/p_1\dots p_k\OO_{F_0},\]
or
\[A_{L'}\cong\OO_{F_0}/2p_1\dots p_k\OO_{F_0},\ A_{K_{\ell}}\cong\OO_{F_0}/p_{k'+1}\dots p_k\OO_{F_0},\]
for some integer $k'$.

In all cases, for any $v$, the local Jordan decompositions have the form (\ref{eq:lattice_chain}).
Hence, by Lemma \ref{Lem:lattice_parahoric}, $\SU(L'\otimes\Z_v)$ and $\SU((\ell^{\perp}\cap L)\otimes\Z_v)$ are parahoric for any $v$.
As before, it follows that $L$ satisfies $(\heartsuit)$.
\end{proof}
\end{prop}

\section{Computation of local factors}
\label{section:comp_local_factors}
Historically, Tits \cite[Example 3.11]{Tits} calculated the maximal reductive quotients in the case of special unitary groups of odd dimension.
For unramified $v$ and ramified $v\neq2$ (resp. ramified $v=2$), Gan-Yu \cite{GY} (resp. Cho \cite{Cho_case1, Cho_case2}) determined the structure of the maximal reductive quotient.
On the other hand, Gan-Hanke-Yu \cite{GHY} classified the maximal reductive quotient in the case of maximal lattices. 
As \cite{Ma}, without loss of generality, up to scaling, we will mainly treat a primitive $L$.
In the following, we will omit the notion of $\f_v$-valued points and denote $M_v^{K_{\ell}}$ for $K_{\ell}$ as $M_v^L$.
\subsection{Unramified case}
\label{subsection:unram}
Gan-Yu clarified the structure of the maximal reductive quotient for unramified $v$.

\subsubsection{\rm{\textbf{Inert case}}}
\label{subsubsection:inert}
By \cite[Proposition 6.2.3]{GY}, according to local Jordan decompositions, the maximal reductive quotients of the mod $\p$ reductions of the smooth integral models of $\U(L\otimes\Z_v)$ and $\U(K_{\ell}\otimes\Z_v)$ are 
\[\U(n_{v,0})\times\dots\times\U(n_{v.\nu_v})\times\dots\times\U(n_{v,k_v})\]
and
\[\U(n_{v,0})\times\dots\times\U(n_{v,\nu_v}-1)\times\dots\times\U(n_{v,k_v}).\]
As in \cite[Introduction]{GY}, this also holds for $v=2$.
Hence, we have
\begin{align*}
    M_v^L&=\Ker(\det:\U(n_{v,0})\times\dots\times\U(n_{v.\nu_v})\times\dots\times\U(n_{v,k_v})\to \f_v^1),\\
    M_v^{K_{\ell}}&=\Ker(\det:\U(n_{v,0})\times\dots\times\U(n_{v.\nu_v}-1)\times\dots\times\U(n_{v,k_v})\to \f_v^1),
\end{align*}
where $\f_v^1$ denotes the set consisting of the elements of $\f_v$ whose norm is 1.
Note that these maps are surjective.
This implies
\begin{align*}
    \frac{|M_v^L|}{|M_v^{K_{\ell}}|}&=\frac{|\U(n_{v,0})|\times\dots\times|\U(n_{v.\nu_v})|\times\dots\times|\U(n_{v,k_v})|}{|\U(n_{v,0})|\times\dots\times|\U(n_{v.\nu_v}-1)|\times\dots\times|\U(n_{v,k_v})|}\\
    &=q_v^{n_{v,\nu_v}-1}(q_v^{n_{v,\nu_v}}-(-1)^{n_{v,\nu_v}})
\end{align*}
and
\[\dim M_v^L-\dim M_v^{K_{\ell}}=n_{v.\nu_v}^2-(n_{v.\nu_v}-1)^2=2n_{v.\nu_v}-1.\]

Then, 
\begin{align}
    \frac{\lambda_v^{K_{\ell}}}{\lambda_v^L}&=\left\{q_v^{(\dim M_v^{K_{\ell}}-n+1)/2}\cdot|M_v^{K_{\ell}}|^{-1}\cdot\prod_{i=2}^{n}(q_v^i-(-1)^i)\right\}\left\{q_v^{(\dim M_v^L-n)/2}\cdot|M_v^L|^{-1}\cdot\prod_{i=2}^{n+1}(q_v^i-(-1)^i)\right\}^{-1}\notag\\
    &=\frac{q_v^{n_{v,\nu_v}}-(-1)^{n_{v,\nu_v}}}{q_v^{n+1}-(-1)^{n+1}}.\notag
\end{align}

\subsubsection{\rm{\textbf{Split case}}}
\label{subsubsection:split}
As in Subsubsection \ref{subsubsection:inert},  by \cite[Proposition 6.2.3]{GY}, the maximal reductive quotients of the mod $\p$ reductions of the smooth integral models of $\U(L\otimes\Z_v)$ and $\U(K_{\ell}\otimes\Z_v)$ are 

\[    \GL(n_{v,0})\times\dots\times\GL(n_{v.\nu_v})\times\dots\times\GL(n_{v,k_v})\]
and
\[     \GL(n_{v,0})\times\dots\times\GL(n_{v,\nu_v}-1)\times\dots\times\GL(n_{v,k_v}).\]
As in \cite[Introduction]{GY}, this also holds for $v=2$.
Hence, we have surjective maps 
\begin{align*}
    M_v^L&=\Ker(\det:\GL(n_{v,0})\times\dots\times\GL(n_{v.\nu_v})\times\dots\times\GL(n_{v,k_v})\to \f_v^1),\\
    M_v^{K_{\ell}}&=\Ker(\det:\GL(n_{v,0})\times\dots\times\GL(n_{v.\nu_v}-1)\times\dots\times\GL(n_{v,k_v})\to \f_v^1).
\end{align*}
This implies 
\begin{align*}
    \frac{|M_v^L|}{|M_v^{K_{\ell}}|}&=\frac{|\GL(n_{v,0})|\times\dots\times|\GL(n_{v.\nu_v})|\times\dots\times|\GL(n_{v,k_v})|}{|\GL(n_{v,0})|\times\dots\times|\GL(n_{v.\nu_v}-1)|\times\dots\times|\GL(n_{v,k_v})|}\\
    &=q_v^{n_{v,\nu_v}-1}(q_v^{n_{v,\nu_v}}-1)
\end{align*}
and
\[\dim M_v^L-\dim M_v^{K_{\ell}}=n_{v,\nu_v}^2-(n_{v,\nu_v}-1)^2=2n_{v,\nu_v}-1.\]

Then, 
\begin{align}
    \frac{\lambda_v^{K_{\ell}}}{\lambda_v^L}&=\left\{q_v^{(\dim M_v^{K_{\ell}}-n+1)/2}\cdot|M_v^{K_{\ell}}|^{-1}\cdot\prod_{i=2}^{n}(q_v^i-1)\right\}\left\{q_v^{(\dim M_v^L-n)/2}\cdot|M_v^L|^{-1}\cdot\prod_{i=2}^{n+1}(q_v^i-1)\right\}^{-1}\notag\\
    &=\frac{q_v^{n_{v,\nu_v}}-1}{q_v^{n+1}-1}.\notag
\end{align}

\subsection{Ramified case: $v\neq 2$}
\label{subsection:ram_ow}
Fix a ramified prime $v\neq 2$. 
Recall the classification of the maximal reductive quotient of the reduction of the integral model by Gan-Yu \cite{GY}.
For a positive integer $x$, let 
\[\{x\}\defeq
\begin{cases}
x& (x: \mathrm{even}),\\
x-1& (x: \mathrm{odd}).\\
\end{cases}
\]

Let
\[H(n_{v,i})\defeq\begin{cases}
\O(n_{v,i})\ \mathrm{or}\ ^{2}\O(n_{v,i})& (i:\mathrm{even}),\\
\Sp(\{n_{v,i}\})& (i:\mathrm{odd}).
\end{cases}\]
Here, $^{2}\O(i)$ denotes the quasi-split but nonsplit special orthogonal group if $i$ is even.
Note that $\O(i)=$$^{2}\O(i)$ is split if $i$ is odd.

Accordingly, we obtain the following description of the maximal reductive quotients of the mod $\p$ reduction of the smooth integral models of $\U(L\otimes\Z_v)$ and $\U(K_{\ell}\otimes\Z_v)$ from \cite[Proposition 6.3.9]{GY};
\[
    H(n_{v,0})\times\dots\times H(n_{v,\nu_v})\times\dots\times H(n_{v,k_v})\]
    and
        \[H(n_{v,0})\times\dots\times H(n_{v,\nu_v}-1)\times\dots\times H(n_{v,k_v}).\]

If $(\nu_v,n_{v,\nu_v})=(\mathrm{even},\mathrm{even})$, then
\begin{align*}
    M_v^L&=\Ker(\det: H(n_{v,0})\times\dots\times\Sp(n_{v,\nu_v})\times\dots\times H(n_{v,k_v})\to\f_v^1),\\
    M_v^{K_{\ell}}&=\Ker(\det:H(n_{v,0})\times\dots\times\Sp(n_{v,\nu_v}-2)\times\dots\times H(n_{v,k_v})\to\f_v^1).
\end{align*}
This implies 
\begin{align*}
  \frac{|M_v^L|}{|M_v^{K_{\ell}}|}&\leq\frac{|\Sp(n_{v,\nu_v})|}{|\Sp(n_{v,\nu_v}-2)|}\\
  &=q_v^{n_{v,\nu_v}-1}(q_v^{n_{v,\nu_v}}-1).
\end{align*}
Note that this is an inequality, unlike the previous case.
This is to treat all cases independently of choices of $H(n_{\nu_i})$ appearing in the case of even $n_{\nu_i}$, which will be also observed below.
Moreover, we have 
\[\dim M_v^L-\dim M_v^{K_{\ell}}= \frac{n_{v,\nu_v}(n_{v,\nu_v}+1)}{2}-\frac{(n_{v,\nu_v}-1)(n_{v,\nu_v}-2)}{2}=2n_{v,\nu_v}-1.\]

Hence, if $n+1=2m+1$, then 
\begin{align}
    \frac{\lambda_v^{K_{\ell}}}{\lambda_v^L}&=\left\{q_v^{(\dim M_v^{K_{\ell}}-m)/2}\cdot|M_v^{K_{\ell}}|^{-1}\cdot\prod_{i=1}^{m}(q_v^{2i}-1)\right\}\left\{q_v^{(\dim M_v^L-m)/2}\cdot|M_v^L|^{-1}\cdot\prod_{i=1}^{m}(q_v^{2i}-1)\right\}^{-1}\notag\\
    &\leq q_v^{-1/2}(q_v^{n_{v,_nu_v}}-1).\notag
\end{align}

If $n+1=2m$, then 
\begin{align}
    \frac{\lambda_v^{K_{\ell}}}{\lambda_v^L}&=\left\{q_v^{(\dim M_v^{K_{\ell}}-m+1)/2}\cdot|M_v^{K_{\ell}}|^{-1}\cdot\prod_{i=1}^{m-1}(q_v^{2i}-1)\right\}\left\{q_v^{(\dim M_v^L-m)/2}\cdot|M_v^L|^{-1}\cdot\prod_{i=1}^{m}(q_v^{2i}-1)\right\}^{-1}\notag\\
    &\leq \frac{q_v^{n_{v,_nu_v}}-1}{q_v^{n+1}-1}.\notag
\end{align}

If $(\nu_p,n_{p,\nu_p})=(\mathrm{even},\mathrm{odd})$, then
\begin{align*}
    M_v^L&=\Ker(\det: H(n_{v,0})\times\dots\times\Sp(n_{v,\nu_v}-1)\times\dots\times H(n_{v,k_v})\to\f_v^1),\\
    M_v^{K_{\ell}}&=\Ker(\det:H(n_{v,0})\times\dots\times\Sp(n_{v,\nu_v}-1)\times\dots\times H(n_{v,k_v})\to\f_v^1).
\end{align*}
Hence, we have $M_v^{L}=M_v^{K_{\ell}}$,
If $n+1=2m+1$, then 
\begin{align}
    \frac{\lambda_v^{K_{\ell}}}{\lambda_v^L}&=\left\{q_v^{(\dim M_v^{K_{\ell}}-m)/2}\cdot|M_v^{K_{\ell}}|^{-1}\cdot\prod_{i=1}^{m}(q_v^{2i}-1)\right\}\left\{q_v^{(\dim M_v^L-m)/2}\cdot|M_v^L|^{-1}\cdot\prod_{i=1}^{m}(q_v^{2i}-1)\right\}^{-1}\notag\\
    &=1.\notag
\end{align}

If $n+1=2m$, then 
\begin{align}
    \frac{\lambda_v^{K_{\ell}}}{\lambda_v^L}&=\left\{q_v^{(\dim M_v^{K_{\ell}}-m+1)/2}\cdot|M_v^{K_{\ell}}|^{-1}\cdot\prod_{i=1}^{m-1}(q_v^{2i}-1)\right\}\left\{q_v^{(\dim M_v^L-m)/2}\cdot|M_v^L|^{-1}\cdot\prod_{i=1}^{m}(q_v^{2i}-1)\right\}^{-1}\notag\\
    &=\frac{q_v^{1/2}}{q_v^{n+1}-1}.\notag
\end{align}

If $(\nu_p,n_{p,\nu_p})=(\mathrm{odd},\mathrm{even})$, then

\begin{align*}
    M_v^L&=\Ker(\det: H(n_{v,0})\times\dots\times{}^{(2)}\O(n_{v,\nu_v})\times\dots\times H(n_{v,k_v})\to\f_v^1),\\
    M_v^{K_{\ell}}&=\Ker(\det:H(n_{v,0})\times\dots\times\O(n_{v,\nu_v}-1)\times\dots\times H(n_{v,k_v})\to\f_v^1).
\end{align*}
Here, $^{(2)}\O(n_{v,\nu_v})$ denotes $\O(n_{v,\nu_v})$ or $^{2}\O(n_{v,\nu_v})$.
Hence, 
\begin{align*}
  \frac{|M_v^L|}{|M_v^{K_{\ell}}|}&\leq\frac{|^{(2)}\O(n_{v,\nu_v})|}{|\O(n_{v,\nu_v}-1)|}\\
  &\leq q_v^{n_{v,\nu_v}/2-1}(q_v^{n_{v,\nu_v}/2}+1)
\end{align*}
and
\[\dim M_v^L-\dim M_v^{K_{\ell}}= \frac{n_{v,\nu_v}(n_{v,\nu_v}-1)}{2}-\frac{(n_{v,\nu_v}-1)(n_{v,\nu_v}-2)}{2}=n_{v,\nu_v}-1.\]

Hence, if $n+1=2m+1$, then 
\begin{align}
    \frac{\lambda_v^{K_{\ell}}}{\lambda_v^L}&=\left\{q_v^{(\dim M_v^{K_{\ell}}-m)/2}\cdot|M_v^{K_{\ell}}|^{-1}\prod_{i=1}^{m}(q_v^{2i}-1)\right\}\left\{q_v^{(\dim M_v^L-m)/2}\cdot|M_v^L|^{-1}\cdot\prod_{i=1}^{m}(q_v^{2i}-1)\right\}^{-1}\notag\\
    &\leq q_v^{-1/2}(q_v^{n_{v,\nu_v/2}}+1).\notag
\end{align}

If $n+1=2m$, then 
\begin{align}
    \frac{\lambda_v^{K_{\ell}}}{\lambda_v^L}&=\left\{q_v^{(\dim M_v^{K_{\ell}}-m+1)/2}\cdot|M_v^{K_{\ell}}|^{-1}\cdot\prod_{i=1}^{m-1}(q_v^{2i}-1)\right\}\left\{q_v^{(\dim M_v^L-m)/2}\cdot|M_v^L|^{-1}\cdot\prod_{i=1}^{m}(q_v^{2i}-1)\right\}^{-1}\notag\\
    &\leq \frac{q_v^{n_{v,\nu_v}/2}+1}{q_v^{n+1}-1}.\notag
\end{align}

If $(\nu_p,n_{p,\nu_p})=(\mathrm{odd},\mathrm{odd})$, then
\begin{align*}
    M_v^L&=\Ker(\det: H(n_{v,0})\times\dots\times\O(n_{v,\nu_v})\times\dots\times H(n_{v,k_v})\to\f_v^1),\\
    M_v^{K_{\ell}}&=\Ker(\det:H(n_{v,0})\times\dots\times{}^{(2)}\O(n_{v,\nu_v}-1)\times\dots\times H(n_{v,k_v})\to\f_v^1).
\end{align*}
This implies 
\begin{align*}
  \frac{|M_v^L|}{|M_v^{K_{\ell}}|}&\leq\frac{|\O(n_{v,\nu_v})|}{|^{(2)}\O(n_{v,\nu_v}-1)|}\\
  &\leq q_v^{(n_{v,\nu_v}-1)/2}(q_v^{(n_{v,\nu_v}-1)/2}+1).
\end{align*}
and
\[\dim M_v^L-\dim M_v^{K_{\ell}}= \frac{n_{v,\nu_v}(n_{v,\nu_v}-1)}{2}-\frac{(n_{v,\nu_v}-1)(n_{v,\nu_v}-2)}{2}=n_{v,\nu_v}-1.\]

Hence, if $n+1=2m+1$, then 
\begin{align}
    \frac{\lambda_v^{K_{\ell}}}{\lambda_v^L}&=\left\{q_v^{(\dim M_v^{K_{\ell}}-m)/2}\cdot|M_v^{K_{\ell}}|^{-1}\cdot\prod_{i=1}^{m}(q_v^{2i}-1)\right\}\left\{q_v^{(\dim M_v^L-m)/2}\cdot|M_v^L|^{-1}\cdot\prod_{i=1}^{m}(q_v^{2i}-1)\right\}^{-1}\notag\\
    &\leq q_v^{(n_{v,\nu_v}-1)/2}+1.\notag
\end{align}

If $n+1=2m$, then 
\begin{align}
    \frac{\lambda_v^{K_{\ell}}}{\lambda_v^L}&=\left\{q_v^{(\dim M_v^{K_{\ell}}-m+1)/2}\cdot|M_v^{K_{\ell}}|^{-1}\cdot\prod_{i=1}^{m-1}(q_v^{2i}-1)\right\}\left\{q_v^{(\dim M_v^L-m)/2}\cdot|M_v^L|^{-1}\cdot\prod_{i=1}^{m}(q_v^{2i}-1)\right\}^{-1}\notag\\
    &\leq q_v^{1/2}\cdot\frac{q_v^{(n_{v,\nu_v}-1)/2}+1}{q_v^{n+1}-1}.\notag
\end{align}

\subsection{Ramified case: $v=2$}
\label{subsection:ram_2}
Cho \cite{Cho_case1, Cho_case2} classified the maximal reductive quotient of the mod $\p$ reduction of the integral models for a ramified quadratic extension $F_2/\Q_2$.
He divided the problem into \textit{Case I} and \textit{Case II}, according to the structure of the lower ramification groups of the Galois group $\Gal(F_2/\Q_2)$; see \cite[Introduction]{Cho_case1}.
We also use his division.

\subsubsection{\rm{\textbf{Case I}}}
Let 
\[H_1^L(n_{2,i})\defeq
\begin{cases}
\Sp(\{n_{2,i}\})& (i:\mathrm{even\ and\ } L_{2,i}:\mathrm{type\ }II), \\
\Sp(\{n_{2,i}-1\})& (i:\mathrm{even\ and\ } L_{2,i}:\mathrm{type\ }I^o), \\
\Sp(\{n_{2,i}-2\})& (i:\mathrm{even\ and\ } L_{2,i}:\mathrm{type\ }I^e), \\
^{(2)}\O(n_{2,i})& (i:\mathrm{odd\ and\ } L_{2,i}:\mathrm{free}), \\
^{(2)}\SO(n_{2,i}+1)& (i:\mathrm{odd\ and\ } L_{2,i}:\mathrm{bounded}). \\
\end{cases}
\]
We define $H_1^{K_{\ell}}(n_{2,i})\defeq H_1^L(n_{2,i})$ if $i\neq\nu_2$ and 
\[H_1^{K_{\ell}}(n_{2,\nu_2}-1)\defeq
\begin{cases}
\Sp(\{n_{2,\nu_2}-1\})& (\nu_2:\mathrm{even\ and\ } K_{\ell,2,\nu_2}:\mathrm{type\ }II), \\
\Sp(\{n_{2,\nu_2}-2\})& (\nu_2:\mathrm{even\ and\ } K_{\ell,2,\nu_2}:\mathrm{type\ }I^o), \\
\Sp(\{n_{2,\nu_2}-3\})& (\nu_2:\mathrm{even\ and\ } K_{\ell,2,\nu_2}:\mathrm{type\ }I^e), \\
^{(2)}\O(n_{2,\nu_2}-1)& (\nu_2:\mathrm{odd\ and\ } K_{\ell,2,\nu_2}:\mathrm{free}), \\
{}^{(2)}\SO(n_{2,\nu_2})& (\nu_2:\mathrm{odd\ and\ } K_{\ell,2,\nu_2}:\mathrm{bounded}). \\
\end{cases}
\]
See \cite[Definition 2.1, Remark 2.6]{Cho_case1} for the definitions of the types of lattices.
We will not use these definitions here, except that the type $I^o$ (resp. $I^e$) means the rank is odd (resp. even) and evaluate the volume independently of the types of lattices.
Moreover, while Cho \cite[Remark 4.7]{Cho_case1} distinguishes between the cases where even-dimensional orthogonal groups are split or non-split, we will not use this description.
By \cite[Theorem 4.12]{Cho_case1}, we can determine the structure of the maximal reductive quotient of the mod $\p$ reduction of the smooth integral model of $\SU(L\otimes\Z_2)$ and $\SU(K_{\ell}\otimes\Z_2)$.
\begin{align*}
    M_2^L&=\Ker(\det: H_1^L(n_{2,0})\times\dots\times H_1^L(n_{2,\nu_2})\times\dots\times H_1^L(n_{2,k_2})\times (\Z/2\Z)^{\beta_L}\to\f_v^1),\\
    M_2^{K_{\ell}}&=\Ker(\det:H_1^{K_{\ell}}(n_{2,0})\times\dots\times H_1^{K_{\ell}}(n_{2,\nu_2}-1)\times\dots\times H_1^{K_{\ell}}(n_{2,k_2})\times(\Z/2\Z)^{\beta_{K_{\ell}}}\to\f_v^1).
\end{align*}

If $(\nu_2,n_{2,\nu_2})=(\mathrm{even},\mathrm{even})$, then
 $H_1^L(n_{2,\nu_2})=\Sp(n_{2,\nu_2})$ or $\Sp(n_{2,\nu_2}-2)$, and $H_1^{K_{\ell}}(n_{2,\nu_2}-1)=\O(n_{2,\nu_2}-1)\cong\Sp(n_{2,\nu_2}-2)$, according to the type of $L_{2,\nu_2}$.
The integers $\beta_L$ and $\beta_{K_{\ell}}$ are defined in \cite[Lemma 4.6]{Cho_case1} and  satisfy $\beta_L, \beta_{K_{\ell}}\leq n+1$ and $\beta_{L}\leq\beta_{K_{\ell}}+2$.
Since 
\begin{align*}
       \frac{|\Sp(n_{2,\nu_2})|}{2^{\dim \Sp(n_{2,\nu_2})/2}}&\geq\frac{|\Sp(n_{2,\nu_2}-2)|}{2^{\dim \Sp(n_{2,\nu_2}-2)/2}}\quad (n_{2,\nu_2}>2),\\
       \frac{|\Sp(2)|}{2^{\dim \Sp(2)/2}}&=3\cdot 2^{-1/2},
\end{align*}

we can bound the ratio of local factors independently of the type of the lattice:
\begin{align*}
    \frac{|M_2^L|}{2^{\dim M_2^L/2}}\cdot \frac{2^{\dim M_2^{K_{\ell}}/2}}{|M_2^{K_{\ell}}|}
    &\leq \frac{|\Sp(n_{2,\nu_2})|}{2^{\dim \Sp(n_{2,\nu_2})/2}}\cdot \frac{2^{\dim \Sp(n_{2,\nu_2}-2)/2}}{|\Sp(n_{2,\nu_2}-2)|}\cdot 2^{(\beta_L-\beta_{K_{\ell}})/2}\\
    &\leq 2^{\frac{1}{2}}(2^{n_{2,\nu_2}}-1)\quad (\mathrm{This\ also\ holds\ for\ }n_{2,\nu_2}=2).
\end{align*}
Hence, if $n+1=2m+1$, then 
\begin{align}
    \frac{\lambda_2^{K_{\ell}}}{\lambda_2^L}&=\left\{2^{(\dim M_2^{K_{\ell}}-m)/2}\cdot|M_2^{K_{\ell}}|^{-1}\cdot\prod_{i=1}^{m}(2^{2i}-1)\right\}\left\{2^{(\dim M_2^L-m)/2}\cdot|M_2^L|^{-1}\cdot\prod_{i=1}^{m}(2^{2i}-1)\right\}^{-1}\notag\\
    &\leq 2^{1/2}(2^{n_{2,\nu_2}}-1).\notag
\end{align}

If $n+1=2m$, then 
\begin{align}
    \frac{\lambda_2^{K_{\ell}}}{\lambda_2^L}&=\left\{2^{(\dim M_2^{K_{\ell}}-m+1)/2}\cdot|M_2^{K_{\ell}}|^{-1}\cdot\prod_{i=1}^{m-1}(2^{2i}-1)\right\}\left\{2^{(\dim M_2^L-m)/2}\cdot|M_2^L|^{-1}\cdot\prod_{i=1}^{m}(2^{2i}-1)\right\}^{-1}\notag\\
    &\leq 2\cdot\frac{2^{n_{2,\nu_2}}-1}{2^{n+1}-1}.\notag
\end{align}

If $(\nu_2,n_{2,\nu_2})=(\mathrm{even},\mathrm{odd})$, then
 $H^L_1(n_{2,\nu_2})=\O(n_{2,\nu_2})\cong\Sp(n_{2,\nu_2}-1)$, and $H^{K_{\ell}}_1(n_{2,\nu_2}-1)=\Sp(n_{2,\nu_2}-1)$ or $\Sp(n_{2,\nu_2}-3)$, according to the type of $K_{\ell,2,\nu_2}$.
Thus, we can bound the ratio of local factors independently of the type of the lattice:
\begin{align*}
    \frac{|M_2^L|}{2^{\dim M_2^L/2}}\cdot \frac{2^{\dim M_2^{K_{\ell}}/2}}{|M_2^{K_{\ell}}|}
    &\leq \frac{|\Sp(n_{2,\nu_2}-1)|}{2^{\dim \Sp(n_{2,\nu_2}-1)/2}}\cdot \frac{2^{\dim \Sp(n_{2,\nu_2}-3)/2}}{|\Sp(n_{2,\nu_2}-3)|}\cdot 2^{(\beta_L-\beta_{K_{\ell}})/2}\\
    &\leq 2^{\frac{1}{2}}(2^{n_{2,\nu_2}}-1).
\end{align*}

Hence, if $n+1=2m+1$, then 
\begin{align}
    \frac{\lambda_2^{K_{\ell}}}{\lambda_2^L}&=\left\{2^{(\dim M_2^{K_{\ell}}-m)/2}\cdot|M_2^{K_{\ell}}|^{-1}\cdot\prod_{i=1}^{m}(2^{2i}-1)\right\}\left\{2^{(\dim M_2^L-m)/2}\cdot|M_2^L|^{-1}\cdot\prod_{i=1}^{m}(2^{2i}-1)\right\}^{-1}\notag\\
    &\leq 2^{1/2}(2^{n_{2,\nu_2}}-1).\notag
\end{align}

If $n+1=2m$, then 
\begin{align}
    \frac{\lambda_2^{K_{\ell}}}{\lambda_2^L}&=\left\{2^{(\dim M_2^{K_{\ell}}-m+1)/2}\cdot|M_2^{K_{\ell}}|^{-1}\cdot\prod_{i=1}^{m-1}(2^{2i}-1)\right\}\left\{2^{(\dim M_2^L-m)/2}\cdot|M_2^L|^{-1}\cdot\prod_{i=1}^{m}(2^{2i}-1)\right\}^{-1}\notag\\
    &\leq 2\cdot\frac{2^{n_{2,\nu_2}}-1}{2^{n+1}-1}.\notag
\end{align}

If $(\nu_2,n_{2,\nu_2})=(\mathrm{odd},\mathrm{even})$, then
 $H_1^L(n_{2,\nu_2})={}^{(2)}\O(n_{2,\nu_2})$ or $\SO(n_{2,\nu_2}+1)$, and $H_1^{K_{\ell}}(n_{2,\nu_2}-1)=\O(n_{2,\nu_2}-1)$ or $^{(2)}\SO(n_{2,\nu_2})$, according to the type of $L_{2,\nu_2}$ and $K_{\ell,2,\nu_2}$.
Since 

\begin{align*}
    \frac{|\SO(n_{2,\nu_2}+1)|}{2^{\frac{\dim \SO(n_{2,\nu_2}+1)}{2}}}&\geq\frac{|^{2}\O(n_{2,\nu_2})|}{2^{\dim {}^{2}\O(n_{2,\nu_2})/2}}\geq\frac{|\O(n_{2,\nu_2})|}{2^{\dim \O(n_{2,\nu_2})/2}}\geq 1\quad (n_{2,\nu_2}> 2),\\
    3\cdot 2^{1/2}=\frac{|^{2}\O(2)|}{2^{\dim {}^{2}\O(2)/2}}&\geq\frac{|\SO(3)|}{2^{\dim \SO(3)/2}}\geq\frac{|\O(2)|}{2^{\dim \O(2)/2}},\\
        \frac{|{}^{(2)}\SO(n_{2,\nu_2})|}{2^{\dim {}^{(2)}\SO(n_{2,\nu_2})/2}}&\geq\frac{|\O(n_{2,\nu_2}-1)|}{2^{\dim \O(n_{2,\nu_2}-1)/2}},
\end{align*}
we can bound the ratio of local factors, independently of the type of the lattice:
\begin{align*}
    \frac{|M_2^L|}{2^{\dim M_2^L/2}}\cdot \frac{2^{\dim M_2^{K_{\ell}}/2}}{|M_2^{K_{\ell}}|}
    &\leq 3\cdot 2^{1/2}\cdot   \frac{|\SO(n_{2,\nu_2}+1)|}{2^{\dim \SO(n_{2,\nu_2}+1)/2}}\cdot \frac{2^{\dim \O(n_{2,\nu_2}-1)/2}}{|\O(n_{2,\nu_2}-1)|}\cdot 2^{(\beta_L-\beta_{K_{\ell}})/2}\\
    &\leq 3\cdot 2(2^{n_{2,\nu_2}}-1).
\end{align*}

Hence, if $n+1=2m+1$, then 
\begin{align}
    \frac{\lambda_2^{K_{\ell}}}{\lambda_2^L}&=\left\{2^{(\dim M_2^{K_{\ell}}-m)/2}\cdot|M_2^{K_{\ell}}|^{-1}\cdot\prod_{i=1}^{m}(2^{2i}-1)\right\}\left\{2^{(\dim M_2^L-m)/2}\cdot|M_2^L|^{-1}\cdot\prod_{i=1}^{m}(2^{2i}-1)\right\}^{-1}\notag\\
    &\leq 3\cdot 2(2^{n_{2,\nu_2}}-1).\notag
\end{align}

If $n+1=2m$, then 
\begin{align}
    \frac{\lambda_2^{K_{\ell}}}{\lambda_2^L}&=\left\{2^{(\dim M_2^{K_{\ell}}-m+1)/2}\cdot|M_2^{K_{\ell}}|^{-1}\cdot\prod_{i=1}^{m-1}(2^{2i}-1)\right\}\left\{2^{(\dim M_2^L-m)/2}\cdot|M_2^L|^{-1}\cdot\prod_{i=1}^{m}(2^{2i}-1)\right\}^{-1}\notag\\
    &\leq 3\cdot2^{3/2}\cdot \frac{2^{n_{2,\nu_2}}-1}{2^{n+1}-1}.\notag
\end{align}

If $(\nu_2,n_{2,\nu_2})=(\mathrm{odd},\mathrm{odd})$, then
 $H_1^L(n_{2,\nu_2})=\O(n_{2,\nu_2})$ or $^{(2)}\SO(n_{2,\nu_2}+1)$, and $H_1^{K_{\ell}}(n_{2,\nu_2}-1)=^{(2)}\O(n_{2,\nu_2}-1)$ or $\SO(n_{2,\nu_2})$, according to the type of $L_{2,\nu_2}$ and $K_{\ell,2,\nu_2}$.
We can bound the ratio of local factors, independently of the type of the lattice:
\begin{align*}
    \frac{|M_2^L|}{2^{\dim M_2^L/2}}\cdot \frac{2^{\dim M_2^{K_{\ell}}/2}}{|M_2^{K_{\ell}}|}
    &\leq \frac{|{}^{2}\SO(n_{2,\nu_2}+1)|}{2^{\dim {}^{2}\SO(n_{2,\nu_2}+1)/2}}\cdot \frac{2^{\dim\O(n_{2,\nu_2}-1)/2}}{|\O(n_{2,\nu_2}-1)|}\cdot 2^{(\beta_L-\beta_{K_{\ell}})/2}\\
    &\leq 2^{1/2}(2^{(n_{2,\nu_2}+1)/2}+1)(2^{(n_{2,\nu_2}-1)/2}+1).
\end{align*}

Hence, if $n+1=2m+1$, then 
\begin{align}
    \frac{\lambda_2^{K_{\ell}}}{\lambda_2^L}&=\left\{2^{(\dim M_2^{K_{\ell}}-m)/2}\cdot|M_2^{K_{\ell}}|^{-1}\cdot\prod_{i=1}^{m}(2^{2i}-1)\right\}\left\{2^{(\dim M_2^L-m)/2}\cdot|M_2^L|^{-1}\cdot\prod_{i=1}^{m}(2^{2i}-1)\right\}^{-1}\notag\\
    &\leq 2^{1/2}(2^{(n_{2,\nu_2}+1)/2}+1)(2^{(n_{2,\nu_2}-1)/2}+1).\notag
\end{align}

If $n+1=2m$, then 
\begin{align}
    \frac{\lambda_2^{K_{\ell}}}{\lambda_2^L}&=\left\{2^{(\dim M_2^{K_{\ell}}-m+1)/2}\cdot|M_2^{K_{\ell}}|^{-1}\cdot\prod_{i=1}^{m-1}(2^{2i}-1)\right\}\left\{2^{(\dim M_2^L-m)/2}\cdot|M_2^L|^{-1}\cdot\prod_{i=1}^{m}(2^{2i}-1)\right\}^{-1}\notag\\
    &\leq 2\cdot \frac{(2^{(n_{2,\nu_2}+1)/2}+1)(2^{(n_{2,\nu_2}-1)/2}+1)}{2^{n+1}-1}.\notag
\end{align}

\subsubsection{\rm{\textbf{Case II}}}
Let 
\[H_2^L(n_{2,i})\defeq
\begin{cases}
^{(2)}\O(n_{2,i})& (i:\mathrm{even\ and\ } L_{2,i}:\mathrm{type\ }II,\ \mathrm{free}), \\
^{(2)}\SO(n_{2,i}+1)& (i:\mathrm{even\ and\ } L_{2,i}:\mathrm{type\ }II,\ \mathrm{bounded}), \\
^{(2)}\SO(n_{2,i})& (i:\mathrm{even\ and\ } L_{2,i}:\mathrm{type\ }I^o), \\
^{(2)}\SO(n_{2,i}-1)& (i:\mathrm{even\ and\ } L_{2,i}:\mathrm{type\ }I^e), \\
\Sp(\{n_{2,i}\})& (i:\mathrm{odd\ and\ } L_{2,i}:\mathrm{type\ }II,\ \mathrm{or}\ \mathrm{type\ }I\ \mathrm{and\ bounded}), \\
\Sp(\{n_{2,i}-2\})& (i:\mathrm{odd\ and\ } L_{2,i}:\mathrm{type\ }I,\ \mathrm{free}). \\
\end{cases}
\]
We define $H_2^{K_{\ell}}(n_{2,i})\defeq H_2^L(n_{2,i})$ if $i\neq\nu_2$ and 
\[H_2^{K_{\ell}}(n_{2,\nu_2}-1)\defeq
\begin{cases}
^{(2)}\O(n_{2,\nu_2}-1)& (\nu_2:\mathrm{even\ and\ } K_{\ell,2,\nu_2}:\mathrm{type\ }II,\ \mathrm{free}), \\
^{(2)}\SO(n_{2,\nu_2})& (\nu_2:\mathrm{even\ and\ } K_{\ell,2,\nu_2}:\mathrm{type\ }II,\ \mathrm{bounded}), \\
^{(2)}\SO(n_{2,\nu_2}-1)& (\nu_2:\mathrm{even\ and\ } K_{\ell,2,\nu_2}:\mathrm{type\ }I^o), \\
^{(2)}\SO(n_{2,\nu_2}-2)& (\nu_2:\mathrm{even\ and\ } K_{\ell,2,\nu_2}:\mathrm{type\ }I^e), \\
\Sp(\{n_{2,\nu_2}-1\})& (\nu_2:\mathrm{odd\ and\ } K_{\ell,2,\nu_2}:\mathrm{type\ }II,\ \mathrm{or}\ \mathrm{type\ }I\ \mathrm{and\ bounded}), \\
\Sp(\{n_{2,\nu_2}-3\})& (\nu_2:\mathrm{odd\ and\ } K_{\ell,2,\nu_2}:\mathrm{type\ }I,\ \mathrm{free}). \\
\end{cases}
\]
Although Cho \cite[Remark 4.6]{Cho_case2} distinguishes cases in which the even-dimensional  orthogonal groups are split or non-split we will not use this description.
From \cite[Theorem 4.11]{Cho_case2}, we can determine the structure of the maximal reductive quotient of the mod $_p$ reduction of the smooth integral model of $\SU(L\otimes\Z_2)$ and $\SU(K_{\ell}\otimes\Z_2)$.
\begin{align*}
    M_2^L&=\Ker(\det: H_2^L(n_{2,0})\times\dots\times H_2^L(n_{2,\nu_2})\times\dots\times H_2^L(n_{2,k_2})\times (\Z/2\Z)^{\beta'_L}\to\f_v^1),\\
    M_2^{K_{\ell}}&=\Ker(\det:H_2^{K_{\ell}}(n_{2,0})\times\dots\times H_2^{K_{\ell}}(n_{2,\nu_2}-1)\times\dots\times H_2^{K_{\ell}}(n_{2,k_2})\times(\Z/2\Z)^{\beta'_{K_{\ell}}}\to\f_v^1).
\end{align*}
Here, $\beta'_L$ and $\beta'_{K_{\ell}}$ are integers defined in \cite[Lemma 4.5]{Cho_case2} and  satisfying $\beta'_L, \beta'_{K_{\ell}}\leq n+1$ and $\beta'_L\leq\beta'_{K_{\ell}}+4$.

Moreover, for later, we remark that  
\begin{align*}
           &1\leq\frac{|\SO(n_{2,\nu_2}-1)|}{2^{\SO(n_{2,\nu_2}-1)/2}}
           \leq\frac{|{}^{(2)}\O(n_{2,\nu_2})|}{2^{\dim {}^{(2)}\O(n_{2,\nu_2})/2}}
           \leq\frac{|\SO(n_{2,\nu_2}+1)|}{2^{\dim\SO(n_{2,\nu_2}+1)/2}}
           \leq\frac{|^{(2)}\SO(n_{2,\nu_2})|}{2^{\dim{}^{(2)}\SO(n_{2,\nu_2})/2}}
           \quad (n_{2,\nu_2}\neq 2:\mathrm{even}),\\
           &2^{-\frac{1}{2}}=\frac{|\SO(2)|}{2^{\dim \SO(2)/2}}
           \leq1=\frac{|\SO(1)|}{2^{\dim\SO(1)/2}}
           \leq\frac{|\O(2)|}{2^{\dim\O(2)/2}}
           \leq\frac{|\SO(3)|}{2^{\dim\SO(3)/2}}
           =\frac{|{}^2\SO(2)|}{2^{\dim{}^2\SO(2)/2}}
           \leq\frac{|{}^2\O(2)|}{2^{\dim{}^2\O(2)/2}}=2^{1/2}\cdot 3,\\
            &\frac{|{}^{(2)}\SO(n_{2,\nu_2}-1)|}{2^{\dim{}^{(2)}\SO(n_{2,\nu_2}-1)/2}}
            \leq\frac{|\O(n_{2,\nu_2})|}{2^{\dim\O(n_{2,\nu_2})/2}}
           =\frac{|\SO(n_{2,\nu_2})|}{2^{\dim\SO(n_{2,\nu_2})/2}}
           \leq\frac{|{}^{(2)}\SO(n_{2,\nu_2}+1)|}{2^{\dim{}^{(2)}\SO(n_{2,\nu_2}+1)/2}}.
           \quad (n_{2,\nu_2}\neq 1:\mathrm{odd}),\\
            &2^{-1/2}=\frac{|\SO(2)|}{2^{\dim\SO(2)/2}}
            \leq1=\frac{|\O(1)|}{2^{\dim\O(1)/2}}
           =\frac{|\SO(1)|}{2^{\dim\SO(1)/2}}
           \leq\frac{|{}^{2}\SO(2)|}{2^{\dim{}^{2}\SO(2)/2}}=3\cdot2^{-1/2}.
\end{align*}

If $(\nu_2,n_{2,\nu_2})=(\mathrm{even},\mathrm{even})$, then
 $H_2^L(n_{2,\nu_2})={}^{(2)}\O(n_{2,\nu_2})$, $\SO(n_{2,\nu_2}+1)$ or $\SO(n_{2,\nu_2}-1)$, and $H_2^{K_{\ell}}(n_{2,\nu_2}-1)=\O(n_{2,\nu_2}-1)$, $^{(2)}\SO(n_{2,\nu_2})$ or $\SO(n_{2,\nu_2}-1)$, according to the type of $L_{2,\nu_2}$ and $K_{\ell,2,\nu_2}$.
Thus, we can bound the ratio of local factors independently of the type of the lattice:
\begin{align*}
    \frac{|M_2^L|}{2^{\dim M_2^L/2}}\cdot \frac{2^{\dim M_2^{K_{\ell}}/2}}{|M_2^{K_{\ell}}|}
    &\leq \frac{|\SO(n_{2,\nu_2}+1)|}{2^{\dim \SO(n_{2,\nu_2}+1)/2}}\cdot2^{1/2}\cdot 3\frac{2^{\dim \SO(n_{2,\nu_2}-1)/2}}{|\SO(n_{2,\nu_2}-1)|}\cdot 2^{(\beta'_L-\beta'_{K_{\ell}})/2}\\
    &\leq 2^2\cdot 3(2^{n_{2,\nu_2}}-1)\quad (\mathrm{This\ also\ holds\ for\ }n_{2,\nu_2}=2).
\end{align*}

Hence, if $n+1=2m+1$, then 
\begin{align}
    \frac{\lambda_2^{K_{\ell}}}{\lambda_2^L}&=\left\{2^{(\dim M_2^{K_{\ell}}-m)/2}\cdot|M_2^{K_{\ell}}|^{-1}\cdot\prod_{i=1}^{m}(2^{2i}-1)\right\}\left\{2^{(\dim M_2^L-m)/2}\cdot|M_2^L|^{-1}\cdot\prod_{i=1}^{m}(2^{2i}-1)\right\}^{-1}\notag\\
    &\leq  2^2\cdot 3(2^{n_{2,\nu_2}}-1).\notag
\end{align}

If $n+1=2m$, then 
\begin{align}
    \frac{\lambda_2^{K_{\ell}}}{\lambda_2^L}&=\left\{2^{(\dim M_2^{K_{\ell}}-m+1)/2}\cdot|M_2^{K_{\ell}}|^{-1}\cdot\prod_{i=1}^{m-1}(2^{2i}-1)\right\}\left\{2^{(\dim M_2^L-m)/2}\cdot|M_2^L|^{-1}\cdot\prod_{i=1}^{m}(2^{2i}-1)\right\}^{-1}\notag\\
    &\leq 2^{5/2}\cdot 3\cdot\frac{2^{n_{2,\nu_2}}-1}{2^{n+1}-1}.\notag
\end{align}

If $(\nu_2,n_{2,\nu_2})=(\mathrm{even},\mathrm{odd})$, then
 $H_2^L(n_{2,\nu_2})=\O(n_{2,\nu_2})$, $^{(2)}\SO(n_{2,\nu_2}+1)$ or $\SO(n_{2,\nu_2})$, and $H_2^{K_{\ell}}(n_{2,\nu_2}-1)= {}^{(2)}\O(n_{2,\nu_2}-1)$, $\SO(n_{2,\nu_2})$ or $\SO(n_{2,\nu_2}-2)$, according to the type of $L_{2,\nu_2}$ and $K_{\ell,2,\nu_2}$.
Thus, we can bound the ratio of local factors independently of the type of the lattice:
\begin{align*}
    \frac{|M_2^L|}{2^{\dim M_2^L/2}}\cdot \frac{2^{\dim M_2^{K_{\ell}}/2}}{|M_2^{K_{\ell}}|}
    &\leq \frac{|{}^{2}\SO(n_{2,\nu_2}+1)|}{2^{\dim {}^{2}\SO(n_{2,\nu_2}+1)/2}}\cdot2^{\frac{1}{2}}\cdot \frac{2^{\dim \SO(n_{2,\nu_2}-1)/2}}{|\SO(n_{2,\nu_2}-1)|}\cdot 2^{(\beta'_L-\beta'_{K_{\ell}})/2}\\
    &\leq 2^{3/2}\cdot(2^{(n_{2,\nu_2}+1)/2}+1)(2^{(n_{2,\nu_2}-1)/2}+1)\quad (\mathrm{This\ also\ holds\ for\ }n_{2,\nu_2}=1).
\end{align*}

Hence, if $n+1=2m+1$, then 
\begin{align}
    \frac{\lambda_2^{K_{\ell}}}{\lambda_2^L}&=\left\{2^{(\dim M_2^{K_{\ell}}-m)/2}\cdot|M_2^{K_{\ell}}|^{-1}\cdot\prod_{i=1}^{m}(2^{2i}-1)\right\}\left\{2^{(\dim M_2^L-m)/2}\cdot|M_2^L|^{-1}\cdot\prod_{i=1}^{m}(2^{2i}-1)\right\}^{-1}\notag\\
    &\leq  2^{3/2}\cdot(2^{(n_{2,\nu_2}+1)/2}+1)(2^{(n_{2,\nu_2}-1)/2}+1).\notag
\end{align}

If $n+1=2m$, then 
\begin{align}
    \frac{\lambda_2^{K_{\ell}}}{\lambda_2^L}&=\left\{2^{(\dim M_2^{K_{\ell}}-m+1)/2}\cdot|M_2^{K_{\ell}}|^{-1}\cdot\prod_{i=1}^{m-1}(2^{2i}-1)\right\}\left\{2^{(\dim M_2^L-m)/2}\cdot|M_2^L|^{-1}\cdot\prod_{i=1}^{m}(2^{2i}-1)\right\}^{-1}\notag\\
    &\leq 2^2\cdot 3\cdot\frac{(2^{(n_{2,\nu_2}+1)/2}+1)(2^{(n_{2,\nu_2}-1)/2}+1)}{2^{n+1}-1}.\notag
\end{align}

If $(\nu_2,n_{2,\nu_2})=(\mathrm{odd},\mathrm{even})$, then
 $H_2^L(n_{2,\nu_2})=\Sp(n_{2,\nu_2})$ or $\Sp(n_{2,\nu_2}-2)$, and $H_2^{K_{\ell}}(n_{2,\nu_2}-1)=\Sp(n_{2,\nu_2}-2)$ or $\Sp(n_{2,\nu_2}-4)$, according to the type of $L_{2,\nu_2}$ and $K_{\ell,2,\nu_2}$.
Thus, we can bound the ratio of local factors independently of the type of the lattice:
\begin{align*}
    \frac{|M_2^L|}{2^{\dim M_2^L/2}}\cdot \frac{2^{\dim M_2^{K_{\ell}}/2}}{|M_2^{K_{\ell}}|}
    &\leq \frac{|\Sp(n_{2,\nu_2})|}{2^{\dim\Sp(n_{2,\nu_2})/2}}\cdot\frac{2^{\dim \Sp(n_{2,\nu_2}-4)/2}}{|\Sp(n_{2,\nu_2}-4)|}\cdot 2^{(\beta'_L-\beta'_{K_{\ell}})/2}\\
    &\leq 2^{3/2}\cdot(2^{n_{2,\nu_2}}-1)(2^{n_{2,\nu_2}-2}-1)\quad (\mathrm{This\ also\ holds\ for\ }n_{2,\nu_2}=2,4).
\end{align*}

Hence, if $n+1=2m+1$, then 
\begin{align}
    \frac{\lambda_2^{K_{\ell}}}{\lambda_2^L}&=\left\{2^{(\dim M_2^{K_{\ell}}-m)/2}\cdot|M_2^{K_{\ell}}|^{-1}\cdot\prod_{i=1}^{m}(2^{2i}-1)\right\}\left\{2^{(\dim M_2^L-m)/2}\cdot|M_2^L|^{-1}\cdot\prod_{i=1}^{m}(2^{2i}-1)\right\}^{-1}\notag\\
    &\leq  2^{3/2}\cdot(2^{n_{2,\nu_2}}-1)(2^{n_{2,\nu_2}-2}-1).\notag
\end{align}

If $n+1=2m$, then 
\begin{align}
    \frac{\lambda_2^{K_{\ell}}}{\lambda_2^L}&=\left\{2^{(\dim M_2^{K_{\ell}}-m+1)/2}\cdot|M_2^{K_{\ell}}|^{-1}\cdot\prod_{i=1}^{m-1}(2^{2i}-1)\right\}\left\{2^{(\dim M_2^L-m)/2}\cdot|M_2^L|^{-1}\cdot\prod_{i=1}^{m}(2^{2i}-1)\right\}^{-1}\notag\\
    &\leq 2^2\cdot\frac{(2^{n_{2,\nu_2}}-1)(2^{n_{2,\nu_2}-2}-1)}{2^{n+1}-1}.\notag
\end{align}

If $(\nu_2,n_{2,\nu_2})=(\mathrm{odd},\mathrm{odd})$, then
 $H_2^L(n_{2,\nu_2})=\Sp(n_{2,\nu_2}-1)$ or $\Sp(n_{2,\nu_2}-3)$, and $H_2^{K_{\ell}}(n_{2,\nu_2}-1)=\Sp(n_{2,\nu_2}-1)$ or $\Sp(n_{2,\nu_2}-3)$, according to the type of $L_{2,\nu_2}$ and $K_{\ell,2,\nu_2}$.
Thus, we can bound the ratio of local factors independently of the type of the lattice:
\begin{align*}
    \frac{|M_2^L|}{2^{\dim M_2^L/2}}\cdot \frac{2^{\dim M_2^{K_{\ell}}/2}}{|M_2^{K_{\ell}}|}
    &\leq \frac{|\Sp(n_{2,\nu_2}-1)|}{2^{\dim\Sp(n_{2,\nu_2}-1)/2}}\cdot\frac{2^{\dim \Sp(n_{2,\nu_2}-3)/2}}{|\Sp(n_{2,\nu_2}-3)|}\cdot 2^{(\beta'_L-\beta'_{K_{\ell}})/2}\\
    &\leq 2^{3/2}\cdot(2^{n_{2,\nu_2}}-1)\quad (\mathrm{This\ also\ holds\ for\ }n_{2,\nu_2}=1,3).
\end{align*}

Hence, if $n+1=2m+1$, then 
\begin{align}
    \frac{\lambda_2^{K_{\ell}}}{\lambda_2^L}&=\left\{2^{(\dim M_2^{K_{\ell}}-m)/2}\cdot|M_2^{K_{\ell}}|^{-1}\cdot\prod_{i=1}^{m}(2^{2i}-1)\right\}\left\{2^{(\dim M_2^L-m)/2}\cdot|M_2^L|^{-1}\cdot\prod_{i=1}^{m}(2^{2i}-1)\right\}^{-1}\notag\\
    &\leq  2^{3/2}\cdot(2^{n_{2,\nu_2}}-1).\notag
\end{align}

If $n+1=2m$, then 
\begin{align}
    \frac{\lambda_2^{K_{\ell}}}{\lambda_2^L}&=\left\{2^{(\dim M_2^{K_{\ell}}-m+1)/2}\cdot|M_2^{K_{\ell}}|^{-1}\cdot\prod_{i=1}^{m-1}(2^{2i}-1)\right\}\left\{2^{(\dim M_2^L-m)/2}\cdot|M_2^L|^{-1}\cdot\prod_{i=1}^{m}(2^{2i}-1)\right\}^{-1}\notag\\
    &\leq 2^2\cdot\frac{2^{n_{2,\nu_2}}-1}{2^{n+1}-1}.\notag
\end{align}

\section{Volume estimation}
\label{section:volume_estimation}
In this section, we will prove
\begin{align*}
V(L,F)\leq \frac{f_F^{odd}(m)}{\theta}\ \mathrm{or}\ \frac{f_F^{even}(m)}{\theta}
\end{align*}
according to whether $n+1$ is odd or even.
Consequently, this implies that $V(L,F)$ converges to 0 faster than the exponential function with respect to $m$.

To obtain the evaluation by different variables, let us introduce a condition $P(\a:L)$.
\begin{defn}
\label{def:P(M)}
Let $\a>0$ be a fixed positive integer.
We say that \textit{$L$ satisfies the condition $P(\a:L)$} if any prime divisor $p_i$ of $D(L)$ is unramified and the inequality $2(n+1-n_{p_i,\nu_{p_i}})\geq a_i/\a$ holds for any $p_i$ and any $[\ell]\in \mathcal{R}_{\mathrm{split}}$, where $a_i$ is defined by the exponent $D(L)=\prod p_i^{a_i}$.
\end{defn}

\subsection{Non-split vectors}
Here, we need to prepare some tools to treat the ``non-split case" as in \cite{Ma} for unitary groups.
For more details, see \cite[Subsection 6.2]{Ma}.

Let $[\ell]\in \mathcal{R}_L(F,i)$ be a non-split vector so that it defines the proper sublattice $L'\defeq\ell\OO_F\oplus K_{\ell}\subsetneq L$.
From Lemma \ref{Lem:classification_reflective_vectors_-1},  \ref{Lem:classification_reflective_vectors_-3},  \ref{Lem:classification_reflective_vectors_ow1},  \ref{Lem:classification_reflective_vectors_ow2} and  \ref{Lem:classification_reflective_vectors_ow3} , $[\ell]\in \mathcal{R}_L(F)\setminus \mathcal{R}_{\mathrm{split}}$ implies

\[[\ell]\in
    \begin{cases}
\mathcal{R}_L(F)\setminus \mathcal{R}_L(F,2)_I & (F\neq\Q(\sqrt{-1}), \Q(\sqrt{-3})), \\
\mathcal{R}_L(\Q(\sqrt{-1}))\setminus \mathcal{R}_L(\Q(\sqrt{-1}),4)_I & (F=\Q(\sqrt{-1})), \\
\mathcal{R}_L(\Q(\sqrt{-3}))\setminus \mathcal{R}_L(\Q(\sqrt{-3}),6) & (F=\Q(\sqrt{-3})).
\end{cases}
\]
We call these vectors \textit{non-split type} in accordance with \cite{Ma}.
Let  
\[\Gamma_{L'}\defeq\U(L)\cap\U(L')\]
in $\U(L\otimes_{\Z}\Q)$.

On the basis of the definition of $R(F,2)_{II}$, let 
\begin{align*}
    T_L(F,2)_{II}&\defeq\{L':\mathrm{sublattice\ of\ }L\mid L'=\OO_F\ell\oplus K_{\ell}\ \mathrm{for\ some\ }[\ell]\in \mathcal{R}_L(F,2)_{II}\},\\
    \mathcal{T}_L(F,2)_{II}&\defeq T_L(F,2)_{II}/\U(L).
\end{align*}
For $L'\in T_L(F,2)$, define
\begin{align*}
    R[L'](F,2)_{II}&\defeq\{\ell'\in L':\mathrm{primitive\ in\ }L'\mid L'=\OO_F\ell'\oplus (\ell'^{\perp}\cap L')\},\\
    \mathcal{R}[L'](F,2)_{II}&\defeq R[L'](F,2)_{II}/\U(L').
\end{align*}
In accordance with $\mathcal{R}_L(F,2)_{III}, \mathcal{R}_L(F,2)_{IV},\mathcal{R}_L(F,2)_{V},\mathcal{R}_L(\Q(\sqrt{-1}),4)_{II}, \mathcal{R}_L(\Q(\sqrt{-3}),3)$, for $\diamond\in\{2, 3, 6\}$ and $\ast\in\{II,III,IV,V\}$, define $T_L(F,\diamond)_{\ast}, \mathcal{T}_L(F,\diamond)_{\ast}, R[L'](F,\diamond)_{\ast}$, and $\mathcal{R}[L'](F,\diamond)_{\ast}$ as above.
Note that 
\[
 \mathcal{R}[L'](F,\diamond)_{\ast}=
\begin{cases}
\mathcal{R}_{L'}(F,2)_{I}& (F\neq\Q(\sqrt{-1}),\Q(\sqrt{-3})),\\
\mathcal{R}_{L'}(\Q(\sqrt{-1}),4)& (F=\Q(\sqrt{-1})),\\
\mathcal{R}_{L'}(\Q(\sqrt{-3}),6)& (F=\Q(\sqrt{-3})).\\
\end{cases}
\]

\begin{lem}[{\cite[Lemma 6.5]{Ma}}]
\label{Lem:non-split}
Fix $\diamond\in\{2, 3, 6\}$ and $\ast\in\{II,III,IV,V\}$.
Then for a possible pair $(\diamond,\ast)$ that makes sense with $\mathcal{R}_L(F,\diamond)_{\ast}$,
we obtain
\[\sum_{[\ell]\in \mathcal{R}_L(F,\diamond)_{\ast}}\v(L,K_{\ell})\leq \sum_{[L']\in\mathcal{T}_L(F,\diamond)_{\ast}}[\U(L):\Gamma_{L'}]\left(\sum_{[\ell]\in \mathcal{R}[L'](F, \diamond)_{\ast}}\v(L',K'_{\ell})\right).\]

\end{lem}

\begin{proof}
This can be proved in a similar way as \cite[Lemma 6.5]{Ma}.
We can embed $\mathcal{R}_L(F,\diamond)_{\ast}$ into the formal disjoint union
\[\coprod_{[L']\in\mathcal{T}_L(F,\diamond)_{\ast}} R[L']/\Gamma_{L'}. \]
Then, we have 
\begin{align*}
    \sum_{[\ell]\in \mathcal{R}_L(F,\diamond)_{\ast}}\v(L,K_{\ell})&=\sum_{[\ell]\in \mathcal{R}_L(F,\diamond)_{\ast}}\frac{[\U(L):\Gamma_{L'}]}{[\U(L'):\Gamma_{L'}]}\v(L',K_{\ell}) \\
    & \leq \sum_{[L']\in\mathcal{T}_L(F,\diamond)_{\ast}}\frac{[\U(L):\Gamma_{L'}]}{[\U(L'):\Gamma_{L'}]}\left(\sum_{[\ell]\in R[L'](F,\diamond)_{\ast}}\v(L',K'_{\ell})\right).
\end{align*}
 Since the number of elements of fibers of the projection $R[L'](F,\diamond)_{\ast}\to\mathcal{R}[L'](F,\diamond)_{\ast}$ is at most $[\U(L'):\Gamma_{L'}]$, we find that
\[\sum_{[\ell]\in R[L'](F,\diamond)_{\ast}}\v(L',K'_{\ell})\leq [\U(L'):\Gamma_{L'}]\cdot\sum_{[\ell]\in\mathcal{R}[L'](F,\diamond)_{\ast}}\v(L',K'_{\ell}).\]

\end{proof}

Now, $[\U(L'):\Gamma_{L'}]$ equals the cardinality of the $\U(L)$-orbit of $L'$ in $T_L(F,\diamond)_{\ast}$, hence 
\begin{align}
    &\sum_{[L']\in\mathcal{T}_L(F,\diamond)_{\ast}}[\U(L):\Gamma_{L'}]\notag\vspace{2pt}\\&=|T_L(F,\diamond)_{\ast}|<\begin{cases}
 2^{n+1} & ((F, \diamond,\ast)=(\mathrm{any}, 2,III/IV/V),  (\Q(\sqrt{-1}),4,II)),  \\ 
  3^{n+1} & ((F,\diamond,\ast)=(\Q(\sqrt{-3}),3,\emptyset)),  \\ 
   4^{n+1} & ((F,\diamond,\ast)=(\mathrm{any},2,II)). 
\end{cases}\label{non-split2}
\end{align}

Below, we bound the value 
\[\sum_{[\ell]\in \mathcal{R}[L'](F, \diamond)_{\ast}}\v(L',K'_{\ell})\]
independently of $L'$, $K'_{\ell}$ and $L$.
Note that $\mathcal{R}[L'](F, \diamond)_{\ast}$ is the set consisting of split reflective vectors of $L'$.

Let $\SU(L)$ be the subgroup of $\U(L)$ consisting of elements whose determinant is 1.
An easy calculation allows us to prove the following propositions.
\begin{prop}
\label{prop:u_to_su_ow}
Let $F\neq\Q(\sqrt{-1}), \Q(\sqrt{-3})$.
If $n$ is even, then 
\[\v(\U(L))=\v(\SU(L)).\]
If $n$ is odd, then 
\[\v(\SU(L))\leq\v(\U(L))\leq2\cdot\v(\SU(L)).\]
\end{prop}

\begin{prop}
\label{prop:u_to_su_-1}
Let $F=\Q(\sqrt{-1})$.
If $n$ is even, then 
\[\v(\U(L))=\v(\SU(L)).\]
Otherwise,
\[\v(\SU(L))\leq\v(\U(L))\leq
\begin{cases}
2\cdot\v(\SU(L))&(n\equiv 1\bmod4),\\
4\cdot\v(\SU(L))&(n\equiv 3\bmod4).\\
\end{cases}\]
\end{prop}

\begin{prop}
\label{prop:u_to_su_-3}
Let $F=\Q(\sqrt{-3})$.
If $n\equiv 0,4\bmod6$, then 
\[\v(\U(L))=\v(\SU(L)).\]
Otherwise, 
\[\v(\SU(L))\leq\v(\U(L))\leq
\begin{cases}
2\cdot\v(\SU(L))&(n\equiv 1,3\bmod6),\\
3\cdot\v(\SU(L))&(n\equiv 2\bmod6),\\
6\cdot\v(\SU(L))&(n\equiv 5\bmod6).
\end{cases}\]
\end{prop}

\subsection{Odd-dimensional\ case\ $\SU(1,2m)$}
\label{subsection:odd_dim}
Here, we consider the case of odd-dimensional unitary groups; i.e., we assume that $L$ is primitive of signature $(1,2m)$ with $m>1$. Let 
\begin{align*}
    \epsilon_{v,j}(1)&\defeq \frac{q_v^{j}-(-1)^{j}}{q_v^{2m+1}-(-1)^{2m+1}}\leq 1,\\
    \epsilon_{v,j}(2)&\defeq \frac{q_v^{j}-1}{q_v^{2m+1}-1}\leq 1,
\end{align*}
and 

\begin{align*}
    \epsilon_v(1)&\defeq \sum_{j,L_{v,j}\neq 0}\epsilon_{v,j}(1)\leq 1, \\
    \epsilon_v(2)&\defeq \sum_{j,L_{v,j}\neq 0}\epsilon_{v,j}(2)\leq 1. 
\end{align*}
Note that since $L$ is primitive, if $p$ does not divide $\det(L)$, then $n_{p,\nu_p}<2m+1$.
For $m>1$, from the computation of $\lambda_{K_{\ell}}/\lambda_L$ performed in Section \ref{section:comp_local_factors} with $n=2m$, it follows that 

\begin{align}
    &\sum_{[\ell]\in \mathcal{R}_{\mathrm{split}}}\frac{\v(\SU(K_{\ell}))}{\v(\SU(L))}\notag\\
    &\leq\frac{(2\pi)^{2m+1}}{D^{2m+1/2}\cdot (2m)!\cdot L(2m+1)}\notag\\
    &\cdot\sum_{[\ell]\in \mathcal{R}_{\mathrm{split}}}\left\{\prod_{v:\mathrm{inert}}\epsilon_{v,n_{v,\nu_v}}(1)\prod_{v:\mathrm{split}}\epsilon_{v,n_{v,\nu_v}}(2)\cdot 2\prod_{v\neq2:\mathrm{ram}}q_v^{n_{v,\nu_v}-1/2}\cdot\prod_{v=2:\mathrm{ram}}2^2\cdot3\cdot 2^{2n_2}\right\}\notag\\
    &\leq \frac{3\cdot2^4\cdot(2\pi)^{2m+1}}{\theta\cdot (2m)!\cdot L(2m+1)}\cdot\sum_{[\ell]\in \mathcal{R}_{\mathrm{split}}}\left\{\prod_{v:\mathrm{inert}}\epsilon_{v,n_{v,\nu_v}}(1)\prod_{v:\mathrm{split}}\epsilon_{v,n_{v,\nu_v}}(2)\right\}\notag\\
    &\leq \frac{3\cdot2^4\cdot(2\pi)^{2m+1}}{\theta\cdot (2m)!\cdot L(2m+1)}\cdot\sum_{J}\left\{\prod_{v|D(L):\mathrm{inert}}\epsilon_{v,j(v)}(1)\prod_{v|D(L):\mathrm{split}}\epsilon_{v,j(v)}(2)\right\}\quad (\because \mathrm{Proposition\ } \ref{prop:cardinality})\notag\\
    &= \frac{3\cdot2^4\cdot(2\pi)^{2m+1}}{\theta\cdot (2m)!\cdot L(2m+1)}\prod_{v|D(L):\mathrm{inert}}\epsilon_{v}(1)\prod_{v|D(L):\mathrm{split}}\epsilon_{v}(2)\notag\\
    &\leq \frac{3\cdot2^4\cdot(2\pi)^{2m+1}}{\theta\cdot (2m)!\cdot L(2m+1)}. \label{final:odd}
\end{align}
Here, $J=(j(v))_{v|D(L)}$ runs through multi-indices such that $L_{v,j(v)}\neq 0$ for every $v$; see \cite[Definition 5.7]{Ma}.

Besides, if $P(\a:L)$ holds, then we have
\begin{align}
    &\sum_{[\ell]\in \mathcal{R}_{\mathrm{split}}}\frac{\v(\SU(K_{\ell}))}{\v(\SU(L))}\notag\\
    &\leq\frac{3\cdot2^4\cdot(2\pi)^{2m+1}}{\theta\cdot (2m)!\cdot L(2m+1)}\prod_{v|D(L):\mathrm{inert}}\epsilon_{v}(1)\prod_{v|D(L):\mathrm{split}}\epsilon_{v}(2)\notag\\
    &\leq\frac{3\cdot2^4\cdot(2\pi)^{2m+1}}{\theta\cdot (2m)!\cdot L(2m+1)\cdot D(L)^{1/\a}}.
    \label{final:strong_o}
\end{align}
We apply these estimates to $V(L,F)$ in Proposition \ref{thm:bigness_criterion}.

\subsubsection{$F\neq\Q(\sqrt{-1}), \Q(\sqrt{-3})$ \rm{\textbf{case}}}
Let $F\neq\Q(\sqrt{-1}), \Q(\sqrt{-3})$.
From (\ref{final:odd}), we have 
\begin{align*}
    &V(L,F)\\
&\defeq
\displaystyle{\sum_{[\ell]\in R(F ,2)_I}}\v(L,K_{\ell})+2^{2m}\displaystyle{\sum_{[\ell]\in \mathcal{R}_L(F, 2)_{III}, \mathcal{R}_L(F, 2)_{IV}, \mathcal{R}_L(F, 2)_{V}}}\v(L,K_{\ell})\\
&+4^{2m}\displaystyle{\sum_{[\ell]\in \mathcal{R}_L(F, 2)_{II}}}\v(L,K_{\ell})\\
&\leq 2\cdot\displaystyle{\sum_{[\ell]\in R(F ,2)_I}}\frac{\v(\SU(K_{\ell}))}{\v(\SU(L))}+2\cdot2^{2m}\displaystyle{\sum_{[\ell]\in \mathcal{R}_L(F, 2)_{III}, \mathcal{R}_L(F, 2)_{IV}, \mathcal{R}_L(F, 2)_{V}}}\frac{\v(\SU(K_{\ell}))}{\v(\SU(L))}\\
&+2\cdot 4^{2m}\displaystyle{\sum_{[\ell]\in \mathcal{R}_L(F, 2)_{II}}}\frac{\v(\SU(K_{\ell}))}{\v(\SU(L))}\quad (\because \mathrm{Proposition\ } \ref{prop:u_to_su_ow})\\
&\leq 2(1+ 2^{2m}\cdot 2^{2m+1} + 4^{2m}\cdot 4^{2m+1})\cdot \frac{3\cdot2^4\cdot(2\pi)^{2m+1}}{\theta\cdot (2m)!\cdot L(2m+1)}\quad (\because (\ref{final:odd}))\\
&= (1+ 2^{4m+1} +  2^{8m+2})\cdot \frac{3\cdot2^5\cdot(2\pi)^{2m+1}}{\theta\cdot (2m)!\cdot L(2m+1)}.
\end{align*}

Moreover, if $P(\a:L)$ holds, we have
\begin{align}
\label{ams:o_ow}
&V(L,F)\leq (1+ 2^{4m+1} +  2^{8m+2})\cdot \frac{3\cdot2^5\cdot(2\pi)^{2m+1}}{\theta\cdot (2m)!\cdot L(2m+1)\cdot D(L)^{1/\a}}
\end{align}
by (\ref{final:strong_o}).
\subsubsection{$F=\Q(\sqrt{-1})$ \rm{\textbf{case}}}
Let $F=\Q(\sqrt{-1})$.
From (\ref{final:odd}), we have 
\begin{align*}
    &V(L,\Q(\sqrt{-1}))\\
    &\defeq\displaystyle {3\sum_{[\ell]\in \mathcal{R}_L(\Q(\sqrt{-1}) ,4)_I}}\v(L,K_{\ell})+3\cdot 2^{2m} \displaystyle{\sum_{[\ell]\in \mathcal{R}_L(\Q(\sqrt{-1}), 4)_{II}}}\v(L,K_{\ell})\\
&+4^{2m} \displaystyle{\sum_{[\ell]\in \mathcal{R}_L(\Q(\sqrt{-1}), 2)_{II}}}\v(L,K_{\ell})\\
&\leq 4\cdot\displaystyle {3\sum_{[\ell]\in \mathcal{R}_L(\Q(\sqrt{-1}) ,4)_I}}\frac{\v(\SU(K_{\ell}))}{\v(\SU(L))}+4\cdot3\cdot 2^{2m} \displaystyle{\sum_{[\ell]\in \mathcal{R}_L(\Q(\sqrt{-1}), 4)_{II}}}\frac{\v(\SU(K_{\ell}))}{\v(\SU(L))}\\
&+4\cdot4^{2m} \displaystyle{\sum_{[\ell]\in \mathcal{R}_L(\Q(\sqrt{-1}), 2)_{II}}}\frac{\v(\SU(K_{\ell}))}{\v(\SU(L))}\quad (\because \mathrm{Proposition\ }\ref{prop:u_to_su_-1})\\
&\leq 4(3+3\cdot 2^{2m}\cdot 2^{2m+1} + 4^{2m}\cdot 4^{2m+1})\cdot \frac{3\cdot2^4\cdot(2\pi)^{2m+1}}{\theta\cdot (2m)!\cdot L(2m+1)}\quad (\because (\ref{final:odd}))\\
&= (3+3\cdot 2^{4m+1} + 2^{8m+2})\cdot \frac{3\cdot2^6\cdot(2\pi)^{2m+1}}{\theta\cdot (2m)!\cdot L(2m+1)}.
\end{align*}

Moreover, if $P(\a:L)$ holds, we have
\begin{align}
\label{ams:o_-1}
&V(L,F)\leq(3+3\cdot 2^{4m+1} + 2^{8m+2})\cdot \frac{3\cdot2^6\cdot(2\pi)^{2m+1}}{\theta\cdot (2m)!\cdot L(2m+1)\cdot D(L)^{1/\a}}
\end{align}
by (\ref{final:strong_o}).

\subsubsection{$F=\Q(\sqrt{-3})$ \rm{\textbf{case}}}
Let $F=\Q(\sqrt{-3})$.
From (\ref{final:odd}), we have  
\begin{align*}
  &V(L,\Q(\sqrt{-3}))\\
  &\defeq
  \displaystyle{5\sum_{[\ell]\in \mathcal{R}_L(\Q(\sqrt{-3}) ,6)}}\v(L,K_{\ell})+2\cdot 3^{2m} \displaystyle{\sum_{[\ell]\in \mathcal{R}_L(\Q(\sqrt{-3}), 3)}}\v(L,K_{\ell})
\\
&+4^{2m} \displaystyle{\sum_{[\ell]\in \mathcal{R}_L(\Q(\sqrt{-3}), 2)}}\v(L,K_{\ell})\\
&\leq   6\cdot\displaystyle{5\sum_{[\ell]\in \mathcal{R}_L(\Q(\sqrt{-3}) ,6)}}\frac{\v(\SU(K_{\ell}))}{\v(\SU(L))}+6\cdot2\cdot 3^{2m} \displaystyle{\sum_{[\ell]\in \mathcal{R}_L(\Q(\sqrt{-3}), 3)}}\frac{\v(\SU(K_{\ell}))}{\v(\SU(L))}
\\
&+6\cdot4^{2m} \displaystyle{\sum_{[\ell]\in \mathcal{R}_L(\Q(\sqrt{-3}), 2)}}\frac{\v(\SU(K_{\ell}))}{\v(\SU(L))}\quad (\because \mathrm{Proposition\ } \ref{prop:u_to_su_-3})\\
&\leq 6(5+2\cdot 3^{2m}\cdot 3^{2m+1} + 4^{2m}\cdot 4^{2m+1})\cdot \frac{3\cdot2^4\cdot(2\pi)^{2m+1}}{\theta\cdot (2m)!\cdot L(2m+1)}\quad (\because (\ref{final:odd}))\\
&= (5+2\cdot 3^{4m+1} + 2^{8m+2})\cdot \frac{3^2\cdot2^5\cdot(2\pi)^{2m+1}}{\theta\cdot (2m)!\cdot L(2m+1)}.
\end{align*}

Moreover, if $P(\a:L)$ holds, we have

\begin{align}
\label{ams:o_-3}
&V(L,F)\leq (5+2\cdot 3^{4m+1} + 2^{8m+2})\cdot \frac{3^2\cdot2^5\cdot(2\pi)^{2m+1}}{\theta\cdot (2m)!\cdot L(2m+1)\cdot D(L)^{1/\a}}
\end{align}
by (\ref{final:strong_o}).

\subsubsection{\rm{\textbf{Summary:\ odd-dimensional\ case}}}
Combining all statements proven above, we can assert as follows.
\begin{thm}
\label{thm:volume_o}
Let $L$ be primitive of signature $(1,2m)$  with $m>1$.
Assume $(\heartsuit)$.
Then, if $m$ or $\theta$ is sufficiently large, the line bundle $\M(a)$ is big.
More precisely, 
\[V(L,F)\leq \frac{f_F^{odd}(m)}{\theta}.\]
Moreover, if $P(\a:L)$ holds for some $\a>0$, we have
\[V(L,F)\leq \frac{f_F^{odd}(m)}{D(L)^{1/\a}\cdot \theta}.\]
\end{thm}

\subsection{Even-dimensional\ case\ $\SU(1,2m-1)$}
\label{subsection:even_dim}
Let 
\begin{align*}
    \epsilon_{v,j}\defeq \frac{q_v^{j}-1}{q_v^{2m+1}-1}\leq 1,
\end{align*}
and 
\begin{align*}
    \epsilon_v\defeq \sum_{j,L_{v,j}\neq 0}\epsilon_{v,j}\leq 1.
\end{align*}

Now, let $L$ be primitive of signature $(1,2m)$ with $m>1$.
Note that since $L$ is primitive, if $p$ does not divide $\det(L)$, then $n_{p,\nu_p}<2m$.
For $m>1$, from the computation of $\lambda_{K_{\ell}}/\lambda_L$ performed in Section \ref{section:comp_local_factors} with $n=2m-1$, it follows that 
\begin{align}
    &\sum_{[\ell]\in \mathcal{R}_{\mathrm{split}}}\frac{\v(\SU(K_{\ell}))}{\v(\SU(L))}\notag\\
    &\leq\frac{(2\pi)^{2m}}{(2m-1)!\cdot\zeta(2m)}\cdot\sum_{[\ell]\in \mathcal{R}_{\mathrm{split}}}\left\{\prod_{v:\mathrm{unram}}\epsilon_{v,\nu_v}\prod_{v\neq2:\mathrm{ram}}\frac{q_v^{n_{v,\nu_v}}-1}{q_v^{2m}-1}\prod_{v=2:\mathrm{ram}}\frac{2^{5/2}\cdot3\cdot 2^{2n_2}}{2^{2m}-1}\right\}\notag\\
    &\leq \frac{2^{2m+5/2}\cdot 3\cdot (2\pi)^{2m}}{\theta\cdot (2m-1)!\cdot\zeta(2m)}\cdot\sum_{[\ell]\in \mathcal{R}_{\mathrm{split}}}\prod_{v|D(L):\mathrm{unram}}\epsilon_{v,\nu_v}\notag\\
    &\leq \frac{2^{2m+5/2}\cdot 3\cdot (2\pi)^{2m}}{\theta\cdot (2m-1)!\cdot\zeta(2m)}\cdot\sum_{J}\prod_{v|D(L):\mathrm{unram}}\epsilon_{v,j(v)}\quad (\because \mathrm{Proposition\ } \ref{prop:cardinality})\notag\\
    &= \frac{2^{2m+5/2}\cdot 3\cdot (2\pi)^{2m}}{\theta\cdot (2m-1)!\cdot\zeta(2m)}\prod_{v|D(L):\mathrm{unram}}\epsilon_{v}\notag\\
    &\leq \frac{2^{2m+5/2}\cdot 3\cdot (2\pi)^{2m}}{\theta\cdot (2m-1)!\cdot\zeta(2m)}.\label{final:even}
\end{align}
More strongly, if $P(\a:L)$ holds, we have
\begin{align}
    &\sum_{[\ell]\in \mathcal{R}_{\mathrm{split}}}\frac{\v(\SU(K_{\ell}))}{\v(\SU(L))}\notag\\
    &\leq\frac{2^{2m+5/2}\cdot 3\cdot (2\pi)^{2m}}{\theta\cdot (2m-1)!\cdot\zeta(2m)}\prod_{v|D(L):\mathrm{unram}}\epsilon_{v}\notag\\
    &\leq\frac{2^{2m+5/2}\cdot 3\cdot (2\pi)^{2m}}{\theta\cdot (2m-1)!\cdot\zeta(2m)\cdot D(L)^{1/\a}}.
    \label{final:strong_e}
\end{align}
Below, we apply these estimates to $V(L,F)$ in Proposition \ref{thm:bigness_criterion}.

\subsubsection{$F\neq\Q(\sqrt{-1}), \Q(\sqrt{-3})$ \rm{\textbf{case}}}
Let $F\neq\Q(\sqrt{-1}), \Q(\sqrt{-3})$.
From (\ref{final:even}), we have
\begin{align*}
    &V(L,F)\\
&\defeq
\displaystyle{\sum_{[\ell]\in R(F ,2)_I}}\v(L,K_{\ell})+2^{2m-1}\displaystyle{\sum_{[\ell]\in \mathcal{R}_L(F, 2)_{III}, \mathcal{R}_L(F, 2)_{IV}, \mathcal{R}_L(F, 2)_{V}}}\v(L,K_{\ell})\\
&+4^{2m-1}\displaystyle{\sum_{[\ell]\in \mathcal{R}_L(F, 2)_{II}}}\v(L,K_{\ell})\\
&\leq\displaystyle{\sum_{[\ell]\in R(F ,2)_I}}\frac{\v(\SU(K_{\ell}))}{\v(\SU(L))}+2^{2m-1}\displaystyle{\sum_{[\ell]\in \mathcal{R}_L(F, 2)_{III}, \mathcal{R}_L(F, 2)_{IV}, \mathcal{R}_L(F, 2)_{V}}}\frac{\v(\SU(K_{\ell}))}{\v(\SU(L))}\\
&+4^{2m-1}\displaystyle{\sum_{[\ell]\in \mathcal{R}_L(F, 2)_{II}}}\frac{\v(\SU(K_{\ell}))}{\v(\SU(L))}\quad (\because \mathrm{Proposition\ }\ref{prop:u_to_su_ow})\\
&\leq (1+ 2^{2m-1}\cdot 2^{2m} + 4^{2m-1}\cdot 4^{2m})\cdot \frac{2^{2m+5/2}\cdot 3\cdot (2\pi)^{2m}}{\theta\cdot (2m-1)!\cdot\zeta(2m)}\quad (\because (\ref{final:even}))\\
&= (1+ 2^{4m-1} + 2^{8m-2})\cdot \frac{2^{2m+5/2}\cdot 3\cdot (2\pi)^{2m}}{\theta\cdot (2m-1)!\cdot\zeta(2m)}.
\end{align*}

Moreover, if $P(\a:L)$ holds, we have
\begin{align}
\label{ams:e_ow}
&V(L,F)\leq (1+ 2^{4m-1} + 2^{8m-2})\cdot \frac{2^{2m+5/2}\cdot 3\cdot (2\pi)^{2m}}{\theta\cdot (2m-1)!\cdot\zeta(2m)\cdot D(L)^{1/\a}}
\end{align}
by (\ref{final:strong_e}).

\subsubsection{$F=\Q(\sqrt{-1})$ \rm{\textbf{case}}}
Let $F=\Q(\sqrt{-1})$.
From (\ref{final:even}), we have 
\begin{align*}
    &V(L,\Q(\sqrt{-1}))\\
    &\defeq
    \displaystyle {3\sum_{[\ell]\in \mathcal{R}_L(\Q(\sqrt{-1}) ,4)_I}}\v(L,K_{\ell})+3\cdot 2^{2m-1} \displaystyle{\sum_{[\ell]\in \mathcal{R}_L(\Q(\sqrt{-1}), 4)_{II}}}\v(L,K_{\ell})\\
    &+4^{2m-1} \displaystyle{\sum_{[\ell]\in \mathcal{R}_L(\Q(\sqrt{-1}), 2)_{II}}}\v(L,K_{\ell})\\
    &\leq\displaystyle {3\sum_{[\ell]\in \mathcal{R}_L(\Q(\sqrt{-1}) ,4)_I}}\frac{\v(\SU(K_{\ell}))}{\v(\SU(L))}+3\cdot 2^{2m-1} \displaystyle{\sum_{[\ell]\in \mathcal{R}_L(\Q(\sqrt{-1}), 4)_{II}}}\frac{\v(\SU(K_{\ell}))}{\v(\SU(L))}\\
&+4^{2m-1} \displaystyle{\sum_{[\ell]\in \mathcal{R}_L(\Q(\sqrt{-1}), 2)_{II}}}\frac{\v(\SU(K_{\ell}))}{\v(\SU(L))}\quad (\because \mathrm{Proposition\ }\ref{prop:u_to_su_-1})\\
&\leq (3+3\cdot 2^{2m-1}\cdot 2^{2m} + 4^{2m-1}\cdot 4^{2m})\cdot \frac{2^{2m+5/2}\cdot 3\cdot (2\pi)^{2m}}{\theta\cdot (2m-1)!\cdot\zeta(2m)}\quad (\because (\ref{final:even}))\\
&= (3+3\cdot 2^{4m-1} + 2^{8m-2})\cdot \frac{2^{2m+5/2}\cdot 3\cdot (2\pi)^{2m}}{\theta\cdot (2m-1)!\cdot\zeta(2m)}.
\end{align*}

Moreover, if $P(\a:L)$ holds, we have
\begin{align}
\label{ams:e_-1}
&V(L,F)\leq (3+3\cdot 2^{4m-1} + 2^{8m-2})\cdot \frac{2^{2m+5/2}\cdot 3\cdot (2\pi)^{2m}}{\theta\cdot (2m-1)!\cdot\zeta(2m)\cdot D(L)^{1/\a}}
\end{align}
by (\ref{final:strong_e}).

\subsubsection{$F=\Q(\sqrt{-3})$ \rm{\textbf{case}}}
Let $F=\Q(\sqrt{-3})$.
From (\ref{final:even}), we have  
\begin{align*}
  &V(L,\Q(\sqrt{-3}))\\
  &\defeq
  \displaystyle{5\sum_{[\ell]\in \mathcal{R}_L(\Q(\sqrt{-3}) ,6)}}\v(L,K_{\ell})+2\cdot 3^{2m-1} \displaystyle{\sum_{[\ell]\in \mathcal{R}_L(\Q(\sqrt{-3}), 3)}}\v(L,K_{\ell})
\\
&+4^{2m-1} \displaystyle{\sum_{[\ell]\in \mathcal{R}_L(\Q(\sqrt{-3}), 2)}}\v(L,K_{\ell})\quad (\because \mathrm{Proposition\ }\ref{prop:u_to_su_-3})\\
&\leq (5+2\cdot 3^{2m-1}\cdot 3^{2m} + 4^{2m-1}\cdot 4^{2m})\cdot \frac{2^{2m+5/2}\cdot 3\cdot (2\pi)^{2m}}{\theta\cdot (2m-1)!\cdot\zeta(2m)}\quad (\because (\ref{final:even}))\\
&= (5+2\cdot 3^{4m-1} + 2^{8m-2})\cdot \frac{2^{2m+5/2}\cdot 3\cdot (2\pi)^{2m}}{\theta\cdot (2m-1)!\cdot\zeta(2m)}.
\end{align*}

Moreover, if $P(\a:L)$ holds, we have
\begin{align}
\label{ams:e_-3}
&V(L,F)\leq (5+2\cdot 3^{4m-1} + 2^{8m-2})\cdot \frac{2^{2m+5/2}\cdot 3\cdot (2\pi)^{2m}}{\theta\cdot (2m-1)!\cdot\zeta(2m)\cdot D(L)^{1/\a}}
\end{align}
by (\ref{final:strong_e}).

\subsubsection{\rm{\textbf{Summary:\ even-dimensional\ case}}}
Combining all statements proven above, we obtain the following.
\begin{thm}
\label{thm:volume_e}
Let $L$ be primitive of signature $(1,2m-1)$  with $m>1$.
Assume $(\heartsuit)$.
Then, if $m$ or $\theta$ is sufficiently large, the line bundle $\M(a)$ is big.
More precisely, 
\[V(L,F)\leq \frac{f_F^{even}(m)}{\theta}.\]
Moreover, if $P(\a:L)$ holds for some $\a>0$, we have
\[V(L,F)\leq \frac{f_F^{even}(m)}{D(L)^{1/\a}\cdot \theta}.\]
\end{thm}
\section{Conclusion}
\label{section:conclusion}
\subsection{Main results}
We shall restate our main results in this paper.
This gives a solution to the problem (A) in Section \ref{introduction}.
\begin{thm}
\label{thm:volume_conclusion_oe}
Let $L$ be a primitive Hermitian lattice over $\OO_F$ of signature $(1,2m)$ $\mathrm{(}$resp. $(1,2m-1)\mathrm{)}$ with $m>1$.
Assume $(\heartsuit)$.
Then, for a positive integer $a$, if $m$ or $\theta$ is sufficiently large, the line bundle $\M(a)$ is big.
More precisely, 
\[V(L,F)\leq \frac{f_F^{odd}(m)}{\theta}\quad (\mathrm{resp.\ }V(L,F)\leq \frac{f_F^{even}(m)}{\theta}).\]
Moreover, if $P(\a:L)$ holds for some $\a>0$, we have
\[V(L,F)\leq \frac{f_F^{odd}(m)}{D(L)^{1/\a}\cdot \theta}\quad (\mathrm{resp.\ }V(L,F)\leq \frac{f_F^{even}(m)}{D(L)^{1/\a}\cdot \theta}).\]
\end{thm}
\begin{proof}
Combine Proposition \ref{thm:bigness_criterion} with Theorem \ref{thm:volume_o} and Theorem \ref{thm:volume_e}.
\end{proof}

For unramified square-free lattices, we obtain a sharper estimate because one can see that $\lambda_v^{K_{\ell}}/\lambda_v^{L}\leq 1$ for $v|D$ and such lattices satisfy $P(1)$.
Hence, we have the following:
\begin{align}
\label{inequality:unimod_sq_free_strong}
    V(L,F)\leq (1+2^{4m+1}+2^{8m+2})\cdot 
\begin{cases}
\displaystyle{\frac{2\cdot (2\pi)^{2m+1}}{D^{2m+1/2}\cdot (2m)!\cdot L(2m+1)\cdot D(L)^{1/\a}}}&(n=2m),\vspace{2pt}\\
\displaystyle{\frac{(2\pi)^{2m}}{(2m-1)!\cdot\zeta(2m)\cdot D(L)^{1/\a}}}&(n=2m-1).\\
\end{cases}
\end{align}
\begin{cor}[Unramified square-free case]
\label{cor:unimodular}
Up to scaling, assume that $L$ is unramified square-free over $\OO_{F_0}$.
Then, for a positive integer $a$, if $n$ is sufficiently large, or $D_0$ is sufficiently large and $n$ is even, then the line bundle $\M(a)$ is big.
\end{cor}
\begin{proof}
Using the better estimate above, to prove that $\M(a)$ is big, it suffices to show that $L$ and $K_{\ell}$ satisfy $(\star)$ for any $[\ell]\in \mathcal{R}_L(F_0,2)$ under the assumption on $L$ and $F_0$.
This was shown in Proposition \ref{ex:unimodular} and \ref{ex:square-free}.
\end{proof}
\subsection{Application I: Unitary modular varieties of general type}
\label{subsection:appI}
\begin{thm}
\label{thm:gen_type}
Let $L$ be primitive, $n\geq 13$ and $F\neq\Q(\sqrt{-1}), \Q(\sqrt{-2}), \Q(\sqrt{-3})$.
Assume that $(\heartsuit)$ holds and there exists a non-zero cusp form of weight lower than $n+1$ with respect to $\U(L)$.
Then,  $X_L$ is of general type if $\dim X_L=n$ or $\theta$ is sufficiently large.
\end{thm}
\begin{proof}
The canonical divisor $K_{\overline{X_L}}$ is big by combining Proposition \ref{thm:bigness_criterion}, \ref{thm:volume_o} and \ref{thm:volume_e} with the existence of a cusp form.
Here, we use the result of Behrens \cite[Theorem 4]{Behrens} which asserts there are no branch divisors at boundary $\overline{X_L}\setminus X_L$ and the author \cite{Maeda2} which asserts that there are no irregular cusps for $\U(L)$.
Then, every pluricanonical form on $\overline{X_L}$ extends to its desingularization since it has at worst canonical singularities \cite[Theorem 4]{Behrens}.
This means that $X_L$ is of general type.
\end{proof}

\subsection{Application\ II:\ Finiteness of Hermitian lattices admitting  reflective modular forms}
\label{subsection:ams}
One might expect that there exist only finitely many Hermitian lattices of signature $(1,n)$ admitting reflective modular forms.
We can prove this consideration for unramified square-free lattices from (\ref{ams:o_ow}), (\ref{ams:o_-1}), (\ref{ams:o_-3}), (\ref{ams:e_ow}), (\ref{ams:e_-1}) and (\ref{ams:e_-3}).

\begin{cor}[Finitness of Hermitian lattices admitting reflective modular forms]
\label{cor:ams}
Up to scaling, the set of reflective lattices with slope less than $r$, which make $P(\a:L)$ and $(\heartsuit)$ hold, is finite for fixed $\a,r>0$.
In particular, the set  
\[\{\mathrm{Unramified\ square}\mathchar`-\mathrm{free\ reflective\ lattices}\ \mathrm{with\ slope\ less\ than\ }r \mid n>2\}/\sim\]
is finite for a fixed $F_0$.
\end{cor}
\begin{proof}
We will only consider the odd-dimensional case of $F\neq\Q(\sqrt{-3})$ because the other cases can be proved in the same way.
Let $L$ be a Hermitian reflective lattice of signature $(1,n)$ with $n>2$, which make $P(\a:L)$ holds.
We may assume that $L$ is primitive.
From (\ref{final:strong_o}) and the fact that there are only finitely many Hermitian lattices with bounded discriminant, it follows that the set of Hermitian lattices making $P(\a:L)$ holds is finite, up to scaling; see also \cite[Proof of Theorem 1.5]{Ma}.
If $L$ is unramified square-free, then the primitivity implies that $L$ satisfies  $P(1)$.
Therefore, we also obtain finiteness of unramified square-free reflective lattices.
\end{proof}

\subsection{Explicit estimation: General case}
\label{subsection:general_case}
We estimate $V(L,F)$ and $W(L,F,1)$ explicitly.
For the concrete form of the function $W(L,F,1)$, see  Proposition \ref{mainthm:criterion} and the explanation there.
We investigate how large values of $m$ we need to take in Theorem \ref{thm:gen_type}.
First, we consider odd-dimensional cases so that assume that $L$ has signature $(1,2m)$ with $m>1$.
Then, from Theorem \ref{thm:volume_o}, $W(L,F,1)<0$ if  
\[m>\begin{cases}
277&(F\neq\Q(\sqrt{-1}), \Q(\sqrt{-3})),\\
550&(F=\Q(\sqrt{-1})),\\
823&(F=\Q(\sqrt{-3})).\\
\end{cases}\]
Second, when $L$ has signature $(1,2m-1)$ with $m>1$, from Theorem \ref{thm:volume_o}, $W(L,F,1)<0$ if 
\[m>\begin{cases}
390&(F\neq\Q(\sqrt{-1}), \Q(\sqrt{-3})),\\
776&(F=\Q(\sqrt{-1})),\\
1163&(F=\Q(\sqrt{-3})).\\
\end{cases}\]
\subsection{Explicit estimation: Unramified square-free case}
\label{subsection:unimod_sqfree_case}
We assume that $L$ is unramified square-free over $\OO_{F_0}$.
From (\ref{inequality:unimod_sq_free_strong}), we have  $W(L,F,1)<0$ if $n>138$ 
where $n=\dim X_L$ as usual.
On the other hand,  if $D_0>30$, then for any even $n\geq 4$, it follows that $W(L,F,1)<0$.

\subsection*{Acknowledgements}
The author wishes to express his thanks to Shouhei Ma for helpful discussions and his warm encouragement.
He also would like to thank Tetsushi Ito, his advisor, for his
constructive suggestions, Kazuma Ohara, Masao Oi, Satoshi Wakatsuki, Takao Watanabe, and Yuki Yamamoto for several comments on Bruhat-Tits theory, Takuya Yamauchi for some remarks on cusp forms, and Ryo Matsuda for his help on some of the calculation.
Finally, he is deeply grateful to the
anonymous referees for helpful suggestions to make the exposition more readable.
This work is supported by JST ACT-X JPMJAX200P.


\begin{thebibliography}{99}

\bibitem{Allcock1}
Allcock, D.,
\textit{New complex- and quaternion-hyperbolic reflection groups},
Duke Math. J. 103 (2000), no. 2, 303-333.

\bibitem{ACT1}
Allcock, D., Carlson, J. A., Toledo, D.,
\textit{The complex hyperbolic geometry of the moduli space of cubic surfaces},
J. Algebraic Geom. 11 (2002), no. 4, 659–724.

\bibitem{ACT2}
Allcock, D., Carlson, J. A., Toledo, D.,
\textit{The moduli space of cubic threefolds as a ball quotient},
Mem. Amer. Math. Soc. 209 (2011), no. 985.

\bibitem{AMRT}
Ash, A., Mumford, D., Rapoport, M., Tai, Y.-S.,
\textit{Smooth compactifications of locally symmetric varieties},
Mem. Amer. Math. Soc. 209 (2011), no. 985.

\bibitem{Behrens}
Behrens, N.,
\textit{Singularities of ball quotients},
Geom. Dedicata 159 (2012), 389-407.

\bibitem{Bor}
Borcherds, R.E.,
\textit{Automorphic forms on $\O_{s+2,2}(\R)$ and infinite products}, Invent. Math. 120 (1995), 161–213.

\bibitem{BKPS}
Borcherds, R.E., Katzarkov, L., Pantev, T., Shepherd-Barron, N.I.,
\textit{Families of K3 surfaces},
J. Algebraic Geom. 7 (1998), 183–193.

\bibitem{BP}
Borel, A., Prasad. G.,
\textit{Finiteness theorems for discrete subgroups of bounded covolume in semi-simple groups},
Inst. Hautes Études Sci. Publ. Math. No. 69 (1989), 119-171.


\bibitem{Bruhat}
Bruhat, F.,
\textit{Sur les réprésentations des groupes classiques $\p$-adiques. I, II},
Amer. J. Math. 83 (1961), 321–338, 343–368.

\bibitem{BT}
Bruhat, F., Tits, J.,
\textit{Groupes réductifs sur un corps local. II. Schémas en groupes. Existence d'une donnée radicielle valuée},
Inst. Hautes Études Sci. Publ. Math. No. 60 (1984), 197-376.

\bibitem{BT2}
Bruhat, F., Tits, J.,
\textit{Schémas en groupes et immeubles des groupes classiques sur un corps local. II. Groupes unitaires},
Bull. Soc. Math. France 115 (1987), no. 2, 141–195.

\bibitem{Cho_case1}
Cho, S.,
\textit{Group schemes and local densities of ramified hermitian lattices in residue characteristic 2: Part I}
Algebra Number Theory 10 (2016), no. 3, 451–532.

\bibitem{Cho_case2}
Cho, S.,
\textit{Group schemes and local densities of ramified hermitian lattices in residue characteristic 2. Part II}
Forum Math. 30 (2018), no. 6, 1487–1520.


\bibitem{DM}
Deligne, P., Mostow, G.D.,
\textit{Monodromy of hypergeometric functions and nonlattice integral monodromy},
Inst. Hautes Études Sci. Publ. Math. No. 63 (1986), 5-89.

\bibitem{Dol}
Dolgachev, Igor V.,
\textit{Reflection groups in algebraic geometry}, 
Bull. Am. Math. Soc., New Ser. 45, No. 1, 1-60 (2008).

\bibitem{DK}
Dolgachev, I. V., Kond\={o}, S.,
\textit{Moduli of K3 surfaces and complex ball quotients},
Arithmetic and geometry around hypergeometric functions, 43–100,
Progr. Math., 260, Birkhäuser, Basel, 2007.


\bibitem{Freitag}
Freitag, E.,
\textit{Siegelsche Modulfunktionen},
Springer-Verlag, Berlin, 1983.



\bibitem{GHY}
Gan, W.T., Hanke, J.P., Yu, J.-K., 
\textit{On an exact mass formula of Shimura. Duke Math}, J. 107 (2001), no. 1, 103-133. 

\bibitem{GY}
Gan, W.T., Yu, J.-K., 
\textit{Group schemes and local densities},
Duke Math. J. 105 (2000), no. 3, 497–524. 

\bibitem{Gritsenko1}
Gritsenko, V., 
\textit{Reflective modular forms in algebraic geometry},
arXiv:1005.3753.

\bibitem{Gritsenko2}
Gritsenko, V., 
\textit{Reflective modular forms and their applications},
Russian Math. Surveys 73 (2018), no. 5, 797-864.

\bibitem{GHS}
Gritsenko, V., Hulek, K., Sankaran, G.K.,
\textit{The Kodaira dimension of the moduli of K3 surfaces},
Invent. Math. 169 (2007), no. 3, 519-567.

\bibitem{GHSHM1}
Gritsenko, V., Hulek, K., Sankaran, G.K.,
\textit{The Hirzebruch-Mumford volume for the orthogonal group and applications},
Doc. Math. 12 (2007), 215-241. 

\bibitem{GHSHM2}
Gritsenko, V., Hulek, K., Sankaran, G.K.,
\textit{Hirzebruch-Mumford proportionality and locally symmetric varieties of orthogonal type},
Doc. Math. 13 (2008), 1-19.

\bibitem{GN}
Gritsenko, V., Nikulin, V.,
\textit{Automorphic forms and Lorentzian Kac-Moody algebras. I},
Internat. J. Math. 9 (1998), no. 2, 153–199.

\bibitem{Hijikata}
Hijikata, H.,
\textit{Maximal compact subgroups of some p-adic classical groups}, 
(1963) New Haven, Conn. : Yale University, Dept. of Mathematics.

\bibitem{Jacobowitz}
Jacobowitz, R.,
\textit{Hermitian forms over local fields}, 
Amer. J. Math. 84 (1962), 441-465.

\bibitem{Kudla}
Kudla, S., 
\textit{On certain arithmetic automorphic forms for $\SU(1,q)$}, 
Invent. Math. 52 (1979), no. 1, 1-25.

\bibitem{Lemaire}
Lemaire, B.,
\textit{Comparison of lattice filtrations and Moy-Prasad filtrations for classical groups},
J. Lie Theory 19 (2009), no. 1, 29–54.

\bibitem{Ma}
Ma, S.,
\textit{On the Kodaira dimension of orthogonal modular varieties},
    Invent. Math. 212 (2018), no. 3, 859-911.

 
 \bibitem{Maeda1}
Maeda, Y.,
\textit{Uniruledness of some low-dimensional ball quotients},
Geom. Dedicata 218, No. 1, Paper No. 3, 17 p. (2024).
    
\bibitem{Maeda2}
Maeda, Y.,
\textit{Irregular cusps of unitary Shimura varieties},
Math. Nachr. 296, No. 4, 1560-1588 (2023).
    
\bibitem{MO}
Maeda, Y., Odaka, Y.,
\textit{Fano Shimura varieties with mostly branched cusps},
Springer Proceedings in Mathematics \& Statistics(PROMS, volume 409), 2023.



\bibitem{Mumford}
Mumford, D.,
\textit{Hirzebruch's proportionality theorem in the noncompact case},
Invent. Math. 42 (1977), 239-272.

\bibitem{Mumford2}
Mumford, D.,
\textit{On the Kodaira dimension of the Siegel modular variety},
Algebraic geometry—open problems, 348-375, Lecture Notes in Math., 997, Springer, Berlin, 1983.

\bibitem{Nik}
Nikulin, V. V.,
\textit{Factor groups of groups of automorphisms of hyperbolic forms with respect to subgroups generated by 2-reflections. Algebro-geometric applications}, 
J. Sov. Math. 22, 1401-1475 (1983).

\bibitem{Prasad}
Prasad, G.,
\textit{Volumes of S-arithmetic quotients of semi-simple groups with an appendix by Moshe Jarden and the author},
Inst. Hautes Études Sci. Publ. Math. No. 69 (1989), 91-117.

\bibitem{PY}
Prasad, G., Yeung, S.-K.,
\textit{Nonexistence of arithmetic fake compact Hermitian symmetric spaces of type other than $A_n$  $(n\leq 4)$},
J. Math. Soc. Japan 64 (2012), no. 3, 683–731.

\bibitem{RC}
Ryzhkov, A. A., Chernousov, V. I.,
\textit{On the classification of maximal arithmetic subgroups of simply connected groups},
Sb. Math. 188 (1997), no. 9, 1385-1413.

\bibitem{Stevens}
Stevens, S.,
\textit{The supercuspidal representations of p-adic classical groups},
Invent. Math. 172 (2008), no. 2, 289–352.

\bibitem{Tai}
Tai, Y.-S.,
\textit{On the Kodaira dimension of the moduli space of abelian varieties},
 Invent. Math. 68 (1982), no. 3, 425-439.
 
 \bibitem{Thilmany}
Thilmany, F.,
\textit{Lattices of Minimal Covolume in Real Special Linear Groups},
Thesis (Ph.D.)–University of California, San Diego. 2019. 93 pp.

\bibitem{Tits}
Tits, J.,
\textit{Reductive groups over local fields}, 
Automorphic forms, representations and L-functions, Part 1, pp. 29–69,
Proc. Sympos. Pure Math., XXXIII, Amer. Math. Soc., Providence, R.I., 1979.
 

\bibitem{Wall}
Wall, C. T. C.,
\textit{On the classification of hermitian forms. I. Rings of algebraic integers}, Compositio Math. 22 (1970), 425-451.

\bibitem{WW}
Wang, H., Williams, B.,
\textit{Free algebras of modular forms on ball quotients},
arXiv:2105.14892.
  

\end{thebibliography}
\end{document}